\documentclass[reqno,english,11pt]{amsart}
\usepackage{times}

\usepackage{amsmath,amsfonts,amssymb,graphicx,amsthm,enumerate,url,mathabx}
\usepackage[noadjust]{cite}
\usepackage{stmaryrd}
\usepackage{mathrsfs,booktabs}
\usepackage{xifthen,xcolor,tikz,setspace}
\usetikzlibrary{decorations.pathmorphing,patterns,shapes,calc,decorations}
\usetikzlibrary{decorations.pathreplacing}
\usepackage{mathtools,upgreek}
\usepackage[final]{showkeys} 

\definecolor{refkey}{gray}{.75}
\definecolor{labelkey}{gray}{.5}

\usepackage[colorlinks=true,citecolor=blue]{hyperref}
\usepackage{verbatim}

\usepackage{float}

\usepackage[letterpaper]{geometry}
\makeatletter
\numberwithin{equation}{section}
\numberwithin{figure}{section}

\newtheorem{theorem}{Theorem}[section]
\newtheorem*{theorem*}{Theorem}
\newtheorem{lemma}[theorem]{Lemma}

\newtheorem{proposition}[theorem]{Proposition}

\newtheorem{fact}[theorem]{Fact}

\newtheorem{cor}[theorem]{Corollary}

\theoremstyle{definition}{

\newtheorem{definition}[theorem]{Definition}

\newtheorem*{definition*}{Definition}

\newtheorem{remark}[theorem]{Remark}

}

\newcommand{\R}{\mathbb R}

\newcommand{\cD}{\ensuremath{\mathcal D}}
\newcommand{\cE}{\ensuremath{\mathcal E}}
\newcommand{\cF}{\ensuremath{\mathcal F}}
\newcommand{\cG}{\ensuremath{\mathcal G}}

\newcommand{\cP}{\ensuremath{\mathcal P}}

\newcommand{\cR}{\ensuremath{\mathcal R}}
\newcommand{\cS}{\ensuremath{\mathcal S}}
\newcommand{\cT}{\ensuremath{\mathcal T}}

\newcommand{\ps}{{\hat{p}}}

\newcommand{\one}{{\mathbf{1}}}
\newcommand{\zero}{{\mathbf{0}}}

\renewcommand{\Pr}{{\mathbb{P}}}

\renewcommand{\epsilon}{{\varepsilon}}

\DeclareMathOperator{\diam}{{\mathrm{diam}}}

\newcommand{\tv}{{\textsc{tv}}}
\newcommand{\tmix}{{t_{\textsc{mix}}}}

\usepackage[normalem]{ulem}

\newcommand{\potts}{\textsc{p}}
\newcommand{\rc}{\textsc{rc}}

\makeatother

\title[Sampling from the Potts model at low temperatures via random-cluster dynamics]
{On the tractability of sampling from the
Potts model at low temperatures via random-cluster dynamics}

\author{Antonio Blanca}
\address{A.\ Blanca\hfill\break
Department of CSE, Pennsylvania State University }
\email{ablanca@cse.psu.edu}
\author{Reza Gheissari}
\address{R.\ Gheissari\hfill\break
Department of Mathematics \\ Northwestern University }
\email{gheissari@northwestern.edu}

\begin{document}

\maketitle

\thispagestyle{empty}

\vspace{-1cm}
\begin{abstract}
Sampling from the $q$-state ferromagnetic Potts model is a fundamental question in statistical physics, probability theory, and theoretical computer science. On general graphs, this problem may be computationally hard, and this hardness holds at arbitrarily low temperatures.
    At the same time, in recent years, there has been significant progress showing the existence of low-temperature sampling algorithms in various specific families of graphs. Our aim in this paper is to understand the minimal structural properties of general graphs that enable polynomial-time sampling from the $q$-state ferromagnetic Potts model at low temperatures. 
    We study this problem from the perspective of random-cluster dynamics. These are non-local Markov chains that have long been believed to converge rapidly to equilibrium at low temperatures in many graphs. However, the hardness of the sampling problem likely indicates that this is not even the case for all bounded degree graphs.

    Our results demonstrate that a key graph property behind fast or slow convergence time for these dynamics is whether the independent edge-percolation on the graph admits a strongly supercritical phase. By this, we mean that at large $p<1$, it has a large linear-sized component, and the graph complement of that component is comprised of only small components. Specifically, we prove that such a condition implies fast mixing of the random-cluster Glauber and Swendsen--Wang dynamics on two general families of bounded-degree graphs: (a)~graphs of at most stretched-exponential volume growth and (b) locally treelike graphs. In the other direction, we show that, even among graphs in those families, these Markov chains can converge exponentially slowly at arbitrarily low temperatures if the edge-percolation condition does not hold. In the process, we develop new tools for the analysis of non-local Markov chains, including a framework to bound the speed of disagreement propagation in the presence of long-range correlations, an understanding of spatial mixing properties on trees with random boundary conditions, and an analysis of burn-in phases at low temperatures.\footnote{An extended abstract of this paper appeared in FOCS 2023~\cite{BG23-FOCS-abstract} under the title ``Sampling from the Potts model at low temperatures via Swendsen--Wang dynamics".}
\end{abstract}

%%%%%% POTENTIAL JOURNAL ABSTRACT %%%%%%%

%\vspace{-.5cm}
\section{Introduction}
The $q$-state ferromagnetic Potts model is a classical spin system model central to probability theory and with applications in statistical physics, theoretical computer science, and other fields. 
It is defined on a graph $G= (V(G),E(G))$ as a probability distribution over configurations in $\Omega_\potts= \{1,...,q\}^{V(G)}$,
with a parameter $\beta>0$, corresponding to the {inverse temperature} in physical applications, controlling the strength of the interaction between the edges of $G$. 
Formally, the probability of each configuration $\sigma \in \Omega_\potts$ is given~by:
\begin{align}\label{eq:Potts-measure}
    \mu_{G,\beta,q}(\sigma)  = \frac{1}{Z_{G,\beta,q}^\potts} \exp\Big( \beta \sum_{\{u,v\}\in E(G)}\mathbf 1\{\sigma(u) = \sigma(v)\}\Big)\,.
\end{align}

\noindent
The factor $Z_{G,\beta,q}^\potts$ is a normalization constant and is known as the partition function; $\sigma(u)$ denotes the 
color or spin value of the configuration $\sigma$ at vertex $u$. 
The classical Ising model corresponds to the $q=2$ case.

The question of sampling from the ferromagnetic Potts model is an important one and has been extensively studied on a variety of graphs and temperature regimes (i.e., different values of $\beta$).
In general, it is known that the problem of approximate sampling from~\eqref{eq:Potts-measure}
for $q \ge 3$ and $\beta$ large is \#BIS-hard, in the sense that there exist graphs, including bounded degree ones, for which the approximate sampling problem is as hard as approximately counting the number of independent sets on bipartite graphs~\cite{GoldbergJerrum,GSVY}. This latter task is a well-studied computational problem that is believed not to have a polynomial time approximation algorithm. 
This sharply contrasts with the ferromagnetic Ising case ($q=2$), where polynomial-time samplers have been known since the 1990s~\cite {JSIsing,RanWil}. At the same time, for some families of graphs, notably including $\mathbb Z^d$ and expander graphs, the existence of polynomial-time sampling algorithms has recently been shown for the Potts model at low temperatures: see, e.g.,~\cite{HPR-PTRF,BCHPT-Potts-all-temp,carlson2020efficient,JKP,CGGPS,HeJePe20,Carlson,GhSi22,Huijben-Patel-Regts}. This raises the question of what are the underlying graph structures and temperatures that cause tractability or hardness of approximately sampling from~\eqref{eq:Potts-measure}.
We study this from the perspective of widely-used Markov chain-based algorithms.

One fundamental approach to sampling from Gibbs distributions of the form of~\eqref{eq:Potts-measure} is via Markov chains whose stationary distribution is exactly $\mu_{G,\beta,q}$. The simplest such Markov chain is the Glauber dynamics for the Potts model (also known as the Gibbs sampler), which updates the state of a randomly chosen vertex at each step.
Its simplicity makes it quite appealing to practitioners, 
but it is known to take exponentially (in $\text{poly}(|V|)$) many steps to equilibrate at low temperatures ($\beta$ large).
In lieu of this, in order to sample from~\eqref{eq:Potts-measure}, an oft-used approach is a different family of Markov chains based on the \emph{Edwards--Sokal} coupling of the ferromagnetic Potts model to a graphical model called the random-cluster model \cite{ES}: see~\eqref{eq:fk-definition} for its definition. This family of Markov chains includes the extensively studied {Swendsen--Wang (SW) dynamics}, and its close relative, the Glauber dynamics for the random-cluster model. 

These Markov chains make non-local updates, and 
are often used to bypass the bottlenecks that slow down the convergence of the Potts Glauber dynamics at low temperatures when $q \ge 3$. 
At the same time, the aforementioned \#BIS-hardness of the sampling problem at low temperatures suggests that these Markov chains could not have a polynomial speed of convergence on all graphs. 
For the sake of completeness, we mention that at high temperatures (small $\beta$), these Markov chains converge quickly but have a larger computational overhead per step than the Potts Glauber dynamics~\cite{BCCPSV};
there are also ``intermediate'' temperature regimes corresponding to first-order phase transitions where these Markov chains are known to converge exponentially slowly to equilibrium~\cite{GoJe,BCT,GL1,GLP,PottsRGMetastabilityCMP}. 

In this paper, we systematically analyze these Edwards--Sokal based Markov chains on general graphs at low temperatures. In the process, we develop new tools for the analysis of non-local chains and arrive at an explanation, in terms of geometric properties of the graph, that dictate whether these Markov chains converge quickly or slowly at 
low temperatures (i.e., large, but independent of the system size, values of~$\beta$).

Let us define the Markov chains of interest. For a unified discussion, it is convenient to reparametrize $\beta$ by $p=1-e^{-\beta}$. Notice that low-temperature settings corresponding to $\beta$ large correspond to $p$ close to~$1$. 
The SW dynamics transitions from a configuration $\sigma_t\in \Omega_\potts$ to $\sigma_{t+1}\in \Omega_\potts$ as follows:
\begin{enumerate}
    \item Independently, for every $e = \{u,v\}\in E(G)$ if $\sigma_t(u) = \sigma_t(v)$ include $e$ in $E_t$ with probability~$p$;
    \item Independently, for every connected component $\mathcal C$ in $(V(G),E_t)$, draw a color $c \in \{1,...,q\}$ uniformly at random, and set $\sigma_{t+1}(v)= c$ for all $v\in \mathcal C$. 
\end{enumerate}
It can be checked that the SW dynamics is reversible with respect to $\mu_{G,\beta,q}$ and thus converges to it. 
In effect, the SW dynamics moves on the larger probability space of Potts model configurations 
together with random-cluster configurations. 
The configurations of this model consist of edge subsets, i.e., $\Omega_\rc = \{0,1\}^{E(G)}$, and the SW dynamics can be interpreted as alternating steps of sampling a random-cluster configuration $E_t$ conditionally on the Potts configuration $\sigma_t$, then sampling the Potts configuration  $\sigma_{t+1}$ conditionally on $E_t$. 

A closely related Markov chain is the Glauber dynamics that moves in the space of random-cluster configurations; for brevity, we call this Markov chain the \emph{FK dynamics} since the random-cluster model is also known as the FK model. Here, given an edge subset $E_t\in \Omega_\rc$, we generate $E_{t+1}$ by: 
\begin{enumerate}
    \item Pick an edge $e\in E(G)$ uniformly at random;
    \item Set $E_{t+1} = E_{t} \cup \{e\}$ with probability:
    \begin{align}\label{eq:FK dynamics-update-rule}
        \begin{cases}
            p & \qquad \text{if $e$ is not a cut-edge in $E_t$}; \\ 
            \ps  := \frac{p}{p+(1-p)q} & \qquad \text{if $e$ is a cut-edge in $E_t$};
        \end{cases}
    \end{align}
    and $E_{t+1} = E_t \setminus \{e\}$ otherwise. 
\end{enumerate}
A cut-edge is an edge whose state affects the number of connected components of the configuration. It can be checked that the FK dynamics
converges to the random-cluster distribution~\eqref{eq:fk-definition}. 
After convergence, one may produce a sample from the corresponding $q$-state Potts measure (the one with $\beta$ so that $p = 1-e^{ - \beta}$) with little overhead by independently assigning states uniformly amongst $\{1,...,q\}$ to each connected component of the random-cluster configuration, as in step (2) of the SW dynamics above.  
As such, the FK dynamics provides an alternative Markov chain that can be used to sample from~\eqref{eq:Potts-measure}. 

To formalize convergence rates of these Markov chains, recall that the mixing time of a Markov chain is the number of steps required to reach a distribution close to the stationary distribution (in total variation distance), assuming the worst possible starting state: see~\eqref{def:tmix} for the formal definition.
It is known that the mixing times of the SW and FK dynamics 
can differ only up to a $O(|E(G)|)$ factor (see \cite{Ullrich-random-cluster}).

Both the SW and FK dynamics  
are conjectured to overcome some of the key difficulties associated with sampling from the Potts distribution quickly at low temperatures. They are, therefore, quite popular, but their non-locality makes the rigorous analysis of their mixing times 
significantly more challenging than their Potts Glauber dynamics counterparts. 
In recent years, significant progress has been made in establishing optimal mixing time bounds for the SW and FK dynamics in high-temperature regimes where the corresponding Potts Glauber dynamics is also known to be fast mixing; see, e.g., \cite{BCPSV,BZSV-SW-trees,BCCPSV,GhSi22}. 
These works have resulted in optimal (or nearly optimal) 
mixing time bounds for SW and FK dynamics 
that hold under various correlation decay conditions (e.g., strong spatial mixing, tree uniqueness, Dobrushin uniqueness, spectral independence, etc.). In particular, $p\lesssim 1/\Delta$ implies fast mixing for all graphs of maximum degree $\Delta$.
By contrast, in the low-temperature setting, where correlations do not decay, the Potts Glauber dynamics converges slowly, and alternative efficient sampling algorithms are most needed, there is no generic criterion guaranteeing that the SW and FK dynamics mix quickly.

In fact, rigorous bounds for the mixing time 
of the SW and FK dynamics at low temperatures are rare and  can be summarized as follows. 
In the Ising case of $q=2$,~\cite{GuoJer} showed that these Markov chains mix in $O(n^{10})$ time on all $n$-vertex graphs and all $p$. On the complete graph,~\cite{BS-MF,GSV,blanca2022critical} establish nearly-optimal mixing time bounds throughout the low-temperature regime. On more sophisticated geometries, progress has been limited to the special case of the integer lattice $\mathbb Z^d$. In particular, in~\cite{Ullrich1,BS}, fast mixing was shown in the low-temperature regime on subsets of $\mathbb Z^2$ via planar duality to high-temperatures (see also~\cite{Martinelli-SW} for sharper bounds for the SW dynamics at low temperatures in the Ising case). Recently~\cite{GhSi22} showed fast mixing at low temperatures in cubes in $\mathbb Z^d$.  For general graphs, the only low-temperature criterion known to ensure fast mixing is $p>1-O(1/|E(G)|)$~\cite{Huber}.

This leaves a wealth of questions to explore on general families of graphs, notably including 
the mixing times of the SW and FK dynamics at values of $p$ close to $1$, but importantly, independent of the graph size. We consider this question for two broad families of bounded-degree graphs: graphs of at most stretched-exponential volume growth, and locally treelike graphs (which allow for exponential volume growth). We show that for all such graphs, fast mixing of the SW and FK dynamics at low enough temperatures is ensured if the independent edge-percolation process on the graph, 
where an edge-set $\widetilde \omega \subset E(G)$ is obtained by keeping each edge with probability ${\tilde p}$, independently,
has a \emph{strongly supercritical phase} (i.e., for $\tilde p$ close to $1$, all large connected sets in $G$ intersect the giant component of $\tilde\omega$; see Definition~\ref{def:supercritical-Bernoulli-perc} and~\ref{def:supercritical-Bernoulli-exp} for precise definitions). To illustrate the necessity of this condition, for any arbitrarily large $p<1$, we construct explicit graphs---both ones of polynomial volume growth, and ones that are locally treelike---on which the edge-percolation is not in a strongly supercritical phase, and, in turn, the SW and FK dynamics mix slowly. 

The class of graphs that have strongly supercritical phases for their edge-percolation is an area of extensive study, and it is closely connected to whether the graph has isoperimetric dimension strictly larger than~$1$. The key takeaway from our results is thus a purely geometric mechanism underlying fast or slow mixing of the SW and FK dynamics at large $p<1$ on two large families of bounded degree graphs.

\subsection{Graphs of at most stretched-exponential growth}\label{subsec:subexp-growth}

The first general class of graphs
for which we establish fast mixing of the SW and FK dynamics at low temperatures under the percolation condition are bounded-degree graphs that have at most stretched-exponential volume growth. Let us introduce some notation: in what follows, we think of $G=(V(G),E(G))$ as a connected graph on $n$ vertices, with maximum degree $\Delta\ge 3$, and we fix any $q\ge 2$. For a vertex $v\in V(G)$, let $B_R(v) = \{w: d_G(v,w) \le R\}$ be the set of vertices at graph distance at most $R$ from $v$. For a subset $A\subset V(G)$, let $\partial_e A\subset E(G)$ denote the edge boundary of $A$, i.e., the set of edges in $E(G)$ with exactly one endpoint in $A$.   

\begin{definition}\label{def:eta-stretched-exponential}
    The graph $G$ has \emph{$\eta$-stretched-exponential volume growth} if $|B_R(v)|\le e^{ R^{\eta}}$ for all $v\in V(G)$ and all $R$ sufficiently large (i.e., $R\ge R_0$ for some $R_0$ independent of $n$; for convenience, take $R_0 = 1/\eta$\footnote{At the level of quantification of our bounds, this choice does not affect our main statements; allowing for general $R_0$ would simply add $R_0$ to the set of constants on which the bounds depend.}). 
\end{definition}

Natural graph families with at most stretched-exponential volume growth include bounded-degree lattices in 
$\R^d$; e.g., finite subsets of $\mathbb Z^d$, the triangular and hexagonal lattices, etc., and Cayley graphs of polynomial growth groups. This notion is closely related to a quantitative version of amenability.

We show that the SW and FK dynamics on these graphs are rapidly mixing when the independent edge-percolation process on the underlying graph $G$ has a ``strong supercritical phase'' which we define next.
For $\tilde p \in (0,1)$, let $\widetilde \pi_G = \bigotimes_{e\in E(G)} \text{Bernoulli}(\tilde p)$ denote the independent edge-percolation distribution for $G$.
We note that if $\widetilde \omega$ is drawn from $\widetilde \pi_G$ (i.e., $\widetilde \omega \sim \widetilde \pi_G$), then we can think of $\widetilde \omega$ as, both, a vector in $\{0,1\}^{E(G)}$ or as a subset of edges of $E(G)$. For $B \subset E(G)$, let $\widetilde \omega(B) \in \{0,1\}^B$ denote the state of the edges from $B$ in $\widetilde \omega(B)$. 

\begin{definition}\label{def:supercritical-Bernoulli-perc}
    Let $\widetilde \omega \sim \widetilde \pi_G$.
    We say that $G$ has a \emph{strong supercritical phase} (with parameters $\delta,\tilde p$)
    if there exists $\tilde p < 1$ and $\delta>0$ 
    such that for every $v \in V(G)$, 
    the probability that there exists a connected set $A\ni v$ having $\ell \le |A|\le n/2$ 
    with $\widetilde \omega(\partial_e A)\equiv 0$ is at most $\exp( - \ell^{\delta/(1+\delta)})$, for all $\ell$  sufficiently large (again, meaning $\ell \ge \ell_0$ for some $\ell_0$ independent of $n$, for instance for convenience $\ell_0 = 1/\delta$). 
\end{definition}

Roughly the definition says that the probability that there exists a set $A\subset V(G)$ that is connected in $G$, contains $v$, and has 
size at least $\ell$, but does not intersect the largest component in $\widetilde \omega$, is stretched-exponentially small in $\ell$ (with the exponent governed by the parameter $\delta>0$, which as we will comment on shortly is related to the isoperimetric dimension of the underlying graph). Notice that the existence of a $\tilde p$ in Definition~\ref{def:supercritical-Bernoulli-perc} implies it for all $\tilde p'>\tilde p$ by monotonicity of the strong supercritical property in $\tilde \omega$. 

\begin{theorem}\label{thm:subexp-mixing}
	There exists $\eta_0(\delta) > 0$ and  $p_0(\Delta,q,\delta,\tilde p) < 1$, such that for every graph $G$ with a strong supercritical phase (with parameters $\delta,\tilde p$) and $\eta$-stretched-exponential volume growth for some $\eta \le \eta_0$: 
        \begin{enumerate}
            \item The mixing time of SW dynamics on $G$ is $O(n^2 \log n)$ for every $p\ge p_0$. 
            \item The mixing time of the FK dynamics on $G$ is $O(n \log n)$ for every $p\ge p_0$.
        \end{enumerate}
\end{theorem}

It is natural to wonder what families of graphs have a strong supercritical phase. The nature of the supercritical phase for edge-percolation on a graph is known to be closely related to the geometric, namely isoperimetric, properties of the graph: see e.g.,~\cite{ABS}. One would expect that general graph families with \emph{isoperimetric dimension} at least $1+\delta$ (meaning 
that $|\partial_e A|\ge |A|^{\delta/(1+\delta)}$ for all subsets $A\subset V(G)$ with $|A|\le n/2$) have a strong supercritical phase in the sense of Definition~\ref{def:supercritical-Bernoulli-perc}. Often at sufficiently large $\tilde p$, a strong supercritical phase can be proven using perturbative Peierls-type arguments; by such means, subsets of $\mathbb Z^d$ and other lattices (e.g., hexagonal and triangular) that are uniformly at least $(1+\delta)$-dimensional,
and planar graphs with a bounded-degree planar dual serve as concrete examples of graphs that have a strong supercritical phase. 
More generally, the structure of the supercritical phase in vertex-transitive graphs of polynomial growth is the subject of deep study (see e.g.,~\cite{contreras2022supercritical,HutchcroftNonTriviality,EasoHutchcroftSupercritical} which tackle the harder problem of understanding the supercritical phase down to a sharp threshold). Natural graphs of super-polynomial but at most stretched-exponential growth are rarer, but one such family are the well-known construction of Grigorchuk groups; in these, the precise $\eta$ in the stretched-exponential growth, and the nature of the graphs' supercritical phase are subjects of active research: see e.g.,~\cite{Grigorchuk-introduction}.

Our proof of Theorem~\ref{thm:subexp-mixing} relies
on a novel framework for controlling the rate at which discrepancies spread between two coupled low-temperature FK dynamics chains that agree inside, say, a ball of radius $R$ around a vertex, but that may disagree outside it. This is sometimes called disagreement percolation, and we use it, 
after a \emph{burn-in period} for the chain (a short period of time after which we can ensure that the certain ``typical" properties of random-cluster configurations are achieved, even though the chain has not equilibrated), to perform space-time recursions to derive our mixing time bounds. 
To the best of our knowledge, this is the first time disagreement percolation has been analyzed in a low-temperature setting for non-local Markov chains like the SW or FK dynamics, where the giant component could hypothetically spread disagreements instantaneously (except in the special case of $\mathbb Z^2$ where low and high temperatures are dual to one another). 
We say more about the obstacles to proving Theorem~\ref{thm:subexp-mixing} using existing tools and the technical novelties in our low-temperature disagreement percolation framework in Section~\ref{subsec:proof-ideas}. 

Since~\cite{VanDenBerg}, bounds on the rate of disagreement percolation have been a tool used to prove a variety of other results for spin systems, including bounds on uniqueness thresholds, equivalences of spatial and temporal mixing~\cite{DSVW}, and tight lower bounds for the mixing time of the Glauber dynamics~\cite{HayesSinclair}.  We do not explore these directions here, but our low-temperature disagreement percolation, which is self-contained to Section~\ref{sec:low-temp-disagreement-percolation}, opens up those same arguments for the low-temperature random-cluster model and its dynamics. For example, extending the lower bounds of~\cite{HayesSinclair} would show that the $O(n\log n)$ in (2) in Theorem~\ref{thm:subexp-mixing} is tight; by contrast, the resulting lower bound for SW dynamics would be $\Omega(\log n)$.

\subsection{Locally treelike graphs}\label{subsec:exp-growth}
We consider next the SW and FK dynamics on
locally treelike graphs.
In this setting, we establish fast mixing on graphs that have a strong supercritical phase with ``$\delta = \infty$" in Definition~\ref{def:supercritical-Bernoulli-perc}. 
That is to say that we assume true exponential tails on the boundaries of non-giant components, with a rate that goes to $\infty$ as $\tilde p\uparrow 1$, to compete with the exponential volume growth.

\begin{definition}\label{def:supercritical-Bernoulli-exp}
    We say that $G$ has an \emph{exponentially strong supercritical phase} 
    if there exists $\tilde p_0<1$ such that for every $\tilde p>\tilde p_0$ and every $v\in V(G)$, the probability that there exists a connected set $A\ni v$ having $\ell \le |A|\le n/2$ with $\widetilde \omega(\partial_e A)\equiv 0$ is at most $\exp( - c_{\tilde p} \, \ell)$ for some $c_{\tilde p}$ going to $\infty$ as $\tilde p$ to $1$. 
\end{definition}

\noindent
While the notion of an exponentially strong supercritical phase is a property of independent edge-percolation on the graph, a simple geometric criterion of expansion, for instance, ensures that this property holds. In particular, if $G$ is an $\alpha$-edge-expander graph, in the sense that for all $A\subset V(G)$ such that $|A|\le n/2$, we have $|\partial_e A|\ge \alpha |A|$, then the assumption holds for a $c_{\tilde p}(\alpha,\Delta)>0$.  

In this regime where exponential volume growth is permitted, we restrict to locally treelike graphs. 

\begin{definition}\label{def:locally-treelike}
     We say a graph $G$ is $(K,L)$-locally treelike if for every $v\in V(G)$, the removal of at most $K$ edges from $E(B_L(V))$ induces a tree on $B_L(v)$.  
\end{definition}

Our main result for locally treelike graphs is the following near-optimal fast mixing bound.

\begin{theorem}\label{thm:exp-growth-mixing}
	% Fix $\Delta \ge 3$, $q > 1$, and $\eta,K > 0$. 
    Fix any $\epsilon,\eta>0$. There exists $p_0(\Delta,q,K,\tilde p_0,c_{\tilde p},\eta, \epsilon)<1$ such that if $G$ has an exponentially strong supercritical phase (with parameter $\tilde p_0$), minimum degree $3$, and is $(K,\eta \log n)$-locally treelike: 
        \begin{enumerate}
            \item The mixing time of SW dynamics on $G$ is $O(n^{2+\epsilon})$ for every $p\ge p_0$. 
            \item The mixing time of FK dynamics on $G$ is $O(n^{1+\epsilon})$ for every $p\ge p_0$.
        \end{enumerate}
\end{theorem}

The most canonical example of a graph 
that satisfies all the conditions in this theorem is a $\Delta$-regular random graph (i.e., a graph drawn uniformly at random from the set of all $\Delta$-regular graphs on $n$-vertices). This is a setting that has attracted plenty of attention (see, e.g.,~\cite{BGGSVY,BG20,PottsRGMetastabilityCMP,HeJePe20,GGS23}), and Theorem~\ref{thm:exp-growth-mixing} provides 
fast mixing bounds for the SW and FK dynamics on these graphs at low temperatures. (Note that the bounds will hold with probability $1-o(1)$ over the choice of the random graph.)

Unlike the sub-exponential growth setting, alternative sampling algorithms were known to exist for the  Potts model on expander graphs at low temperatures using the cluster expansion and polymer dynamics (see, e.g.,~\cite{JKP,CGGPS,BlancaCannonPerkinsPolymerSampling,Carlson}). 
Still, to our knowledge, ours is the first proof of sub-exponential mixing times for the SW and FK dynamics at low temperatures (even just for random graphs). 

Regarding the proof techniques, on graphs of exponential growth, the low-temperature disagreement percolation framework used to establish Theorem~\ref{thm:subexp-mixing} breaks down. Even in ideal situations like the Ising model, the optimal recursion obtained from the disagreement percolation framework would not yield rapid mixing on graphs of exponential growth. We, therefore, resort to a vastly different approach, where we utilize a burn-in phase, the censoring technique of~\cite{PWcensoring}, and new spatial mixing results for the random-cluster model on trees amongst a (random) class of sufficiently wired boundary conditions. The latter bound applies in settings where spatial mixing between the wired and free boundary conditions does \emph{not} hold. 

\begin{remark}\label{rem:non-integer-q}
    The bounds of Theorems~\ref{thm:subexp-mixing} and~\ref{thm:exp-growth-mixing} are stated for any integer $q\ge 2$ so that statements apply both to the SW and FK dynamics. The random-cluster model also makes sense for non-integer $q\ge 1$ and our fast mixing results for the FK dynamics apply in this level of generality. 
    In fact, the random-cluster model is defined for $q>0$ but has very different features (negative vs. positive correlations) when $q\in (0,1)$. It was shown in~\cite{LogConcaveII-Annals} that the FK dynamics mixes in polynomial time on all graphs when $q\in (0,1)$.
\end{remark}

\subsection{Slow mixing in worst-case graphs}\label{subsec:slow-mixing}
We complement our fast mixing result by establishing the existence of graphs for which, even at arbitrarily low temperatures, the SW and FK dynamics slow down exponentially. This is already suggested, though not guaranteed, by the \#BIS-hardness of the sampling problem at low temperatures, and our constructions will illuminate the relationship between the notion of a strong supercritical phase for the underlying edge-percolation and the slow mixing of the dynamics. 

\begin{theorem}\label{thm:slow-mixing}
	Fix any $q\ge 3$ and any $p_0<1$. There exists $p\in (p_0,1)$ and a sequence of graphs $G_n$ on $n$ vertices and maximum degree $\Delta$ such that the mixing time of the SW and FK dynamics on $G_n$ is $\exp(\Omega(n))$.  
\end{theorem}
The constructions for Theorem~\ref{thm:slow-mixing} are simple and explicit. In particular, any family of graphs $H_n$ that have slow mixing at \emph{some} parameter value $p_s\in (0,1)$---typically the location of its order/disorder phase transition---can be used as a gadget to construct augmented graphs $G_n$ (depending on $p_s$ and $p_0$) with many of the same properties as $H_n$ (in terms of degree, rate of volume growth, etc.), and a comparable number of edges, for which the SW and FK dynamics are slowly mixing at some $p\in (p_0,1)$. The graph augmentation leverages the \emph{series law} of the random-cluster model to repeatedly split the edges of $H_n$, effectively inducing the behavior at $p_s$ in $H_n$ to occur in $G_n$ at $p\in (p_0,1)$. Using the slow mixing of SW and FK dynamics at the critical point on random regular graphs from~\cite{PottsRGMetastabilityCMP} as the gadget, 
Theorem~\ref{thm:slow-mixing} holds even if we impose that the graph is locally treelike and has exponential volume growth. 
Using the slow mixing at the critical point on $(\mathbb Z/n\mathbb Z)^2$ from~\cite{GL1}, a variant of this theorem also holds for graphs of polynomial growth, but the lower bound there is of the form  $\exp(\Omega(\sqrt{n}))$: see Theorem~\ref{thm:slow-mixing-FK}.

\begin{remark} 
Let us comment on the relationship of Theorem~\ref{thm:slow-mixing} to Theorems~\ref{thm:subexp-mixing} and~\ref{thm:exp-growth-mixing}, given the slow mixing constructions can either have stretched-exponential growth or be locally treelike. Even if $H_n$ has a strongly supercritical phase for its edge-percolation, when we perform the graph augmentation with the series law, the $\tilde p$ for which the edge-percolation on $G_n$ has a strongly supercritical phase is pushed closer to $1$. In geometric language, this is because the isoperimetric dimension is decreasing to $1$, or the edge expansion is decreasing to $0$, as the edges are split in series. In turn, this makes the $p_0$ in Theorems~\ref{thm:subexp-mixing} and~\ref{thm:exp-growth-mixing} (above which we can prove fast mixing) larger than the $p$ for which Theorem~\ref{thm:slow-mixing} gives slow mixing. 
\end{remark}

\subsection{Proof ideas}\label{subsec:proof-ideas}
We now discuss our proof ideas for the fast mixing results, which are the more technically involved. (Our bounds on the SW dynamics follow from the bounds on the FK dynamics by~\cite{Ullrich-random-cluster}, so we focus on the FK dynamics.)
We begin by describing some of the issues one runs into when applying standard proof approaches to general families of graphs at large $p$. 

\subsubsection{Difficulty with classical arguments}
The first tool one might try is path coupling, arguing that the number of discrepancies between two configurations that differ on one edge contracts in expectation. The non-locality of the FK dynamics, however, and the presence of $\Omega(\log n)$ sized components at equilibrium means that a single discrepancy at an edge $e$ can cause discrepancies at some $\Omega(\log n)$ many nearby edges, whereas the discrepancy only decreases if the edge $e$ is selected to be updated. A smarter path coupling was used in~\cite{CF,Huber} to deduce fast mixing for the SW dynamics at high enough (but constant) temperatures, but in the low-temperature regime, their argument for fast mixing requires $p>1-O(1/|E(G)|)$.

Many of the early fast mixing bounds on, say, the high-temperature Ising model, are based on space-time recursions, i.e., arguments that compare the distance to stationarity across balls of time-dependent radii. 
When translated to the FK dynamics, this type of argument runs into the problem that the updates in a small portion of the graph (say, a small ball around a vertex) could  depend on the configuration in the entire remainder rather than a local neighborhood. The one exception to this is the approach of~\cite{MaOl1} for the torus in $\mathbb Z^d$, which gave an implication from the weak-spatial mixing (WSM) condition to fast mixing of the Glauber dynamics. This implication was seen to generalize to the FK dynamics in~\cite{HarelSpinka} (see also~\cite{GhSi22} where finite boxes with boundary conditions were allowed). WSM is known to hold for the random-cluster model at large $p$ on $\mathbb Z^d$ (see e.g.,~\cite{Grimmett}); however, this is a delicate property whose proofs are very geometry specific. It is the case, for example, that on locally treelike graphs like the random regular graph, WSM fails at arbitrarily large $p$.

At high temperatures (small $p$), some of the difficulties with non-locality can be handled using the fact that in the  random-cluster model, all connected components are small, and information is only propagated through these connected components. 
For instance, such an argument was used in~\cite{BS} to implement the disagreement percolation space-time recursion for the high-temperature regime on $\mathbb Z^2$. (In $\mathbb Z^2$, the high and low-temperature regimes are dual to one another, so the same argument could be performed using the dual model at low temperatures; that would be similar to the work of~\cite{Martinelli-SW} on the Ising SW dynamics.)   

\subsubsection{Low-temperature disagreement percolation bounds}

On graphs where the low-temperature random-cluster model does not have a natural high-temperature dual model, however, even at equilibrium, the non-locality of the dynamics is hypothetically not confined since the giant component percolates through the entire graph. The starting point for many of our observations is that a (well-connected) giant component does not create non-local dependencies on its own. In particular, if two configurations that agree at distance $R$ away from an edge $e$ induce different marginals on $e$, it must be the case that in one of the two configurations, either $e$ is incident to a \emph{non-giant} component of size at least $R$, or it disconnects a portion of the giant of size at least $R$ from the 2-connected core of the giant. Whereas the giant component percolates throughout the whole graph, we show that under the assumption of a strong supercritical phase (Definition~\ref{def:supercritical-Bernoulli-perc}), these non-giant, or non-2-connected core connections have (stretched) exponential tails in FK dynamics configurations after an $O(n)$ burn-in period. (In the first $O(n)$  steps, disagreements can spread arbitrarily quickly.)

There are various other delicate points in implementing this argument, both combinatorial and probabilistic in nature, that we describe in greater detail in Section~\ref{sec:low-temp-disagreement-percolation} and~\ref{sec:fast-mixing-subexp}. These include having to carefully approach various counting arguments and union bounds due to the non-localities and possible stretched exponential volume growth: see Remark~\ref{rem:disagreement-percolation-proof-difficulties} and the proof strategy described in Section~\ref{subsec:burn-in-subexp}. In all, we are able to obtain a space-time recursion on the probability of a disagreement at an edge after time $t$; the large value of $p$ is used as a crude initial bound on this probability, 
which the recursion boosts into exponential decay, leading to the optimal $O(n\log n)$ mixing time bound of Theorem~\ref{thm:subexp-mixing}.  

\subsubsection{Mixing after a burn-in phase on locally treelike graphs}
When the volume growth is exponentially fast, the bounds and resulting space-time recursions from disagreement percolation break. Our approach here is therefore closer in inspiration to high-temperature arguments from~\cite{BG20,BG22} (also~\cite{MS09} in the Ising setting). Those papers localized the dynamics to the treelike balls $B_R(v)$ of the underlying graph using the censoring technique of~\cite{PWcensoring} and used the high-temperature uniqueness on trees to reason that if two censored dynamics chains mix in $B_R(v)$ with their respective boundary conditions, then they are coupled at the root of the ball with high probability. The mixing time on the local balls was relatively simple to deduce since the trees would have nearly free boundary conditions, which induce product chains. 

In our low-temperature setting, the key intuition is that after a burn-in period, the boundary conditions induced on balls of radius $\eta \log n$ are ``sufficiently wired", i.e., that they have one (random) linear-sized wired component, and only $O(1)$ many other $O(1)$-sized components. Given this, to get Theorem~\ref{thm:exp-growth-mixing}, we show that FK dynamics on trees with such boundary conditions mix in polynomial time, and that although there is no WSM between the free and wired boundary conditions on trees, 
two sufficiently wired boundary conditions do induce similar marginals on the root. This latter step requires a careful revealing scheme to prove spatial mixing on trees with randomly wired boundary conditions; that is the content of Section~\ref{sec:tree-spatial-mixing}.  

\subsubsection{On the strong supercritical phase condition}
We conclude by remarking on whether the notions of strong supercritical phase in Definitions~\ref{def:supercritical-Bernoulli-perc} and Definition~\ref{def:supercritical-Bernoulli-exp} could be relaxed. One attempt at such a relaxation would be to only ask that the independent $\hat p$-edge percolation have a unique giant component (of arbitrarily small linear size) and exponential, or stretched-exponential, tails on all its non-giant components. For technical reasons related to the fact that this is a non-monotone criterion, our proofs do not go through with this weaker notion of supercriticality. Nonetheless, we expect that such a condition could be sufficient for the fast mixing of the FK and SW dynamics.

    \section{Notation and preliminaries}\label{sec:prelim}
    In this section, we outline our global notation and describe some preliminaries on the random-cluster model and the FK dynamics. Throughout the paper, $n$ will be assumed to be sufficiently large. We also use $C$ to denote a generic constant $C>0$, not depending on $n$, which may vary from line to line. 
    Our underlying graph will be $G = (V(G),E(G))$ and will have $n$ vertices and maximum degree $\Delta$. 
    
    For an edge-subset $A\subset E(G)$, we write $V(A)$ for the set of vertices contained in edges in $A$, though we sometimes abuse notation and write $v\in A$ for $v\in V(A)$. 
    Its vertex boundary $\partial A$ is the set of vertices in $A$ with neighbors in $A^c = E(G)\setminus A$. Its (outer) edge-boundary $\partial_e A$ is the set of edges in $E(G)\setminus A$ that have one end-point in $V(A)$ and one endpoint in $V(G)\setminus V(A)$. 
    We use $\mathcal C_v(A)$ to denote the connected component of $v$ in the subgraph $(V(G),A)$. 
    An edge $e$ is a \emph{cut-edge} in $A$ if there is a vertex $v$ for which $\mathcal C_v(A\cup \{e\}) \ne \mathcal C_v(A\setminus \{e\})$. We use $\mathcal C_1(A)$ to denote the largest component in $\omega$ (chosen arbitrarily if two have the same size). 
    
    \subsection{The random-cluster model}
    The random-cluster model with parameters $p \in (0,1)$ and $q >0$ is a probability distribution over edge-subsets $\omega\subseteq E(G)$, equivalently identified with $\omega\in \{0,1\}^{E(G)}$, given by
    \begin{equation}\label{eq:fk-definition}
        \pi_{G,p,q}(\omega) \propto \, p^{|\omega|}(1-p)^{|E(G)|-|\omega|} q^{k(\omega)}\,,
    \end{equation}
    where $k(\omega)$ denotes the number of connected components in $(V(G),\omega)$.
    When clear from context, we drop $p$ and $q$ and sometimes $G$ from the notation.
    The random-cluster model satisfies the following \emph{domain Markov property}: for $A\subset E(G)$, conditional on $\omega(E(G)\setminus A)$, the distribution of $\omega(A)$ is a random-cluster model with the same parameters, with \emph{boundary conditions} on $\partial A$ induced by $\omega(E(G)\setminus A)$. 
    Random-cluster boundary conditions are defined in generality as follows. 
    
    \begin{definition}\label{def:boundary-conditions}
        Given a graph $G$ and a vertex subset $\partial B$, a boundary condition $\xi$ on $\partial B$ is a partition of $\partial B$. The random-cluster model on $G$ with boundary conditions $\xi$, denoted $\pi_G^\xi$ is defined as in~\eqref{eq:fk-definition}, except that all components intersecting the same element of $\xi$ are identified (``wired") when counting the number of components $k(\omega)$.
    \end{definition}
    Certain important boundary conditions are the \emph{wired} one, denoted $\mathbf{1}$, where all vertices of $\partial A$ are in the same component of $\xi$; the \emph{free}, denoted $\mathbf 0$, where all vertices of $\partial A$ are in distinct components of $\xi$, and if $A$ is a subgraph of $G$, those \emph{induced} by $\omega(E(G)\setminus E(A))$, meaning vertices of $\partial A$ are in the same component of $\xi$ if and only if they are connected through $\omega(E(G)\setminus E(A))$. 
    In this paper, we restrict our attention to the case of $q\ge 1$ where the model exhibits positive correlations. As a consequence, given two boundary conditions $\xi$ and $\xi'$ on $G$, where $\xi \ge \xi'$ (meaning $\xi$ is a coarser partition than $\xi'$), we have $\pi^\xi_G \succeq \pi^{\xi'}_G$.

    \subsection{Mixing times}
    For a Markov chain $(X_k)_{k}$ on a finite state space $\Omega$ with transition matrix $P$, reversible with respect to a distribution $\mu$, 
    its \emph{mixing time} is defined as:
    \begin{align}\label{def:tmix}
    \tmix = \tmix(1/4)\,,\qquad \mbox{where}\qquad \tmix(\epsilon) = \min \{k: \max_{x_0 \in \Omega} \| \mathbb P(X_k^{x_0} \in \cdot ) - \mu\|_\tv \le \epsilon\}\,,
    \end{align}
    where $X_k^{x_0}$ indicates that $X_k$ is initialized from $x_0$, and where $\|\cdot \|_\tv$ denotes total-variation distance. The total-variation distance to $\mu$ satisfies a sub-multiplicativity property, whereby $\tmix(\epsilon) \le \tmix\cdot \log_2(2/\epsilon)$. 
    
    \subsection{FK dynamics}
    \label{subsec:coupling}
    Recall the definition of (discrete-time) FK dynamics from the introduction. It will be preferable in our proofs to work with the continuous-time FK dynamics $(X_t)_{t>0}$. In this variant, the edges of $E(G)$ are assigned rate-1 Poisson clocks, and if the clock at an edge $e$ rings at time $t$, we make an update according to~\eqref{eq:FK dynamics-update-rule}.  It is a standard fact that the mixing time of the discrete-time chain is comparable, up to constants, with $|E(G)|$ times the mixing time of the continuous-time process. In particular, it suffices to show an $O(\log n)$ bound for Theorem~\ref{thm:subexp-mixing} and an $n^{\epsilon}$ bound for Theorem~\ref{thm:exp-growth-mixing} for the continuous-time FK dynamics. FK dynamics updates with boundary conditions $\xi$ are like~\eqref{eq:FK dynamics-update-rule}, except that the cut-edge status of $e$ is determined taking into account the wirings of the components of $\xi$. 

   The FK dynamics is monotone, meaning that if $x_0 \ge y_0$ (under the natural partial order on subsets) then $X_t^{x_0}\succeq X_t^{y_0}$ for all $t\ge 0$, where $\succeq$ denotes stochastic domination. I.e., there exists a \emph{grand coupling} of all the Markov chains $\{(X_t^{x_0})_t\}_{x_0\in \Omega_\rc}$ (generated by independent Poisson clocks and independent sequences of i.i.d.\ $\text{Unif}[0,1]$ random variables at every edge) such that $X_t^{x_0}\ge X_t^{y_0}$ for all $x_0\ge y_0$ and $t\ge 0$. 

    A further monotonicity property we use is with respect to the independent edge-percolation (the random-cluster model with $q=1$). Recall $\ps$ from~\eqref{eq:FK dynamics-update-rule}; it is standard fact that $\pi_{G,p,q} \succeq \pi_{G,\ps,1}$ (see (3.23) in~\cite{Grimmett}).

\section{Low-temperature disagreement percolation}\label{sec:low-temp-disagreement-percolation}

In this section, we develop the FK dynamics disagreement percolation framework that works at sufficiently low temperatures, in particular in the presence of a giant component. In reality, this new disagreement percolation bound works simultaneously at high and low temperatures and localizes the spread of disagreements even in the presence of a large giant component, so long as all other components (and portions of the giant dangling off of its 2-connected core) are small. Moreover, it can work on graphs that have volume growth up to a stretched exponential, which requires new ideas: see Remark~\ref{rem:disagreement-percolation-proof-difficulties}.

In this section, $G$ can be an arbitrary graph of maximum degree $\Delta$. We fix an arbitrary $o\in V(G)$ and $R > 0$, and let $B_R= B_R(o)$. The dependencies on $o$ will be kept implicit.  
The fundamental building blocks of our disagreement set will be \emph{finite-connectivity} clusters; these will disentangle the non-locality of the giant component, which percolates at low temperature, from the edges through which disagreements arise. 

\begin{definition}
    Define the \emph{finite (or non-giant) component} of a vertex $v$ in a random-cluster configuration $\omega$ as $\mathcal C_v(\omega\setminus E(\mathcal C_1(\omega)))$, and denote it by $\mathcal C_v^{\ne 1}(\omega)$.     
\end{definition}

    Since the FK dynamics updates at edge $e=\{u,v\}$ 
    are oblivious to the state of $e$ in the configuration, 
    and only care about the connectivity of $u$ and $v$ in $\omega\setminus e = \omega\setminus \{e\}$, we consider the finite component of a vertex with respect to the configuration $ \omega\setminus e$ rather than $\omega$ itself. 
    Specifically, we often consider 
    $\mathcal C_v^{\ne 1}(\omega\setminus e)$
    for an edge $e$ that is incident to $\mathcal C_v(\omega)$.
%    In that case,  
%    if $v \not\in \mathcal C_1(\omega)$, then
%    $\mathcal C_v^{\ne 1}(\omega\setminus e) = \mathcal C_v(\omega\setminus e)$;
%    if $v \in \mathcal C_1(\omega)$ but $v \not\in \mathcal C_1(\omega\setminus e)$, we also have $\mathcal C_v^{\ne 1}(\omega\setminus e) = \mathcal C_v(\omega\setminus e)$; otherwise, $\mathcal C_v^{\ne 1}(\omega\setminus e) = \{v\}$.

    \begin{definition}
            Let $\mathsf{CE}_v^{\ne 1}(\omega)$ be the set of cut-edges in $\omega$  that are in $B_R$ and incident to 
            $\mathcal C_v^{\ne 1}(\omega\setminus e)$; i.e., 
        \begin{align*}
            \mathsf{CE}_v^{\ne 1}(\omega) := \{ e \in E(B_R): e\in \mathrm{Cutedge}(\omega)\,,\, e \sim \mathcal C^{\ne 1}_v(\omega \setminus e)\}\,,
        \end{align*}
        where $\text{Cutedge}(\omega)$ denotes the set of cut-edges in $\omega$. Here we are using the notation $e\sim H$ for a subset $H\subset E$ if $V(e)\cap V(H) \neq \emptyset$. 
    \end{definition}
    
%    Observe that for a configuration $\omega$ and a vertex $v$, if $v \not\in C_1(\omega)$, then $\mathsf{CE}_v^{\ne 1}(\omega)$
%           contains all edges of $\partial_e \mathcal C_v(\omega)$ together with all cut-edges from $E(C_v(\omega))$. 
%    Otherwise, if $v \in \mathcal C_1(\omega)$, the set $\mathsf{CE}_{v}^{\ne 1}$ essentially captures the set of cut-edges in $\mathcal C_v(\omega)$ that disconnect $v$ from the 2-connected core of $C_1(\omega)$. 
    We refer to Figure~\ref{fig:CE-sets} for some illustrative depictions of such sets in $\mathbb Z^2$ (this is easiest for visualization, but it is key that our definitions do not rely on properties of $\mathbb Z^2$ like its dual graph, and thus work on general graphs). We also refer to the proof of Proposition~\ref{prop:disagreement-region} which yields additional insight into these constructions.

    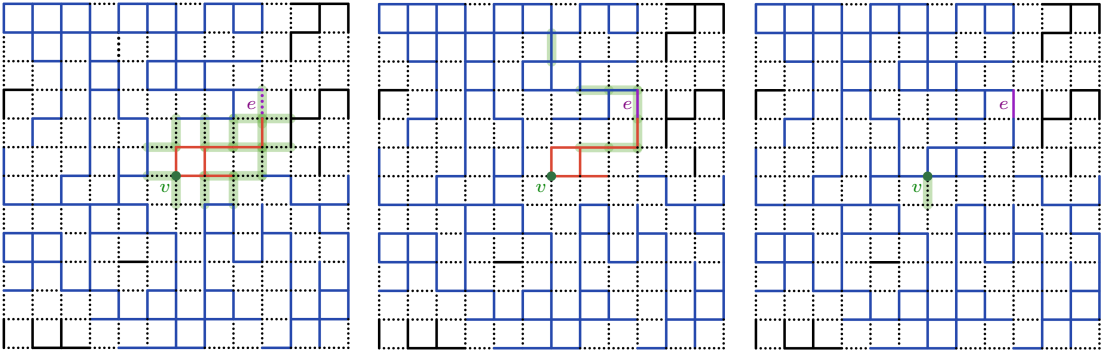
\begin{figure}
        \centering
        \begin{tikzpicture}
        \node at (-5,0) {
        \includegraphics[scale=0.2]{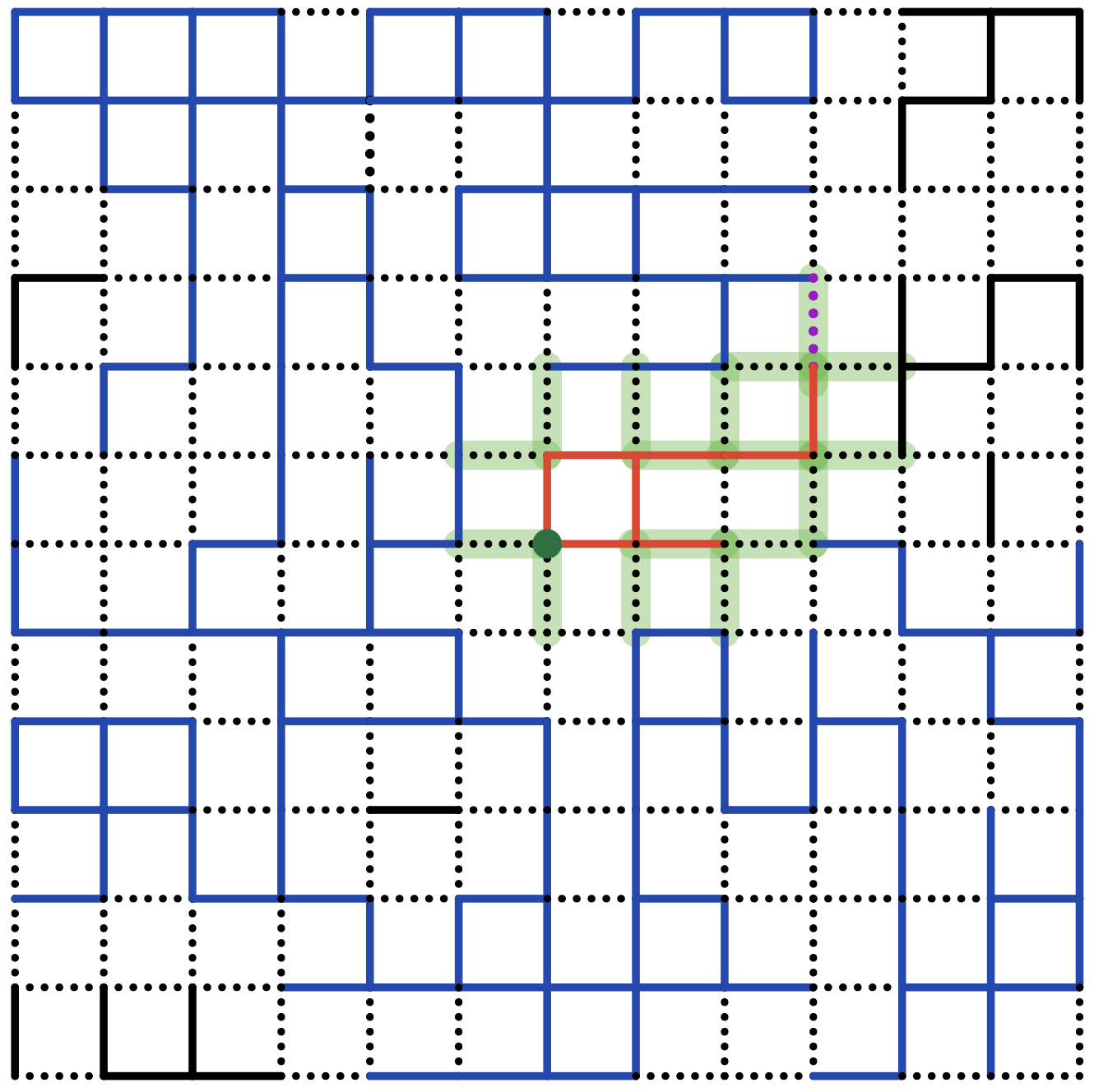}};  
        \node[font = \tiny, color = green!50!black] at (-5.15,-.15) {$v$};
        \node[font=\tiny, color = blue!50!red] at (-4,.95) {$e$};
        \node at (0,0) {
        \includegraphics[scale=0.2]{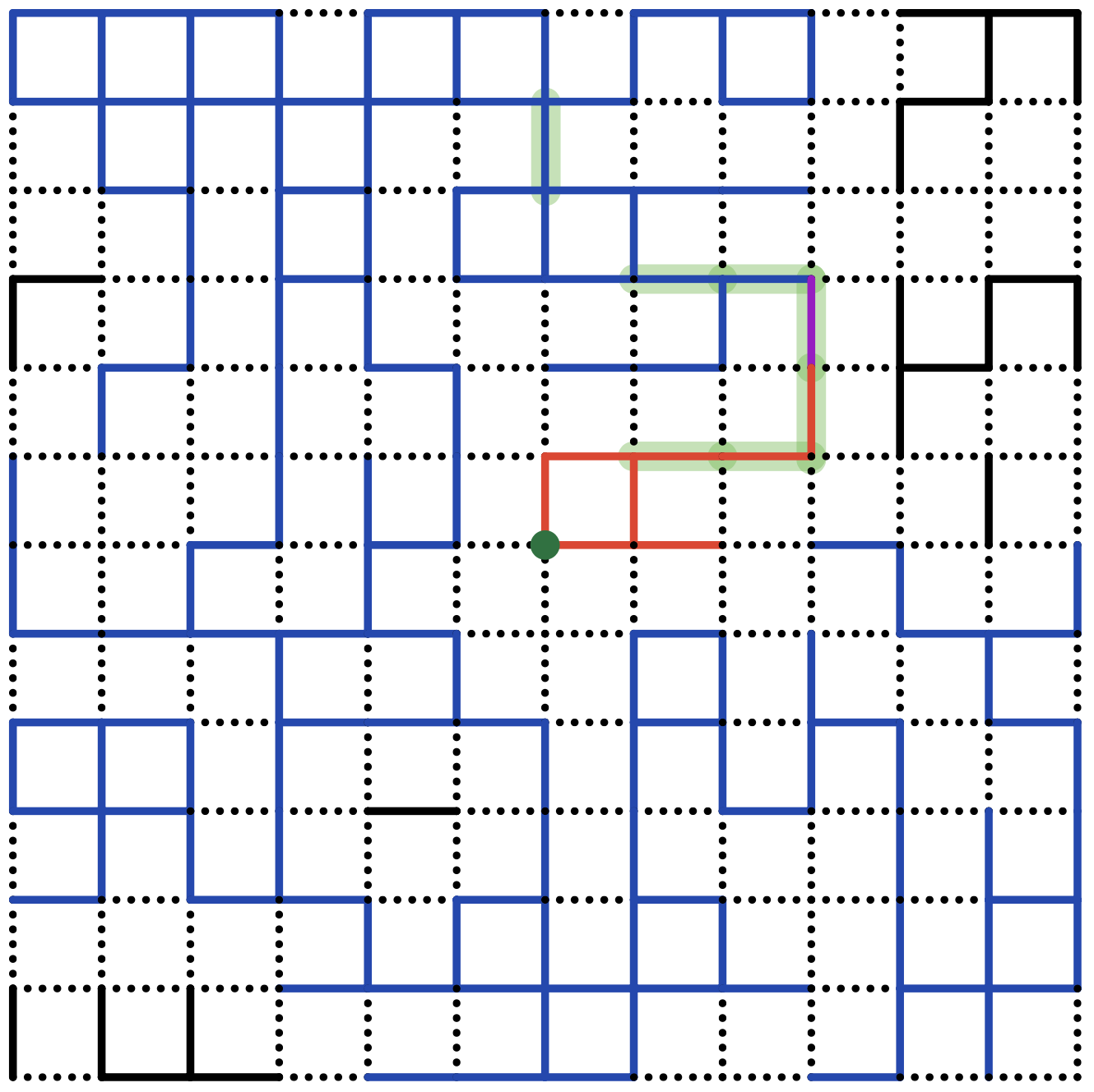}};
        \node[font = \tiny, color = green!50!black] at (-.15,-.15) {$v$};
        \node[font=\tiny, color = blue!50!red] at (1,.95) {$e$};
        \node at (5,0) {
        \includegraphics[scale=0.2]{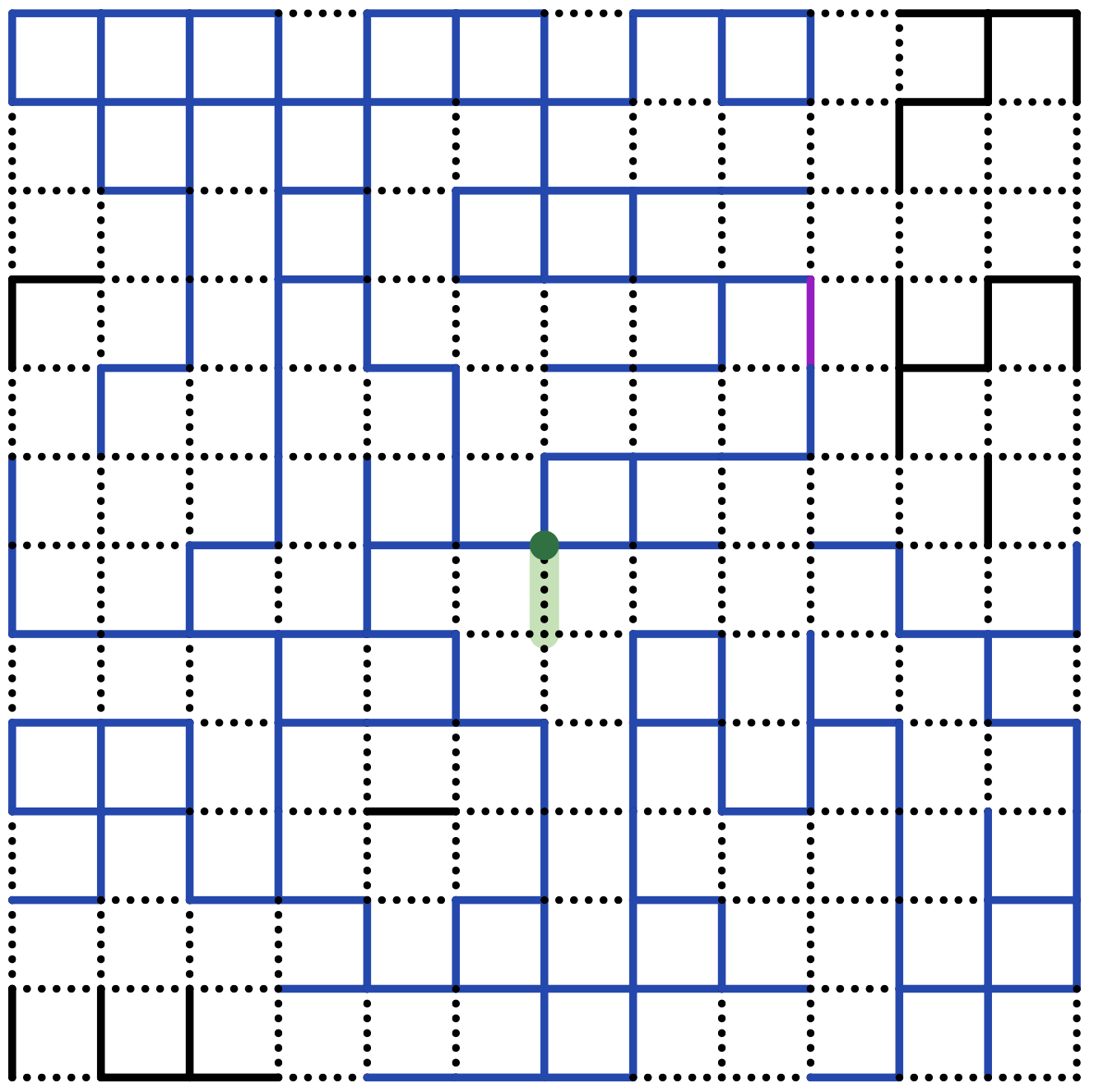}};
        \node[font = \tiny, color = green!50!black] at (4.85,-.15) {$v$};
        \node[font=\tiny, color = blue!50!red] at (6,.95) {$e$};
          \end{tikzpicture}
        \caption{For a vertex $v$ (dark green) and edge $e$ (purple), the sets $\mathcal C_v^{\ne 1}(\omega\setminus e)$ (edges in red) and $\mathsf{CE}_v^{\ne 1}(\omega)$ (edges highlighted in green) are shown in three different cases. Left: $v$ is not part of the giant (blue edges) in $\omega$ (blue and black edges). Middle: $v$ is part of the giant component but not of its $2$-connected core. Right: $v$ in the $2$-connected core of the giant.}
        \label{fig:CE-sets}
      
    \end{figure}

    Suppose $(X_t)$ and $(Y_t)$ are two instances of the FK dynamics on $G$ coupled via the grand coupling introduced in Section~\ref{subsec:coupling}; suppose also that $X_0(E(B_R))=Y_0(E(B_R))$. 
    
    \begin{definition}\label{def:disagreement-region}
        Iteratively construct what we call the \emph{disagreement set} as follows. Let $(t_i)_{i\ge 1}$ be the times of the clock rings in $E(B_R)$, let $t_0 =0$, and let $I_i = [t_{i-1},t_i)$. Then 
        \begin{enumerate}
            \item Initialize $\mathcal D_t = E(G) \setminus E(B_R)$ for $t\in I_1$. 
            \item Suppose $e_i$ is the edge whose clock rings at time $t_i$. If $e_i$ is in $E(B_R) \setminus \mathcal D_{t_i^-}$ and $e_i$ is in $\mathsf{CE}_v^{\ne 1}(Z)$ for some $v\in \partial \cD_{t_{i}^-}$ and $Z\in \{X_{t_i^-},Y_{t_i^-}\}$, let 
            $$ \cD_{t} = \cD_{t_i^-} \cup \{e_i\}\,\qquad \text{for all $t\in I_{i+1}$};$$
            else, let $\cD_t = \cD_{t_i^-}$ for all $t\in I_{i+1}$. (We use the standard notation $A_{t^-}$ to denote $\lim_{s\uparrow t}A_s$.)
        \end{enumerate}
    \end{definition}

    The role of $\cD_t$ is that it confines the set of edges on which a disagreement can possibly exist at time $t$.

    \begin{proposition}\label{prop:disagreement-region}
        For all $t\ge 0$, we have $X_t(e) = Y_t(e)$ for all $e\notin \mathcal D_t$. 
    \end{proposition}

    \begin{proof}
        We prove the claim inductively. It holds for all $t\in I_1$ since we assumed $X_t(e) = Y_t(e)$ for $e\in E(B_R)$, and no clock rings occur in the interval $I_1$. Supposing it holds for $I_{i}$, the only way it can not hold for $t\in I_{i+1}$ is if the disagreement arises at the edge $e_i=\{u_i, v_i\}$ when the clock rings at time $t_i$. If $e_i \in \cD_{t_i^-}$, then since $\cD_{t_i^-} \subset \cD_t$, we have the claim. 
        Per~\eqref{eq:FK dynamics-update-rule}, if $e_i \notin \cD_{t_i^-}$, in order for a disagreement to arise at $e_i$, it must be the case that $e_i$ is a cut-edge in one of $X_{t_i^-}$, $Y_{t_i^-}$ but not in the other; that is, that $e_i \in \text{Cutedge}(Z)$ for a $Z\in \{X_{t_i^-}, Y_{t_i^-}\}$, but $e_i \not\in \text{Cutedge}(\hat Z)$ for $\hat Z = \{X_{t_i^-}, Y_{t_i^-}\} \setminus \{Z\}$.
        In $\hat Z\setminus \{e_i\}$, the two endpoints of $e_i$ must be connected. 
        Since $X_{t_i^-}(\cD_{t_i^-}^c) = Y_{t_i^-}(\cD_{t_i^-}^c)$, this can only happen via a pair of paths from $u_i$ and $v_i$ that reach $\partial \cD_{t_i^-}$ in $\cD_{t_i^-}^c$. 
        One of these paths must be in 
        $Z\setminus E(\mathcal C_1(Z\setminus\{e_i\}))$, 
        since if both are in $E(\mathcal C_1(Z\setminus \{e_i\}))$, then $u_i$ and $v_i$ are in the same connected component of $Z\setminus e_i$, contradicting the claim that $e_i$ was a cutedge in $Z$. As such, it must be that $e_i \in \mathsf{CE}_v^{\ne 1}(Z)$ for some $v\in \partial \cD_{t_i^-}$. 
    \end{proof}

\noindent
    We now define an event on the realizations of the coupling $(X_t,Y_t)_{t \ge 0}$ that guarantees that disagreements are unlikely to spread rapidly. This will be a bound on the size of finite connections in $E(B_R)$, as well as a bound on the number of cut-edges in a finite component, which we will later show holds with high probability for the FK dynamics after a burn-in for $G$ satisfying the conditions of Theorem~\ref{thm:subexp-mixing}: see Proposition~\ref{prop:burnt-in-is-good}

    \begin{definition}\label{def:E-l-t}
        A random-cluster configuration $\omega$ is in  $\mathcal E_{\ell,\alpha,\gamma}$ if:
        $$
        (1) \quad \max_{v\in B_R} \max_{e\in E(B_R)}\, \diam( \mathcal C^{\ne 1}_{v}(\omega\setminus e))\le \ell^\alpha\,,\quad \text{and} \quad(2) \quad\max_{v\in B_R} |\mathsf{CE}_v^{\ne 1}(\omega)| \le \Delta \ell^\gamma.
        $$
    \end{definition}
    In the applications in Section~\ref{sec:fast-mixing-subexp}, we will take $\alpha = \gamma$ (depending on the $\delta$ for which Definition~\ref{def:supercritical-Bernoulli-perc} holds); however, we write this section for general $\alpha,\gamma$ in case there are better choices in specific situations.

    \begin{proposition}\label{prop:disagreement-probability}
        Suppose $(X_t)_{t},(Y_t)_t$ are FK dynamics chains, coupled by the grand coupling, and such that $X_0(E(B_R)) = Y_0(E(B_R))$. There exists $\epsilon_0(\Delta)>0$ such that for all $\epsilon<\epsilon_0$ and all $t\le \epsilon R/ \ell^{\alpha+\gamma}$,
        \begin{align*}
            \mathbb P\Big(X_t(B_{R/2})\ne Y_t(B_{R/2})\, ,\, \bigcap\nolimits_{s=0}^t \{X_s,Y_s\in \mathcal E_{\ell,\alpha,\gamma}\}\Big) \le C|\partial B_{R}| \exp( - R/(2\ell^\alpha))\,.
        \end{align*}
    \end{proposition}

    \begin{proof}
        By Proposition~\ref{prop:disagreement-region}, the probability of the event under consideration is at most that of the event  $\{\mathcal D_t \cap B_{R/2} \ne \emptyset\}\cap \mathcal E_{\ell,t}$ where for ease of notation,  
        $\mathcal E_{\ell,t} : = \bigcap_{s=0}^t \{X_s,Y_s\in \mathcal E_{\ell,\alpha,\gamma}\}\,.$
        We construct a \emph{witness} to that pair of events as follows. 

        Let $f_0$ be the edge whose clock rang at time $t_{j_0}:= \inf\{s: B_{R/2} \cap \mathcal D_s \ne \emptyset\}$ (note that $f_0\in E(B_{R/2})$). Let $w_0$ be the vertex in $\partial \cD_{t_{j_0}^-}$ for which $f_0\in \mathsf{CE}_{w_0}^{\ne 1}(Z_0)$ for  $Z_0 \in \{X_{t_{j_0}^-},Y_{t_{j_0}^-}\}$ (if there are multiple choices for $w_0$ or $Z$, we choose arbitrarily).          
        Given $(f_j,w_j,Z_j)_{j<i}$, we construct the witness iteratively as follows: 
        \begin{itemize}
            \item Let $f_i$ be the first edge incident to $w_{i-1}$ to be included in $(\cD_{s})_{s \ge 0}$; i.e., $f_i$'s clock rang at time
            \begin{align*}
                t_{j_i}:= \inf \{s: w_{i-1}\in \cD_{s}\}
            \end{align*}
            \item Let $w_i$ be the vertex in $\partial \cD_{t_{j_{i}}^-}$ and $Z_i\in \{X_{t_{j_i}^-},Y_{t_{j_i}^-}\}$ for which $f_i \in \mathsf{CE}_{w_i}^{\ne 1}(Z_i)$. 
        \end{itemize}
        (Again, ambiguities are resolved arbitrarily.)
        Under this construction, the event $X_t(B_{R/2})\ne Y_t(B_{R/2})$ implies the existence of a witness $(f_i,w_i,Z_i)_{i=0}^K$ for some $K$ such that 
        \begin{enumerate}
            \item $f_0\in E(B_{R/2})$ and $w_K \in \partial B_R$;
            \item $f_i\in \mathsf{CE}_{w_i}^{\ne 1}(Z_i)$ for all $i$;
            \item $w_{i-1} \in f_{i}$ for all $i$;
            \item the clock ring at time $t_{j_i}$ is at edge $f_i$.
        \end{enumerate}
        Notice that this construction is done \emph{backwards in time}, i.e., $t_{j_0}> t_{j_1} >\cdots >t_{j_K}$. See Figure~\ref{fig:witness-construction} for a depiction.

            \begin{figure}
        \centering
        \begin{tikzpicture}
         \node at (-5.5,0) {
        \includegraphics[scale=0.22]{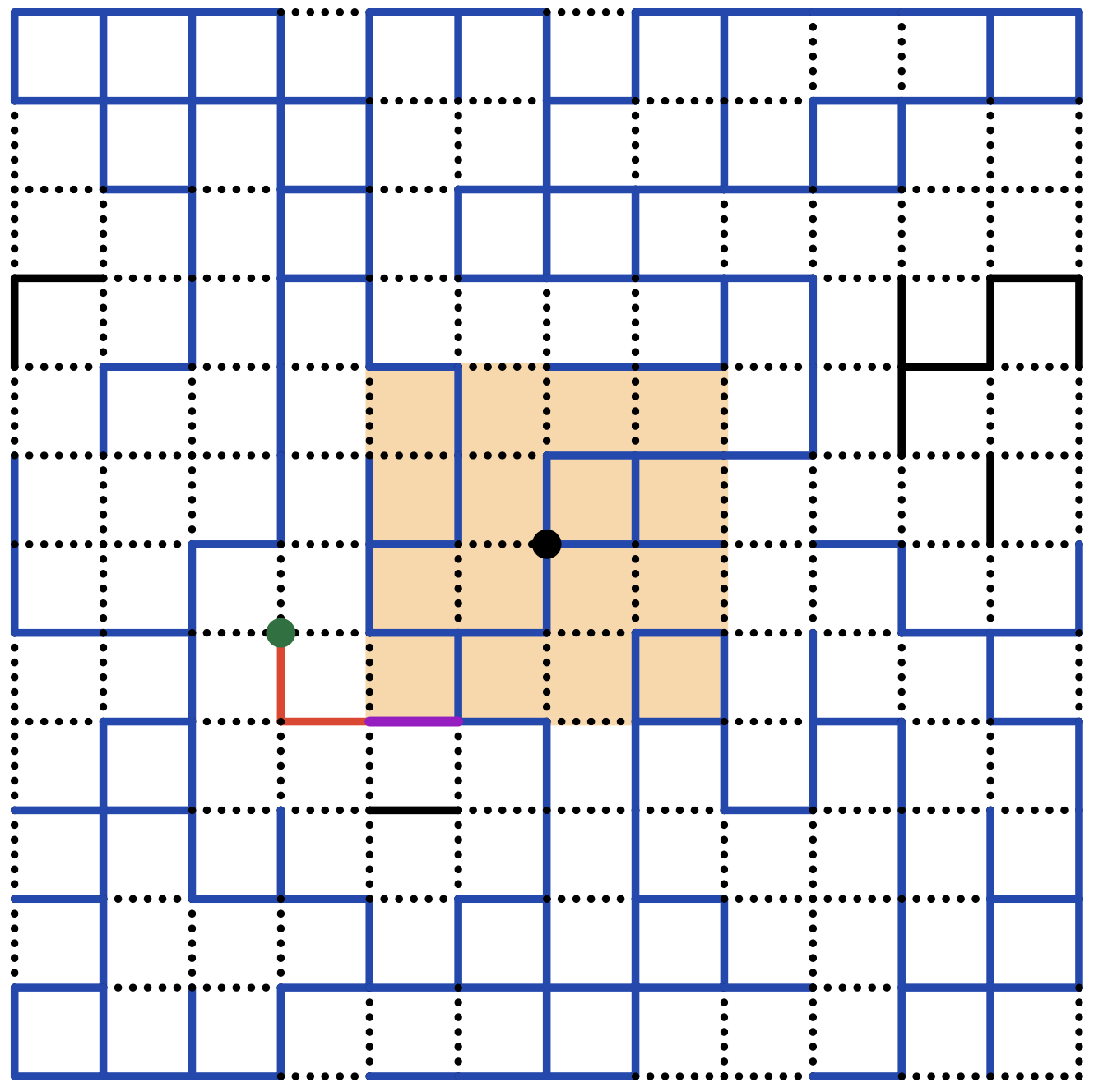}};
        \node[font = \tiny, color = green!50!black] at (-6.95,-.55) {$w_K$};
        \node[font=\tiny, color = blue!50!red] at (-6.11,-.98) {$f_K$};        
         \node at (0,0) {
        \includegraphics[scale=0.22]{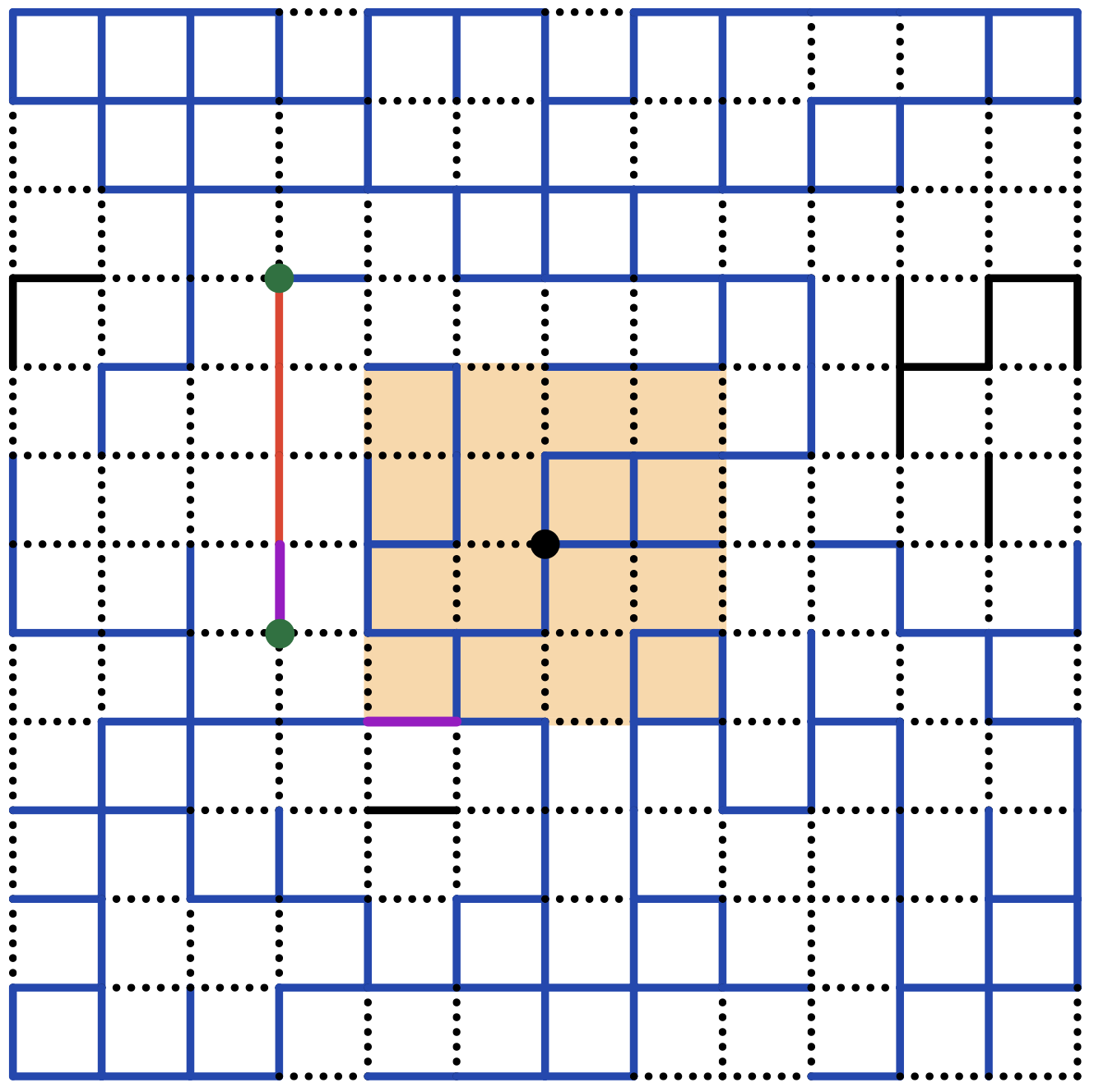}};
         \node[font = \tiny, color = green!50!black] at (-1.45,-.55) {$w_K$};
          \node[font = \tiny, color = green!50!black] at (-1.15,1.45) {$w_{K\text{-}1}$};
        \node[font=\tiny, color = blue!50!red] at (-0.61,-.98) {$f_K$};
        \node[font=\tiny, color = blue!50!red] at (-.99,-0.14) {$f_{K\text{-}1}$};
         \node at (5.5,0) {
        \includegraphics[scale=0.22]{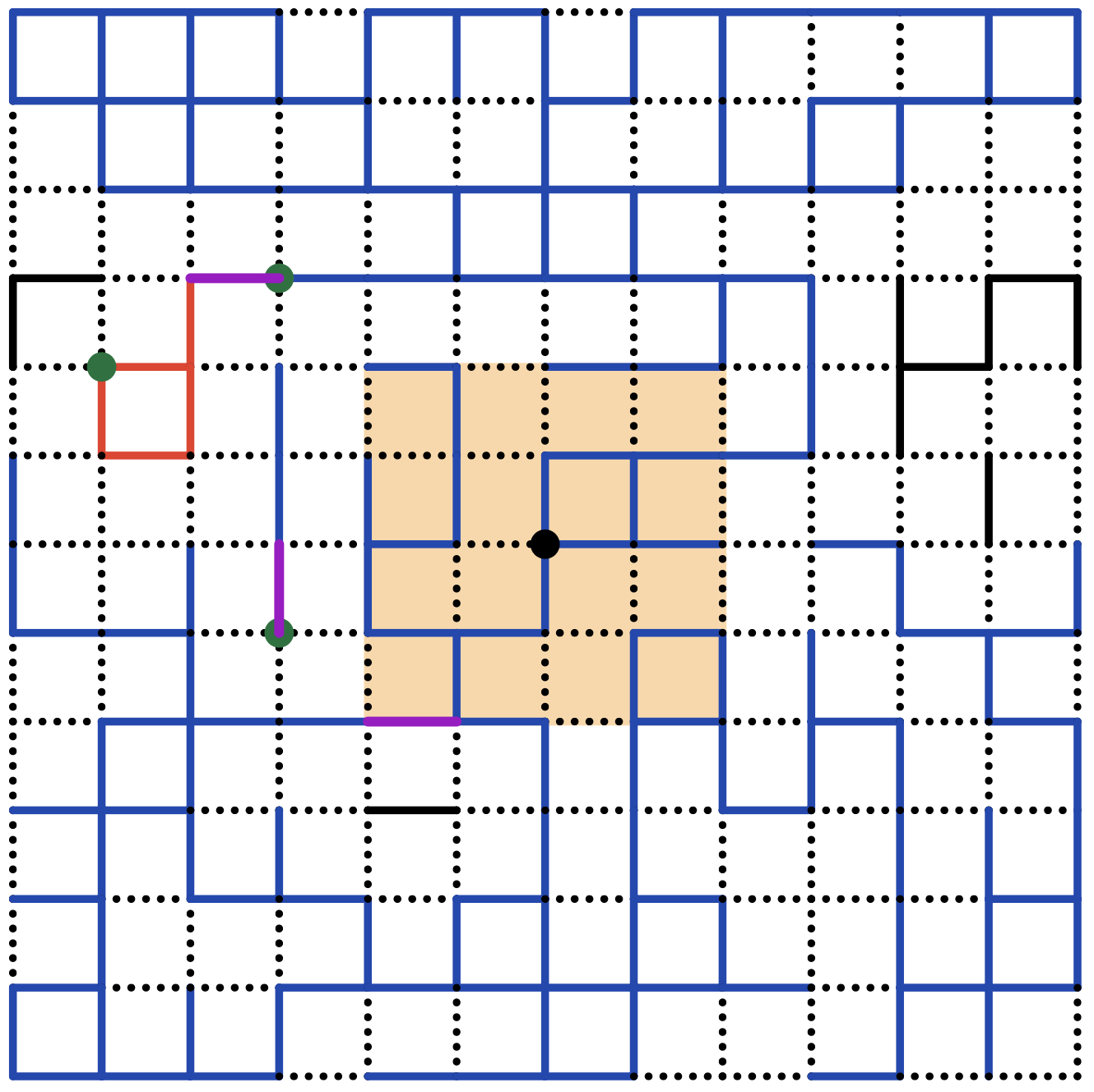}};
         \node[font = \tiny, color = green!50!black] at (4.05,-.55) {$w_K$};
          \node[font = \tiny, color = green!50!black] at (4.35,1.45) {$w_{K\text{-}1}$};
          \node[font = \tiny, color = green!50!black] at (3.3,1.05) {$w_{K\text{-}2}$};
        \node[font=\tiny, color = blue!50!red] at (4.89,-.98) {$f_K$};
        \node[font=\tiny, color = blue!50!red] at (4.51,-0.14) {$f_{K\text{-}1}$};
        \node[font=\tiny, color = blue!50!red] at (4.07,1.09) {$f_{K\text{-}2}$};
        \end{tikzpicture}
        \caption{Three steps of the construction of the witness are shown. The ball $B_{R/2}$ is the highlighted region. For each $i$, the edges of the finite-connectivity cluster of $w_i$ (green) to $f_i$ (purple) in $Z_{i}$ (blue and black edges) are depicted in red. Note that  the configuration changes from left to right, depicting the evolution of the dynamics (backwards in time).}
        \label{fig:witness-construction}
    \end{figure}
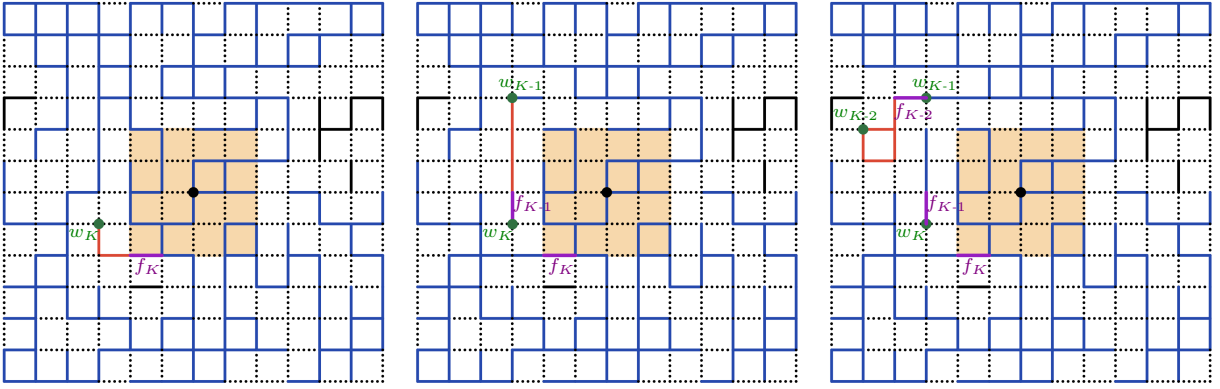

        We will show that the probability that there exists such a witness and the event $\mathcal E_{\ell,t}$ occurs satisfies the claimed bound. On the event $\mathcal E_{\ell,t}$, it must be the case that $K \ge R/(2\ell^{\alpha})$ since the distances between $f_{i}$ and $f_{i-1}$ are bounded by $\ell^{\alpha}$. Furthermore, for any witness $(f_i,w_i,Z_i)_{i=0}^K$, there is a projection, which we also call a witness, $(f_i,w_i,L_i)_{i=0}^K$ where the label $L_i\in \{X,Y\}$ indicates whether $Z_i = X_{t_{j_i}}^-$ or $Z_i = Y_{t_{j_i}^-}$.   

        The total number of clock rings in $G$ in $[0,t]$ has a $\text{Poisson}(t|E(G)|)$ distribution; let $M$ denote this quantity. Note that by standard Poisson concentration, we have 
        $
            \mathbb P(M \ge  4 t |E(G)|) \le \exp( - t |E(G)|)\,.
        $        
        Let us work on the event that $M\le 4t|E(G)|$. Let ${\mathcal T}= \{t_1,...,t_M\}$ be the sequence of clock ring times in $|E(G)|$. 
        We start by bounding $\mathbb P(  X_t(B_{R/2}) \ne Y_t(B_{R/2}),\mathcal E_{\ell,t})$ by 
        \begin{align*}
            e^{ - t|E(G)|} + \max_{M\le 4t|E(G)|}\max_{\mathcal T} \! \sum_{K \ge R/(2\ell^{\alpha})} 2^K \max_{L\in \{X,Y\}^K} \mathbb P\big(\exists (f_i,w_i)_{i=0}^{K} :  (f_i,w_i,L_i)_{i=0}^{K} \text{ is witness}, \cE_{\ell,t} \! \mid \! \mathcal T)\,,
        \end{align*}
         where the conditioning on ${\mathcal T}$ indicates conditioning on the clock ring times in $E(G)$ in $[0,t]$ being ${\mathcal T}$ (but importantly not revealing their location yet).

        Fix any $M\le 4t|E(G)|$ and any realization $\mathcal T$ and consider the probability on the right. 
        For any subset of times $J = \{j_K,...,j_0\}$, given ${\mathcal T}$, we denote by $W_{J,L}$ the event that there exist $(f_i,w_i)_{i=0}^{K}$ such that the triple $(f_i,w_i,L_i)_{i=0}^{K}$ is a witness and the clock ring at $f_i$ is at time $t_{j_{i}}$ for $i=0,\dots,K$. 
        There being $\binom{M}{K}$ choices of $J$, for any $L \in \{X,Y\}^K$:
        \begin{align*}
        \mathbb P\big(\exists (f_i,w_i)_{i=0}^{K} :  (f_i,w_i,L_i)_{i=0}^{K} \text{ is witness}, \cE_{\ell,t}\mid \mathcal T) \le 
            \max_{J}
            \binom{M}{K}  \mathbb P\big({W}_{J,L}, \cE_{\ell,t} \mid {\mathcal T}\big)\,.
        \end{align*}

        For $s\ge 0$, let $\mathcal F_{s}$ be the $\sigma$-algebra generated by the grand coupling up to time $s$. We will now sequentially
        condition on $\mathcal F_{t_{j_i}^-}$, and enumerate over the possible choices for the edge $f_i$ (of which there will be at most $\Delta \ell^\gamma$ per item (2) of Definition~\ref{def:E-l-t}), and then ask that the clock ring at time $t_{j_i}$ be at $f_i$. 

        More precisely, if for an edge $g$, ${\mathcal A}_l^{g}$ is the event that in the witness $f_l = g$, then we have for every $i< K$, 
        \begin{align}\label{eq:step-i-witness-bound}
            &\max_{g_{i+1},...,g_K}  \mathbb P\big(W_{J,L},\cE_{\ell,t} \mid {\mathcal T}, \mathcal F_{t_{j_i}^-}, ({\mathcal A}_{l}^{g_l})_{l>i}\big) \nonumber \\
            & \quad\le \max_{g_{i+1},...,g_K} \sum_{w_i\in g_{i+1}} \sum_{g_{i}\in \mathsf{CE}_{w_i}^{\ne 1}(Z_i)} \mathbb P( {\mathcal A}_i^{g_i} \mid {\mathcal T}, \cF_{t_{j_i}^-}, ({\mathcal A}_l^{g_l})_{l>i}) \mathbb P(W_{J,L}, \cE_{\ell,t} \mid {\mathcal T}, \cF_{t_{j_i}^-}, ({\mathcal A}_l^{g_l})_{l\ge i}) \nonumber \\
            & \quad \le \frac{2\Delta\ell^{\gamma}}{|E(G)|} \cdot \max_{g_i,...,g_K} \mathbb P\big(W_{J,L},\cE_{\ell,t} \mid {\mathcal T}, \mathcal F_{t_{j_i}^-}, ({\mathcal A}_{l}^{g_l})_{l\ge i}\big)\,.
        \end{align}
        In the above, and in what follows in this proof, when conditioning on a $\sigma$-algebra, we mean the inequalities to hold almost surely, i.e., for almost surely every realization of the random variables generating the $\sigma$-algebra.
        Here, the first inequality is a union bound over the potential choices of $w_i$ and then the potential choices of $g_i$ in the witness.
        (Notice that conditional on $\mathcal F_{t_{j_K}^-}$, the configuration $Z_K$ can be read-off from its label $L_K$.) For the second inequality, we used the definition of $\mathcal E_{\ell,t}$ to bound the number of summands in the second line by $2\Delta\ell^\gamma$, and we used the fact that conditionally on ${\mathcal T}$, the locations of the clock rings are independent and uniform on $E(G)$, so given ${\mathcal T}$ and $\mathcal F_{t_{j_i}^-}$, the probability that the clock ring at time $t_{j_i}$ is at a fixed edge $g_i$ is $1/|E(G)|$.
        
        We can condition the right-hand side of~\eqref{eq:step-i-witness-bound} further on $\mathcal F_{t_{j_{i-1}}^-}$ (recalling that $t_{j_{i}}<t_{j_{i-1}}$) to arrive at the following relation between the probabilities for index $i$ and index $i-1$:
        \begin{align*}
            \max_{g_{i+1},...,g_K} \mathbb P\big(W_{J,L},\cE_{\ell,t} \mid {\mathcal T}, \mathcal F_{t_{j_i}^-}, ({\mathcal A}_{l}^{g_l})_{l>i}\big) \le \frac{2\Delta\ell^{\gamma}}{|E(G)|} \cdot \max_{g_i,...,g_K} \mathbb P\big(W_{J,L},\cE_{\ell,t} \mid {\mathcal T}, \mathcal F_{t_{j_{i-1}}^-}, ({\mathcal A}_{l}^{g_l})_{l\ge i}\big).
        \end{align*}
        The same inequality holds for $i=K$ with an extra multiplicative factor of $|\partial B_R|$ for the initial choice of $w_K$. 
        Iterating this over all $i$, we arrive at the following bound on $\mathbb P(X_t(B_{R/2})\ne Y_t(B_{R/2}),\mathcal E_{\ell,t})$:
        \begin{align*}
            |\partial B_R| \!\sum_{K\ge R/(2\ell^{\alpha})} \!\! & 2^K \!\! \max_{M\le 4t|E(G)|} \binom{M}{K} \Big(\frac{2\Delta\ell^{\gamma}}{|E(G)|}\Big)^K + e^{-t|E(G)|}  \le |\partial B_R| \! \sum_{K\ge R/(2\ell^{\alpha})} \! \Big(\frac{8 \Delta \ell^{\gamma} e t}{K}\Big)^K + e^{- t|E(G)|}\,.
        \end{align*}
        At this stage, we see that if $t \le \epsilon R/(16\Delta e\ell^{\alpha+ \gamma})$ for $\epsilon<1/e$, then this is at most 
        \begin{align*}
            |\partial B_R| \epsilon^{R/(2\ell^{\alpha})} \sum_{j \ge 0}\epsilon^j + e^{-t|E(G)|} \le C |\partial B_R| e^{-R/(2\ell^{\alpha})}  + e^{ - t|E(G)|}\,.
        \end{align*}
        The term $e^{ - t|E(G)|}$ is absorbed since we have $R\le \diam(G)\le E(G)$ trivially. 
    \end{proof}

    \begin{remark}\label{rem:disagreement-percolation-proof-difficulties}
        Beyond the low-temperature construction of the disagreement region, we point out a subtlety in the above that may have gone unnoticed.  
        In disagreement percolation bounds for high-temperature FK dynamics (e.g., in~\cite{BS}), the typical analog of $\mathcal E_{\ell,\alpha,\gamma}$ is simply that the largest cluster in $B_R$ has volume at most $\ell$. When counting the number of possible witnesses, one takes $|B_\ell(w_i)|$ as a worst-case bound for the number of locations of the next disagreement along the chain in the witness. If the volume growth is stretched exponential, however, this does not work. The careful conditioning in~\eqref{eq:step-i-witness-bound} was essential to only count those edges $\mathsf{CE}_{w_i}^{\ne 1}$ that could be vulnerable to be the next edge in the witness, rather than the entire volume of a ball, keeping the count to a polynomial even in the presence of exponential volume growth. 
    \end{remark}

\section{Fast mixing of FK dynamics on graphs of sub-exponential growth}\label{sec:fast-mixing-subexp}

With the bound on the speed of information propagation at low temperatures from the previous section on hand, we proceed to establish Theorem~\ref{thm:subexp-mixing}. The event $\mathcal E_{\ell,\alpha,\gamma}$ from Definition~\ref{def:E-l-t} was crucial to controlling the speed of disagreement propagation, and our first aim (Section~\ref{subsec:burn-in-subexp}) is to establish that after an $O(1)$ time burn-in period, $\mathcal E_{\ell, \alpha,\gamma}$ holds for a further $O(1)$ amount of time. Then in Section~\ref{subsec:proof-subexp-mixing}, we will build a space-time recursion to establish the desired mixing time bound. 

\subsection{Dominating edge-percolation after a burn-in period}
    We start with a simple estimate showing that after an $O(1)$ (continuous-time) burn-in period, the FK dynamics started from any initialization stochastically dominates the edge-percolation at a parameter arbitrarily close to $\ps = \frac{p}{q(1-p)+p}$. This will be crucial to many of our arguments throughout the paper. 
    
    \begin{lemma}\label{lem:domination-after-burn-in}
    Fix $p,q,\delta$. There exists $T_0(\delta)$ such that $X_{t}^\zero$ stochastically dominates $\pi_{\ps -\delta,1}$ for all $t\ge T_0$. This also holds conditioned on $\mathscr{T}_{[T_0,\infty)}$ (the $\sigma$-algebra generated by the clock rings from time $T_0$ on).  
    \end{lemma}

    \begin{proof}
    Consider any edge $e$. Uniformly over all the possible randomness (Poisson clocks and uniform random variables) on edges of $E(G)\setminus \{e\}$, as well as all clock rings at $e$ after time $T_0$ (a to be determined constant depending only on $\delta$), on the event that the clock at $e$ has rung by time $t$, its distribution stochastically dominates $\text{Ber}(\ps)$ per~\eqref{eq:FK dynamics-update-rule}. The result follows if we let $T_0$ be large enough that the probability that the clock at $e$ has not rung by time $T_0$ is less than $\delta$ (this is independent of the clock rings at $e$ after $T_0$). 
    \end{proof}

    We consistently use the notation 
    $\widetilde \omega$ for independent edge-percolation processes on $E(G)$ with parameter $\tilde p$, i.e., $\widetilde \omega \sim \pi_{G,\tilde p,1}$. For ease of notation, we simply write $\widetilde \pi_G$ for the law of $\widetilde \omega$ on $G$. 

\subsection{Burnt-in FK dynamics are in {$\mathcal E_\ell$}}\label{subsec:burn-in-subexp}
Our next aim is to show that burnt-in FK dynamics configurations are in $\mathcal E_{\ell,\alpha,\gamma}$ with high probability. Recall the main assumption on our underlying graphs for Theorem~\ref{thm:subexp-mixing}, that the independent percolation on them has a strongly supercritical phase: Definition~\ref{def:supercritical-Bernoulli-perc}.

 Recall the event $\mathcal E_{\ell,\alpha,\gamma}$ from Definition~\ref{def:E-l-t} that governed the size of regions through which disagreements could possibly spread. In what follows, we fix $\delta>0$ given to us by Definition~\ref{def:supercritical-Bernoulli-perc}, and let 
 \begin{align}\label{eq:alpha-gamma-choices}
     \mathcal E_\ell := \mathcal E_{\ell,(1+\delta)/\delta,(1+\delta)/\delta}\,, \qquad \text{i.e., $\alpha = \gamma = (1+\delta)/\delta$}\,.
 \end{align}
 For ease of notation, we use $\alpha = (1+\delta)/\delta$ in the below. We continue to imagine a fixed vertex $o\in V(G)$, and fixed $R$ large, but independent of the graph, and let $B_R = B_R(o)$; events and sets from the previous section are all defined with respect to this ball. Finally, $\ell,R$ can be assumed to be sufficiently large (depending on $\eta,\delta)$. Our main result in this subsection is the following.

\begin{proposition}\label{prop:burnt-in-is-good}
    Suppose $G$ satisfies Definition~\ref{def:supercritical-Bernoulli-perc} and has $\eta$-stretched exponential growth for $\eta$ less than some $\eta_0(\delta)$. There exists $T_0(\delta,\eta,q)$ such that for every initial configuration $\omega_0$ and every $t\ge T_0$, 
    \begin{align*}
        \mathbb P\Big(\bigcup_{s = t}^{2t} \{X_s^{\omega_0} \notin \mathcal E_{\ell}\}\Big) \le C t e^{R^\eta} e^{- \ell}\,.
    \end{align*}
\end{proposition}

\smallskip
\noindent \textbf{Proof strategy.} Establishing Proposition~\ref{prop:burnt-in-is-good} is quite a bit more involved than it would be in a high-temperature setting (i.e., for small $p$). This is due both to the delicate graph-theoretic aspects of the event $\mathcal E_\ell$, its non-monotonicity, and its non-locality. In particular, the last one means that in the time interval $[t,2t]$ the number of edge updates that could hypothetically cause the FK dynamics to leave $\mathcal E_\ell$ are $t|E(G)|$, rather than $t|B_R|$; a naive union bound over this number would fail. 
We outline our strategy as follows:
\begin{enumerate}
    \item In Definition~\ref{def:G-l} below, we define a proxy event $\mathcal G_\ell$ which is monotone, ensures that $\mathcal E_\ell$ occurs, and is more ``local" than $\mathcal E_\ell$. The relationship to $\mathcal E_\ell$ under minimal assumptions on $G$ is in Lemma~\ref{lem:E-ell-G-ell}.  
    \item Using Definition~\ref{def:supercritical-Bernoulli-perc}, we will show in Lemma~\ref{lem:G-l-high-probability} that $\mathcal G_\ell$ holds with high probability for independent edge-percolation on $G$ with a large enough parameter $\tilde p$. Since $\mathcal G_\ell$ is monotone, we can translate this bound to the FK dynamics after an $O(1)$ burn-in time per Lemma~\ref{lem:domination-after-burn-in}. 
    \item We then perform a careful ``union bound" over the update times between $s\in [t,2t]$. This could be a problem since there are order $|E(G)|$ many updates in this time interval, whereas the probability of $\mathcal G_\ell^c$ is only exponentially small in the local quantity $\ell$. Importantly, though, we use that the event $\mathcal G_\ell$ is ``localized" to reason that far away edge updates are unlikely to induce a change in $\mathcal G_\ell^c$, in a summable manner. This argument is executed via Lemma~\ref{lem:far-away-update-influence} in the proof of Proposition~\ref{prop:burnt-in-is-good}. 
\end{enumerate}

Let us begin by defining the proxy event $\mathcal G_\ell$ and its variant $\bar \cG_{m}$ for $m\ge \ell$. 

\begin{definition}\label{def:G-l}
    Define the event $\mathcal G_\ell$ as the event that there does not exist a connected set $A$ intersecting $B_R$ having $\ell^{\alpha} \le |A|\le n/2$, and an edge $e\in \partial_e A$ such that $\omega(\partial_e A \setminus \{e\}) \equiv 0$. 

    Define the event $\bar \cG_{m}$ as the event that there does not exist a connected set $A$ intersecting $B_R$, of size $m^{\alpha}\le |A|\le n/2$, and a pair of edges $e_1,e_2\in \partial_e A$ such that $\omega(\partial_e A \setminus \{e_1,e_2\}) \equiv 0$. 
\end{definition}
Notice that, unlike $\mathcal E_\ell^c$, the event $\mathcal G_\ell^c$ is a monotone decreasing event, since it is the union (over $A,e$) of decreasing events. Similarly, the event $\bar \cG^c_{m}$ is a decreasing event.

The following graph theoretic lemma demonstrates that $\mathcal G_\ell$ controls $\mathcal E_\ell$. It is important here to relate the number of cut-edges in the carefully constructed set of vulnerable edges in the disagreement percolation $\mathsf{CE}_v^{\ne 1}$ to an easier quantity: the volume of a set of size smaller than $n/2$ with closed boundary. 

\begin{lemma}\label{lem:E-ell-G-ell}
    The event $\mathcal E_\ell^c$ is a subset of the event $\mathcal 
    G_\ell^c$. 
\end{lemma}
\begin{proof}
    On the complement of item (1) in Definition~\ref{def:E-l-t}, there exists $v\in B_R$ and $e\in E(B_R)$ such that $\diam(\mathcal C_v^{\ne 1}(\omega\setminus e))\ge \ell^{\alpha}$. Letting $A = \mathcal C_v^{\ne 1}(\omega\setminus e)$, we notice that $|A|\ge \diam(A) \ge \ell^{\alpha}$. At the same time, $|A|\le n/2$ since if $|A|\ge n/2$ then $\mathcal C_v(\omega\setminus e) = \mathcal C_1(\omega\setminus e)$ and $\mathcal C_v^{\ne 1}(\omega \setminus e)$ would be the trivial $\{v\}$. Finally $A$ intersects $B_R$ since it contains $v\in B_R$, and under $\omega$, all of $\partial_e A\setminus \{e\}$ must be closed since $A$ is a connected component of $\omega \setminus e$. 

    We now show that the complement of item (2) in Definition~\ref{def:E-l-t} also implies $\mathcal G_\ell^c$. The essence of the argument is that $\mathsf{CE}_v^{\ne 1}$ lower bounds the size of $\mathcal C_v^{\ne 1}(\omega\setminus e)$ for some $e$, but a little care must be taken due to the definition of $\mathsf{CE}_v^{\ne 1}$. 
    We begin by constructing a tree from the set of all $e\in \text{Cutedge}(\omega)$ that are incident to $\mathcal C_v(\omega\setminus e)$; note this is a larger set than those $e\in \text{Cutedge}(\omega)$ that have $e\sim \mathcal C_v^{\ne 1}(\omega\setminus e)$, but we will subsequently restrict to this latter set. 
     Suppose $\mathbf e$ is the set of all edges $e$ in $\text{Cutedge}(\omega)$ such that $e\sim \mathcal C_v(\omega\setminus e)$.  
     We iteratively associate a tree $T_v= T_v(\omega)$ to $\{v\} \cup \mathbf e$ in the following natural way. 
         \begin{enumerate}
         \item Identify the root of the tree with the vertex $v$; 
         \item All cut-edges in $\mathbf{e}$ are associated to descendants of the root. 
        \item For a vertex $w$ of the tree (asssociated to a cut-edge $e_w$ in $\mathbf{e}$), a cut-edge $f\in \mathbf{e}$ is associated to a descendant of $w$ if and only if it is disconnected from $v$ by $e_w$ in $\omega$. 
         \end{enumerate}
         This process uniquely determines the tree since the children of $w_{i,j}$ are those descendants of $w_{i,j}$ that are not descendants of any of $w_{i,j}$'s other descendants. The fact that all of these are cut-edges also ensures that no cycles arise in the construction. Notice that the leaves of $T_v$ are exactly $\partial_e \mathcal C_v(\omega)$. 

         We next claim that the edges in
         $$\mathbf{e}' = \{e\in \mathbf{e}: |\mathcal C_v(\omega\setminus e)|\le n/2\}\,,$$
         are a sub-tree of $T_v$. By definition, an edge $e$ can only be in $\mathbf{e}$ but not in $\mathbf{e}'$ if the component $|\mathcal C_v(\omega\setminus e)| >n/2$. If this occurs for an edge $e$ associated to vertex $w$ in the tree, any edge $f$ associated to a descendant of $w$ will also have $|\mathcal C_v(\omega\setminus f)| > n/2$ and not be in $\mathbf{e}'$ since the difference in the component structures of $\omega\setminus e$ and $\omega\setminus f$ is that the latter has a larger $\mathcal C_v$ and one other component is correspondingly smaller. Therefore, the event that an edge is in $\mathbf{e}$ but not in $\mathbf{e}'$ is a decreasing event on the tree $T_v$. As such the restriction of $T_v$ to $\{v\}\cup \mathbf{e}'$ is itself a tree, which we can call $T_v'$.

         Select an arbitrary vertex in $\partial T_v'$, i.e., its descendants are all in $T_v$ but not in $T_v'$, and call its corresponding cut-edge $e_\star$; also define $\omega_\star = \omega \setminus e_\star$. Note that the tree $T_v(\omega_\star)$ is exactly $T_v(\omega)\setminus S_\star$ where $S_\star$ are all descendants of $e_\star$. 
         If we let $A=\mathcal C_v(\omega_\star)$, then evidently $\omega(\partial_e A\setminus e_\star) \equiv 0$ since $\omega_\star(\partial_e A) \equiv 0$. Also, $|A|\le n/2$ since otherwise $e_\star$ would not belong to $T_v'$. All edges of $\mathbf{e}'$ belong to $T_v(\omega_\star)$ so they are all incident to $A$; therefore $$\Delta |A| \ge |\mathbf{e}'|\ge |\mathsf{CE}_v^{\ne 1}| \ge \Delta \ell^{\alpha}\,,$$
         using the fact that $\mathbf{e}'\supset \mathsf{CE}_v^{\ne 1}$. Lastly, $A$ contains $v\in B_R$ since $v\in T_v(\omega_\star)$. Thus, $A$ violates $\mathcal G_\ell$. 
\end{proof}

The following lemma relates the vulnerability of a configuration to leaving $\mathcal G_\ell$ by means of an edge-update at distance $m$ from $B_R$ to the event $\bar \cG_{m}^c$, allowing us to control the probability that far away updates (of which there are many in order-one continuous time) induce the dynamics to leave $\mathcal G_\ell$. 

\begin{lemma}\label{lem:far-away-update-influence}
    For any $m\ge \ell$, in order for an edge $e\notin B_{R+m^{\alpha}}$ to be pivotal to $\omega\in\mathcal G_\ell$, i.e., for $\omega\oplus \{e\} \in \mathcal G_\ell^c$ while $\omega \in \mathcal G_\ell$, it must be that $\omega \in \bar \cG_{m}^c$.  
\end{lemma}

\begin{proof}
    Suppose $\omega \in \mathcal G_\ell$, $e\notin E(B_{R+m^{\alpha}})$ such that $e\notin \omega$ and $\omega\cup \{e\} \in \mathcal G_\ell^c$ or $e\in \omega$ and $\omega \setminus \{e\} \in \mathcal G_\ell^c$.  

    Since $\cG_\ell$ is an increasing event, the first case is not possible. 
     In the second case, the removal of an edge $e\in \omega$ outside $B_{R+m^{\alpha}}$ causes a component $A$ to become part of $\mathcal G_\ell^c$. Call $A$ the set in $\omega \setminus e$ that violates $\mathcal G_\ell$. Then the edge $e$ must be in $\partial_e A$,  meaning that the set $A$ is a set of size at most $n/2$, intersecting $B_R$ and with one other edge $e_1 \in \partial_e A$ such that $\omega(\partial_e A\setminus \{e,e_1\})\equiv 0$. 
    Since the distance of $e$ to $B_R$ is at least $m^{\alpha}$, it must be that $|A|\ge \diam(A)\ge m^{\alpha}$.  
\end{proof}

We now turn to the probabilistic estimates. The following lemma utilizes Definition~\ref{def:supercritical-Bernoulli-perc} to bound the probability of $\mathcal G_\ell^c$ (as well as of $\bar \cG_{m}^c$) under the independent edge-percolation measure $\widetilde{\pi} = \pi_{G,\tilde p,1}$. 

\begin{lemma}\label{lem:G-l-high-probability}
    If $G$ satisfies Definition~\ref{def:supercritical-Bernoulli-perc} and has $\eta$-stretched exponential volume growth for $\eta<\eta_0(\delta)$, for all $\tilde p$ sufficiently large, 
    \begin{align*}
        \widetilde\pi(\widetilde\omega \notin \mathcal G_\ell)\le C |B_R| \exp( - \ell)\,,
    \end{align*}
    for some $C(\tilde p, \eta,\delta)$. Similarly, 
    \begin{align*}
         \widetilde\pi( \widetilde\omega \notin \bar \cG_{m}) \le C |B_R| \exp(- m)\,.
    \end{align*}
\end{lemma}

\begin{proof}
         The lemma is almost a union bound together with  Definition~\ref{def:supercritical-Bernoulli-perc}, the only distinction being that we allow one or two of the edges in $\partial_e A$ to be open. Fix a vertex $v$, and for every configuration $\widetilde \omega$ in $\mathcal G_\ell^c$ by means of a set $A= A_v(\tilde \omega)$ containing $v$, such that $\widetilde \omega(\partial_e A\setminus e)\equiv 0$, let $\phi_e(\widetilde \omega)= \omega\setminus e$. Evidently,  
        \begin{align*}
            \frac{\widetilde \pi(\widetilde \omega)}{\widetilde \pi(\phi_e (\widetilde \omega))} \le \frac{\tilde p}{1-\tilde p}\,.
        \end{align*}
        For every $\widetilde \omega$, the configuration $\phi_e(\widetilde \omega)$ has a set $A$ intersecting $v$ such that $\ell^{\alpha} \le |A|\le n/2$ and such that $(\phi_e(\widetilde\omega))(\partial_e A) \equiv 0$, the probability of which is governed by Definition~\ref{def:supercritical-Bernoulli-perc}. Furthermore, if this $A$ has size exactly $r$, the set of all pre-images of a single $\widetilde \omega$ under the map $\phi_e$ is bounded by the set of all $e$ such that $d(v,e)\le r$, which is at most $e^{r^{\eta}}$ per the $\eta$-stretched exponential growth assumption. Putting this all together, we get 
        \begin{align}\label{eq:G_m-Peierls}
           \widetilde \pi(\widetilde \omega \notin \mathcal G_\ell) \le  \sum_{v\in B_R} \sum_{r = \ell^{\alpha}}^{n/2} \sum_{\tilde \omega \in \cG_{\ell}^c: |A_v(\tilde \omega)| = r} \widetilde \pi(\widetilde \omega) \le \frac{\tilde p}{1-\tilde p} \sum_{v\in B_R}\sum_{r=\ell^{\alpha}}^{n/2} e^{r^{\eta}} e^{ - r^{\delta/(1+\delta)}}\,.
        \end{align}
        As long as $\eta$ is smaller than $\alpha = (1+\delta)/\delta$, the above quantity is at most some constant (depending on $\eta,\delta,\tilde p$) times $|B_R|e^{ - \ell}$. 
        
        The argument for $\bar \cG_{m}$ is essentially identical, with the only differences being that in~\eqref{eq:G_m-Peierls}, the pre-factor $\tilde p/(1-\tilde p)$ is squared, and the number of choices of \emph{two} edges that could be closed contributes a factor of $e^{2r^\eta}$ instead of $e^{r^\eta}$. 
\end{proof}

The last lemma we need towards proving Proposition~\ref{prop:burnt-in-is-good} is one for bounding the number of clock rings in $B_{R+m^\alpha}$. This follows from a standard Poisson tail bound together with a union bound. 

\begin{lemma}\label{lem:Poisson-clock-ring-concentration}
    For a set $A$, let $N_{A}^{[t,2t]}$ be the number of clock rings in $A$ in the time interval $[t,2t]$. We have 
    \begin{align*}
        \mathbb P\Big(\bigcup_{m\ge 1} \{N_{B_{R+m^\alpha}}^{[t,2t]} \ge 4t|B_{R+m^{\alpha}}|\} \Big) \le \exp( - t |B_R|)\,.
    \end{align*}
\end{lemma}

\begin{proof}
    The number of clock rings in a set $A$ in an interval of length $t>0$ is distributed as a Poisson with rate $|A|t$. Therefore, 
    \begin{align*}
        \mathbb P\big(N_{A}^{[t,2t]} \ge 4t|A| \big) \le \exp( - 2t |A|)\,.
    \end{align*}
    By a union bound, we then get 
        \begin{align*}
        \mathbb P\Big(\bigcup_{m} \{N_{B_{R+m^{\alpha}}}^{[t,2t]} \ge 4t|B_{R+m^{\alpha}}|\} \Big) \le \sum_m e^{ - 2 t |B_{R+m^{\alpha}}|}\,.
    \end{align*}
    Using that $|B_{R+m^{\alpha}}|$ is at least $|B_R| + m^{\alpha}$ by the fact that $m$ only ranges until expanding the radius doesn't add any vertices, and using that $t>0$, this sums out to give $C e^{ - 2t |B_R|}$, whence we absorb the constant $C$ by changing the $2$ in the exponent. 
\end{proof}

\begin{proof}[\textbf{\emph{Proof of Proposition~\ref{prop:burnt-in-is-good}}}]
    By Lemma~\ref{lem:E-ell-G-ell} it suffices to bound the probability of $\bigcup_{s=t}^{2t}\{X_s \notin \mathcal G_\ell\}$. 
    Condition on the clock rings between times $t$ and $2t$; this set of clock rings generates a $\sigma$-algebra we denote by $\mathscr{T}_{[t,2t]}$. Let 
    \begin{align*}
        E_{[t,2t]} := \bigcap_{m\ge \ell} \big\{N_{B_{R+m^{\alpha}}}^{[t,2t]}\le 4t|B_{R+m^{\alpha}}|\big\}\,,
    \end{align*}
    measurable with respect to $\mathscr{T}_{[t,2t]}$. 
    Lemma~\ref{lem:Poisson-clock-ring-concentration} showed that $\mathbb P(E_{[t,2t]}^c) \le \exp(- t |B_R|)$. We can then write  
    \begin{align*}
        \mathbb P\Big(\bigcup_{s=t}^{2t} \{X_s \notin \mathcal G_\ell \} \Big)\le \max_{(e_i,s_i)_i\in E_{[t,2t]}} \mathbb P\Big(\bigcup_{i} \{X_{s_i}\notin \cG_\ell\} \mid (e_i,s_i)_{i}\Big) + e^{ - t|B_R|}\,,
    \end{align*}
    where $(e_i,s_i)_i$ denotes the sequence of pairs of edges and corresponding clock rings between times $t$ and $2t$. For ease of notation, let $s_0 = t$ and let $\mathcal G_{\ell,s}^c$ be the event $\{X_s\notin \mathcal G_\ell\}$. We now write the union above as 
    \begin{align*}
        \bigcup_{i\ge 1} \mathcal G_{\ell,s_i}^c \, \subset \, \cG_{\ell,s_0}^c \cup \bigcup_{i\ge 1} (\cG_{\ell,s_{i-1}}\cap \cG_{\ell,{s_i}}^c) \,.
    \end{align*}
    Furthermore, given the clock ring times and locations $(e_i,s_i)_i$, we can let $I_0 = B_{R+\ell^{\alpha}}$ and for $m\ge \ell$, let $I_m$ be the set of $i$'s for which $e_i$ is in $B_{R+(m+1)^{\alpha}}\setminus B_{R+m^{\alpha}}$. 
    Then, 
    \begin{align*}
        \bigcup_{i\ge 1} \mathcal G_{\ell,s_i}^c \, \subset \, \mathcal G_{\ell,s_0}^c \cup \bigcup_{i\in I_0} (\mathcal G_{\ell,s_{i-1}} \cap \mathcal G_{\ell,s_i}^c) \cup \bigcup_{m\ge \ell} \bigcup_{i\in I_m} (\mathcal G_{\ell,s_{i-1}} \cap \mathcal G_{\ell,s_i}^c)\,.
    \end{align*}
    By Lemma~\ref{lem:far-away-update-influence}, for $i\in I_m$ for $m\ge \ell$, we have  
    \begin{align*}
        \big(\mathcal G_{\ell,s_{i-1}} \cap \mathcal G_{\ell,s_i}^c\big) \, \subset \, \{X_{s_{i-1}}  \notin \bar \cG_{m}\} \,.
    \end{align*}
    Using this bound for $i\in I_m$, and the obvious bound $(\mathcal G_{\ell,s_{i-1}} \cap \mathcal G_{\ell,s_i}^c)\subset \mathcal G_{\ell,s_i}^c$ for $i\in I_0$, we obtain 
    \begin{align*}
        \bigcup_{i\ge 1} \mathcal G_{\ell,s_i}^c \, \subset \, \{\mathcal G_{\ell,t}^c\} \cup \bigcup_{i\in I_0} \{ \mathcal G_{\ell,s_i}^c\} \cup  \bigcup_{m\ge \ell} \bigcup_{i\in I_m} \{X_{s_{i-1}}\notin \bar \cG_{m}\} \,.
    \end{align*}
    Taking the probability on either side, conditioning on $\mathscr{T}_{[t,2t]}$ and using a union bound, we get 
    \begin{align*}
        \mathbb P\Big(\bigcup_{s=t}^{2t} \{X_s \notin \mathcal G_\ell \} \Big) & \le \max_{(e_i,s_i)_i\in E_{[t,2t]}} \mathbb P\big(X_t\notin \mathcal G_\ell \mid (e_i,s_i)_i\big) + \max_{(e_i,s_i)_i \in E_{[t,2t]}} \sum_{i\in I_0} \mathbb P(X_{s_i}\notin \mathcal G_{\ell} \mid (e_i,s_i)_i) \\
        & \quad + \max_{(e_i,s_i)_i \in E_{[t,2t]}} \sum_{m\ge \ell} \sum_{i \in I_m} \mathbb P\big(X_{s_{i-1}}\notin \bar \cG_{m} \mid (e_i,s_i)_i\big)  + e^{ - t |B_R|}\,.
    \end{align*}
    By Lemma~\ref{lem:domination-after-burn-in}, conditionally on any $(e_i,s_i)_i \in \mathscr{T}_{[t,2t]}$, the law of $X_{s_i} \succeq \widetilde \omega$ where $\widetilde \omega$ is drawn from a $\text{Ber}(\tilde p)$ distribution, so long as $t\ge T_0$ from that lemma. Since the events $\cG_\ell^c$ and $\bar \cG_{m}^c$ are decreasing events, each of the probabilities above is bounded above by their analogs for $\widetilde \omega$. Finally, the number of summands $|I_k| \le 4t|B_{R+m^{\alpha}}|$ since $(e_i,s_i)_i\in E_{[t,2t]}$. Together, this means  
    \begin{align*}
        \mathbb P\Big(\bigcup_{s=t}^{2t} \{X_s \notin \mathcal G_\ell \} \Big)\le (1+4t|B_{R+\ell^{\alpha}}|)\mathbb P(\widetilde \omega \notin \mathcal G_\ell) + \sum_{m\ge \ell} 4t |B_{R+(m+1)^{\alpha}}| \mathbb P(\widetilde \omega \notin \bar \cG_{m}) + e^{ - t |B_R|}\,.
    \end{align*}
    By Lemma~\ref{lem:G-l-high-probability}, this is at most 
    \begin{align*}
        (1+4t|B_{R+\ell^{\alpha}}|) e^{ - \ell}  + 4C t \sum_{m\ge \ell} |B_{R+(m+1)^{\alpha}}| e^{ - m} + e^{ - t|B_R|}\,.
    \end{align*}
    Using the stretched-exponential volume growth bound $|B_{R+ m^{\alpha}}|\le e^{(R+m^{\alpha})^\eta}\le e^{R^\eta + m^{\alpha\eta}}$, the first two terms above are summable and yield $4t e^{R^{\eta}} e^{ - \ell}$ as long as $\eta$ is small enough depending on $\delta$, and $\ell,R$ are large enough. The additional term $e^{-t|B_R|}$ is absorbed since $t|B_R|\ge R$ for $t\ge 1$. 
    \end{proof}

\subsection{Exponential relaxation to equilibrium after burn-in}\label{subsec:proof-subexp-mixing}
Our aim is to now combine the above ingredients to establish that after a burn-in period that keeps our configuration in the set $\mathcal E_{\ell}$ per Proposition~\ref{prop:burnt-in-is-good}, the disagreement propagation bounds of Section~\ref{sec:low-temp-disagreement-percolation} can be implemented to guarantee exponential relaxation to equilibrium as long as $p$ is sufficiently close to $1$ to kickstart the spacetime recursion.

        \begin{proposition}\label{prop:exp-relaxation-to-equilibrium}
             Fix $q,\Delta,\delta$. There exists $\eta_0(\delta)>0$ and $p_0(q,\Delta,\delta,\eta)<1$ and $C(p,q,\Delta,\delta,\eta)$ such that for every $\eta<\eta_0$ and $p\ge p_0$ we have the following. For any $G$ satisfying Definition~\ref{def:supercritical-Bernoulli-perc} and $\eta$-stretched-exponential volume growth, the FK dynamics satisfies
            \begin{align*}
               \max_{e\in E(G)} \big(\mathbb P(X^{\one}_{s}(e) =1) - \mathbb P(X^{\zero}_{s}(e) =1)\big) \le e^{-  s/C}\,, \qquad \text{for all $s\le (\log n)^2$}\,.
            \end{align*}
        \end{proposition}
        
	\begin{proof} 
        Abusing notation slightly, let $(X_s)_{s\ge 0} = (X_s^\zero)_{s\ge 0}$ and $(Y_s)_{s\ge 0} = (X_s^\one)_{s\ge 0}$. Define 
        \begin{align*}
            \rho(t) : = \max_{e\in E} \mathbb P(X_t(e)\ne Y_t(e))\,,
        \end{align*}
        under the grand coupling (whence the probability is exactly the difference of the probabilities of $e$ taking value $1$). Recall the definition of $\mathcal E_{\ell} = \mathcal E_{\ell,\alpha,\alpha}$ for $\alpha = (1+\delta)/\delta$ from Definition~\ref{def:E-l-t} and~\eqref{eq:alpha-gamma-choices}.  
        Our first aim is to establish the following recurrence relation, 
        \begin{align}\label{eq:wts-recurrence}
            \rho(2t) \le e^{\sqrt{t}}\rho(t)^2 + e^{-t}\,.
        \end{align}
        for all $t\ge T_0$ for a large enough $T_0$, 
        Toward this aim, let  
        \begin{align*}
            A_{t,R,e}= \{ X_t(B_R(e)) \ne Y_t(B_R(e))\}\,.
        \end{align*}
        Then, for any fixed $e\in E(G)$, we have 
        \begin{align}
        \Pr(X_{2t}(e) \neq Y_{2t}(e))  & \le
		\Pr(X_{2t}(e) \neq Y_{2t}(e) \mid {A}_{t,R,e}^c)  \Pr(A_{t,R,e}^c)  \label{eq:triang:1} \\
		&\quad + 	\Pr\Big(X_{2t}(e) \neq Y_{2t}(e), A_{t,R,e}, \bigcap_{s=t}^{2t} \{X_s,Y_s\in \mathcal{E}_{\ell}\}\Big) \label{eq:triang:2} \\
		&\quad + 
		\Pr\Big(\bigcup\nolimits_{s=t}^{2t} \{X_s\notin \mathcal{E}_{\ell}\}\Big)  +  \Pr\Big(\bigcup\nolimits_{s=t}^{2t} \{Y_s\notin \mathcal{E}_{\ell}\}\Big) \label{eq:triang:3}\,.
        \end{align}
        First notice that 
        \begin{align*}
            \Pr(X_{2t}(e) \neq Y_{2t}(e) \mid {A}_{t,R,e}^c) \le \rho(t)\,,
        \end{align*}
        since we can condition on $X_t,Y_t$, w.r.t.\ which $A_{t,R,e}$ is measurable, and use the property of the grand coupling that 
        \begin{align*}
            \max_{\omega_0,\omega_0'} \mathbb P(X_t^{\omega_0}(e)\ne X_t^{\omega_0'}(e)) \le \mathbb P(X_t^{\zero}(e)\ne X_t^{\one}(e))\,.
        \end{align*}
        At the same time, by a union bound over $e \in E(B_R)$, we can bound $\mathbb P(A_{t,R,e}^c)\le |E(B_R)| \rho(t)$. These give the bound on~\eqref{eq:triang:1} of
        \begin{align*}
            \Pr(X_{2t}(e) \neq Y_{2t}(e) \mid {A}_{t,R,e}^c)  \Pr(A_{t,R,e}^c)\le |E(B_R)| \rho(t)^2\,.
        \end{align*}
        The quantity in~\eqref{eq:triang:2} is controlled by Proposition~\ref{prop:disagreement-probability}, whence as long as $t\le \epsilon R/\ell^{2\alpha}$, 
        \begin{align*}
            \Pr\Big(X_{2t}(e) \neq Y_{2t}(e), A_{t,R,e}, \bigcap_{s=t}^{2t} \{X_s,Y_s \in \mathcal{E}_{\ell}\}\Big)\le C |\partial B_R| \exp( - R/(2\ell^{\alpha}))\,.
        \end{align*}
        Finally, we control each of the  terms in~\eqref{eq:triang:3} by Proposition~\ref{prop:burnt-in-is-good} giving
        \begin{align*}
            \mathbb P\Big( \bigcup_{s=t}^{2t} \{X_s \notin \mathcal E_{\ell}\}\Big) + \mathbb P\Big(\bigcup_{s=t}^{2t}\{Y_s \notin \mathcal E_\ell\}\Big) \le  Ct e^{ R^\eta} e^{ - \ell} \,.
        \end{align*}
        as long as $t\ge T_0(\delta,\eta,q)$ and $\eta<\eta_0(\delta)$. 
        Putting the above together, and using the bounds on $|E(B_R)|$ and $|\partial B_R|$ from the fact that $G$ has $\eta$-stretched-exponential volume growth, for all $T_0(\delta,\eta,q)\le t \le \epsilon R/\ell^{2\alpha}$,  
        \begin{align*}
            \rho(2t) \le e^{R^\eta} \rho(t)^2 + C e^{R^\eta}e^{ - R/(2\ell^{\alpha})} + Cte^{R^\eta} e^{ - \ell}\,.
        \end{align*}
        If we make the choices    
        \begin{align*}
            \ell = 2 R^{2\eta}\,,\qquad \text{and} \qquad t=\ell/2= R^{2\eta}\,.
        \end{align*}
        we find that as long as $\eta <\eta_0(\delta)$ and $t,\ell,R$ are sufficiently large (as a function of $\delta,\eta,\epsilon$), we maintain $t\le \epsilon R/\ell^{2\alpha}$ and we can absorb the pre-factors above to obtain the claimed~\eqref{eq:wts-recurrence}. That recurrence will hold for all $t\ge T_0(q,\delta,\eta)$ and as an upper bound, for all $t\le (\log n)^2$ since our arguments are all valid as long as $R\le \diam(G)$ which for $G$ of $\eta$-subexponential volume growth is for all $R\le (\log n)^{1/\eta}$, which translates to $t\le (\log n)^2$. 
        
        It remains to deduce the exponential decay on $\rho(t)$ from~\eqref{eq:wts-recurrence}; consider the function 
        \begin{align*}
            \phi(t) = e^{\sqrt{t}} \big(\rho(t) +  e^{-t/2}\big)^{1/2}\,. 
            % \qquad \mbox{so} \qquad \phi(t)^2 = e^{2c \sqrt{t}} (\rho(t)+ e^{ - t/2}) 
        \end{align*}
        Then by~\eqref{eq:wts-recurrence}, and the fact that $\sqrt{a+b}\le \sqrt{a}+\sqrt{b}$, 
        \begin{align*}
            \phi(2t) \le e^{\sqrt{2}\sqrt{t}} \big(e^{\sqrt{t}} \rho(t)^2 + e^{-t} +  e^{ - t}\big)^{1/2} \le e^{(\sqrt{2} + .5 )\sqrt{t}} \rho(t) + \sqrt{2} e^{ \sqrt{2t}} e^{ - t/2}\,.
        \end{align*}
         Since $2\ge \sqrt{2}+ .5$ and $\sqrt{2}e^{\sqrt{2t}}\le e^{ 2\sqrt{t}}$ for all $t\ge 1$, this is at most $\phi(t)^2$. Therefore, for any $t_0 \ge 1$, we have 
        \begin{align*}
            \phi(2^k t_0) \le \phi(t_0)^{2^k}\,,
        \end{align*}
        whence if $\phi(t_0)<1/e$, then for $r=2^k$, we have $\phi(r t_0)\le e^{ - r}$. From there, using the definition of $\phi(t)$ in terms of $\rho(t)$, we see that $\rho(r t_0) \le e^{ - 2r}$, whence $\rho(t)\le e^{ - 2 t/t_0}$ for $t=2^k t_0$. The fact that $\rho(t)$ is monotone decreasing in time implies the bound $\rho(t) \le e^{ -t/t_0}$ for all $t \ge t_0$. 
        
%      	From there, we note that $\rho(r t_0)\le e^{ - 2r}- e^{ - rt_0/2}$. As long as $t_0$ is at least $6$, say, this is at most $e^{ - r}$, whence $\rho(t)\le e^{ - t/t_0}$ for $t=2^k t_0$. The fact that $\rho(t)$ is monotone decreasing in time implies the bound $\rho(t) \le e^{ -t/(2t_0)}$ for all $t \ge t_0$.  

        The last step is to show that $\phi(t_0)<1/e$ for some $t_0$ larger than $\max\{6, T_0(q,\Delta,\delta,\eta)\}$. Towards this purpose, notice that by the update rule~\eqref{eq:FK dynamics-update-rule}, 
        \begin{align*}
            \rho(s) \le (1-\ps) + \mathbb P(\text{Pois}(s) =0) = (1-\ps) + e^{ - s}\,.
        \end{align*}
        There exists $s_0$ independent of everything else such that  $e^{\sqrt{s}}(2e^{-s}+e^{-s/2})^{1/2}$ is less than $1/e$ for all $s>s_0$ because $e^{\sqrt{t}}(2e^{-t}+ e^{-t/2})^{1/2}$ is at most $3e^{ - t/4 + \sqrt{t}}$, say. Let $p_0(q)$ be large enough that $(1-\ps) \le e^{ - s_0}$ for all $p\ge p_0$. Then for all $p\ge p_0$ and $t_0\ge s_0$, $\rho(t_0)\le 2e^{ - t_0}$ and we obtain the claimed $\phi(t_0) <1/e$. 
        \end{proof}

        \begin{proof}[\textbf{\emph{Proof of Theorem~\ref{thm:subexp-mixing}}}]
Under the monotone grand coupling, we have for every initial state $\omega_0$, 
    \begin{align*}
        \max_{\omega_0} \|\mathbb P(X_t^{\omega_0} \in \cdot) -\pi\|_\tv\le \sum_{e\in E(G)} \mathbb P(X_t^{\omega_0}(e) \ne X_t^\pi(e)) \le \sum_{e\in E(G)} \mathbb P(X_t^\one(e) \ne X_t^\zero(e))\,.
    \end{align*}
    In turn, by monotonicity, the right-hand side is at most  
    \begin{align*}
        \max_{\omega_0} \|\mathbb P(X_t^{\omega_0} \in \cdot) -\pi\|_\tv \le \sum_{e\in E(G)} \big(\mathbb P(X_t^\one(e) = 1) - \mathbb P(X_t^\zero(e) = 1)\big)\,.
    \end{align*}
    Let $t= C_1 \log n$ for $C_1$ a large enough constant (depending on $q,\Delta,\delta,\eta$). Then by Proposition~\ref{prop:exp-relaxation-to-equilibrium}, each term in the right-hand side is bounded by $n^{-4}$ for large enough $n$. Since there are at most $\Delta n\le n^2$ many summands, the sum above is $o(1)$, implying mixing in $O(\log n)$ time. This gives $O(n\log n)$ mixing time for the discrete-time FK dynamics as described in the preliminaries.

    The result for the SW dynamics follows from the comparison result of~\cite{Ullrich-random-cluster}.
\end{proof}

\section{Spatial and temporal mixing on trees with $r$-wired boundary}\label{sec:trees}
Our next goal in the paper is to establish Theorem~\ref{thm:exp-growth-mixing} concerning FK dynamics on treelike expanders. Recursive mixing time arguments based on disagreement percolation, like those in the previous section, are known to fail on graphs with exponential volume growth. At the same time, localizing the dynamics can be difficult in treelike graphs because there is no weak spatial mixing when $p$ is close to $1$; see the discussion at the beginning of Section~\ref{sec:tree-spatial-mixing} for more details. In this section, we define a class of boundary conditions that are sufficiently ``wired" to support a notion of weak spatial mixing with respect to the wired boundary conditions. This class of boundary conditions captures the boundary conditions induced by the FK dynamics configuration on a treelike ball centered at a vertex of an expander graph after a short burn-in. 

This section focuses on general rooted trees $\cT_h = (V(\cT_h),E(\cT_h))$ having depth $h$, minimum internal degree $3$  and maximum degree $\Delta$. For any $m\le h$, $\cT_m$ will denote the tree given by truncating $\cT_h$ at depth $m$. The boundary $\partial \cT_m$ is the set of vertices of $\cT_m$ at depth $m$. For any vertex $w\in V(\cT_h)$, we use $\cT_{h,w}$ to denote the sub-tree of $\cT_h$ rooted at $w$, with boundary $\partial \cT_{h,w} = \partial \cT_h \cap \cT_{h,w}$.

\begin{definition}
    A boundary condition $\xi$ of $\cT_h$ is \emph{single-component} if the boundary partition corresponding to $\xi$ has at most one non-singleton element; we call this its \emph{wired component}.
\end{definition}

\begin{definition}
    A distribution $\mathsf{P}$ over boundary conditions $\xi$ on $\cT_h$ is $r$-wired if it is supported on single-component boundary conditions, and the distribution of the wired component of $\xi$ stochastically dominates the distribution over random subsets $A\subset \partial \cT_h$ in which each vertex of $\partial \cT_h$ is included in $A$ with probability $r$ independently (the partial order being the natural one on vertex subsets).
\end{definition}

The following shows that except with \emph{double-exponentially} small probability, the random-cluster model on $\cT_h$ with $r$-wired boundary conditions satisfies weak spatial mixing with respect to the all-wired boundary condition (the TV distance between the two decays exponentially in the distance from the boundary).   

\begin{lemma}\label{lem:tree-spatial-mixing}
Let $\mathsf{P}$ be $r$-wired and let $\xi\sim \mathsf{P}$. Then, with $\mathsf{P}$-probability $1- e^{ - c_{p,r} (1.1)^h}$, we have 
\begin{align*}
    \|\pi_{\cT_{h}}^\xi(\omega(\cT_{h/2}) \in \cdot) - \pi_{\cT_h}^{\one}(\omega(\cT_{h/2})\in \cdot) \|_\tv \le e^{ - c_{p,r} h}\,,
\end{align*}
for some $c_{p,r} = c_{p,r}(q,\Delta)$, which is positive as long as $p>p_0(q,\Delta)$ and $r>r_0(q,\Delta)$ for suitable $p_0(q,\Delta) , r_0(q,\Delta)  \in (0,1)$, 
and which goes to $\infty$ as $p,r \to 1$. 
\end{lemma}

\begin{remark}
    The double exponential concentration under $\mathsf{P}$ in Lemma~\ref{lem:tree-spatial-mixing} is not strictly needed for the results in this paper. A single-exponential concentration under $\mathsf{P}$ would suffice and would be easier to establish by averaging over $\xi\sim \mathsf{P}$ and then applying Markov's inequality. However, we include this stronger form  since it provides insight into the good qualities of $r$-wired boundary conditions and could be used to get an improved mixing time bound for the FK dynamics on trees with $r$-wired boundary. 
    Unfortunately, with our current methods in Section~\ref{sec:treelike-expanders}, this improvement would not translate into better bounds for the FK dynamics on treelike graphs, so we do not pursue them.
\end{remark}

\subsection{Spatial mixing on trees with $r$-wired boundary conditions}\label{sec:tree-spatial-mixing}
Our first aim in this section is to establish Lemma~\ref{lem:tree-spatial-mixing}. 
Spatial mixing in the traditional sense, where one takes a maximum over the boundary conditions on the tree, does \emph{not} hold for large $p$, even if it is very close to $1$. This can be seen by considering $\cT_h$ with wired vs.\ free boundary conditions, and noticing that the marginal 
of any edge in the tree with the free boundary condition is $\mbox{Ber}(\ps)$, whereas the marginal of an edge in 
the wired tree gives at least constant probability to that edge being distributed as $\mbox{Ber}(p)$, and otherwise, as $\mbox{Ber}(\ps)$, so the total-variation distance on that edge does not go to zero as $h\to\infty$. 
Our solution to this issue is to restrict attention to $r$-wired boundary conditions and establish that, at least among these boundary conditions, the random-cluster model exhibits spatial mixing on $\cT_h$. Considering only such boundary conditions will suffice for us, 
since these are the boundary conditions that appear after a burn-in period of the FK dynamics 
on treelike expander graphs.  We note that with \emph{fully} wired boundary conditions on trees,~\cite{Jonasson} showed decay of correlations to establish uniqueness at $p$ close to $1$.

The mechanism for coupling random-cluster configurations with $r$-wired and wired boundary conditions is based on what we call wired separating sets. These will be a set of vertices that are all connected down to the wired component of $\xi$ in $\partial \cT_h$ and therefore wired together; as such, they separate the influence of the boundary condition of $\partial \cT_h$ from, say, $\partial \cT_{h/2}$. 

\begin{definition}\label{def:separating-set}
    A separating set in $\cT_h$ is a set of vertices $S \subset V(\cT_h)\setminus V(\cT_{h/2})$ such that every path from $\partial \cT_{h/2}$ to $\partial \cT_h$ must intersect $S$. 
    A configuration $\omega \subset E(\cT_h)$ has a \emph{wired separating set} if there exists a separating set $S$ such that every vertex $v\in S$ is connected in $\omega(\cT_{h,v})$ to a vertex $u\in \partial \cT_{h,v}$ belonging to the wired component of $\xi$. 
    Let $\mathcal S_{h,\xi}$ be the event that $\omega$ has a wired separating set in $\cT_h$ with boundary condition~$\xi$. 
\end{definition}

The following lemma shows how the event of having a wired separating set governs 
the probability of coupling random-cluster configurations to $\omega \sim \pi_{\cT_h}^\one$ in $\cT_{h/2}$. 

\begin{lemma}\label{coupling-with-separating-set}
    For any single-component boundary condition $\xi$ on $\cT_h$, 
     \begin{align*}
        \|\pi_{\cT_{h}}^\xi (\omega(\cT_{h/2}) \in \cdot) - \pi_{\cT_{h}}^\one(\omega(\cT_{h/2}) \in \cdot) \|_\tv &\le \pi_{\cT_h}^\xi(\mathcal S_{h,\xi}^c)\,.
    \end{align*}
\end{lemma} 

\begin{proof}
    We construct a monotone coupling for
    $\pi_{\cT_h}^\xi$ and $\pi_{\cT_h}^\one$
    such that if $(\omega_\xi,\omega_\one)$ is sampled from this coupling, 
    then
    $\omega_\xi \sim \pi_{\cT_h}^\xi$, $\omega_\one\sim \pi_{\cT_h}^\one$ and
     \begin{align*}
         \{\omega_\xi \in \cS_{h,\xi}\} \implies \{\omega_\xi (\cT_{h/2}) = \omega_\one (\cT_{h/2})\}.
     \end{align*}
    Our construction relies on revealing the values of $(\omega_\xi,\omega_\one)$ on an edge set $\mathcal R$ under a monotone coupling, where crucially, $\mathcal R$ will be designed to be the set of edges in the ``lowest'' wired separating set (if one exists).
    We construct $\omega_\xi$ and $\omega_{\one}$ as follows (and refer to Figure~\ref{fig:tree-revealing-process} for a depiction):
    \begin{enumerate}
    \item Initialize $\cP_0$ as the parents of $\partial \cT_h$, and $\mathcal R_0 = \emptyset$.
    \item Starting with $i=1$, pick a vertex $w_i \in \cP_{i-1}$, and sample the configurations $(\omega_\xi(\cT_{h,w_i}),\omega_{\one}(\cT_{h,w_i}))$ 
    from the monotone coupling between the marginals $$\pi_{\cT_h}^\xi(\omega_\xi(\cT_{h,w_i})\in \cdot \mid \omega_\xi(\mathcal R_{i-1}))\,, \qquad \mbox{and}\qquad \pi_{\cT_h}^\one(\omega_\one(\cT_{h,w_i})\in \cdot  \mid \omega_1(\mathcal R_{i-1}))\,.$$
    \item Let $\mathcal R_i = \mathcal R_{i-1} \cup E(\cT_{h,w_i})$ and form $\omega_\xi(\mathcal R_i)$ and $\omega_{\one}(\mathcal R_i)$ by adding the configurations $\omega_\xi(\cT_{h,w_i})$ and $\omega_{\one}(\cT_{h,w_i})$, respectively.
    \item If $w_i $ is connected to the wired component of $\xi$ in $\omega_\xi(\cT_{h,w_i})$, then let $\cP_{i} = \cP_{i-1}\setminus \cT_h(w_i )$; else, let $\cP_i = (\cP_{i-1} \setminus \cT_h(w_i )) \cup \{w_i '\}$ where $w_i '$ is the parent of $w_i $. 
    \end{enumerate}
    While $i$ is such that $\cP_{i-1}$ is non-empty, $\mathcal R_i \setminus \mathcal R_{i-1} \ne \emptyset$. This is because the edges from $w_i$
    to its children will be in $\mathcal R_i$ but not in $\mathcal R_{i-1}$ (if they were in $\mathcal R_{i-1}$ then $w_i \in \cT_{w_j}$ for some $j<i$ and $w_i$ would have been removed from $\cP_{j}$). Therefore, the revealing process will terminate after a finite number of steps, and we can call $\cR = \cR_f$ if $f$ is the first index for which $\cP_f = \emptyset$.  

    \begin{figure}
        \centering
        \begin{tikzpicture}
        \node at (-4,0) {
        \includegraphics[width = .45\textwidth]{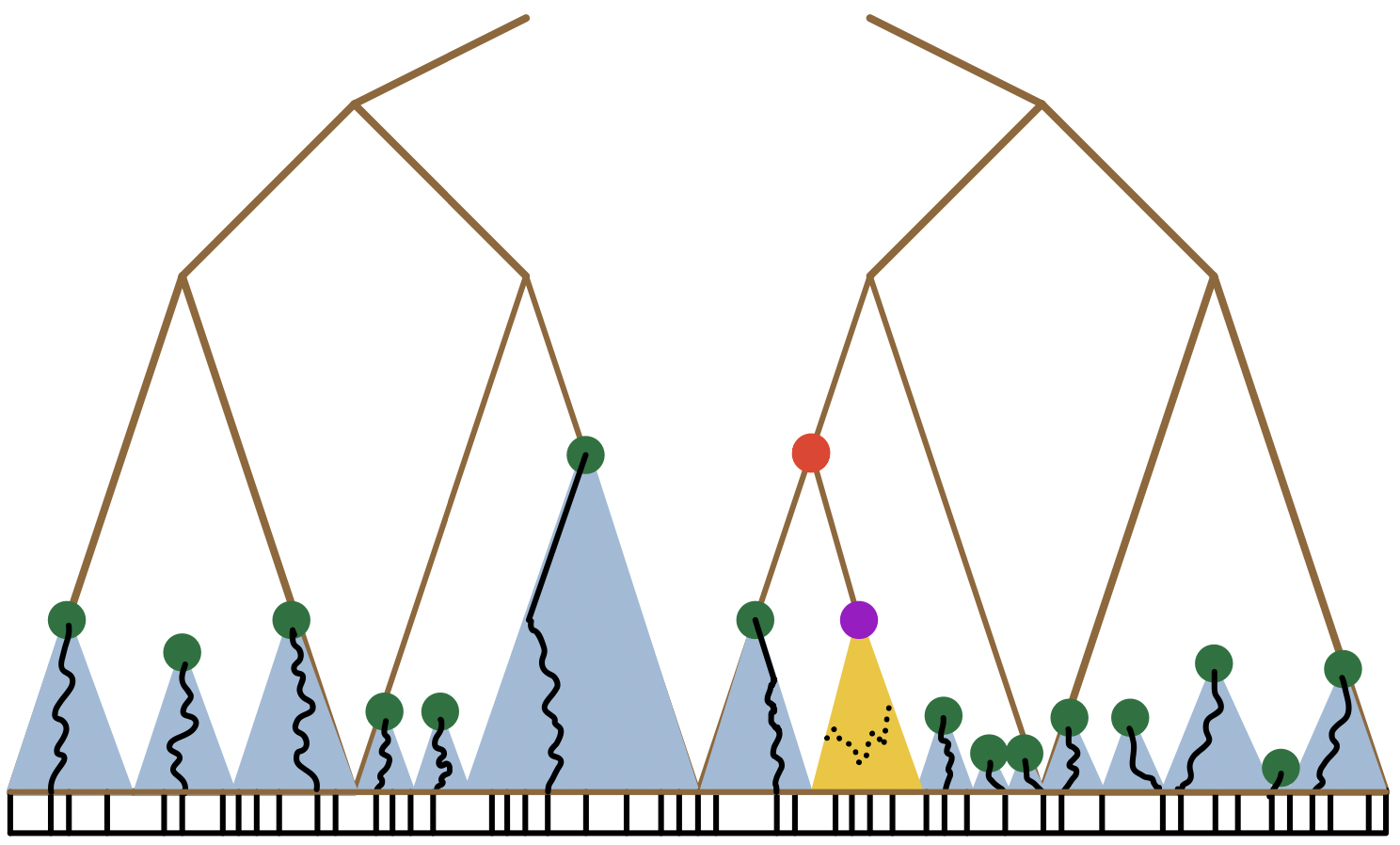}};
        \node at (4,0) { 
        \includegraphics[width = .45\textwidth]{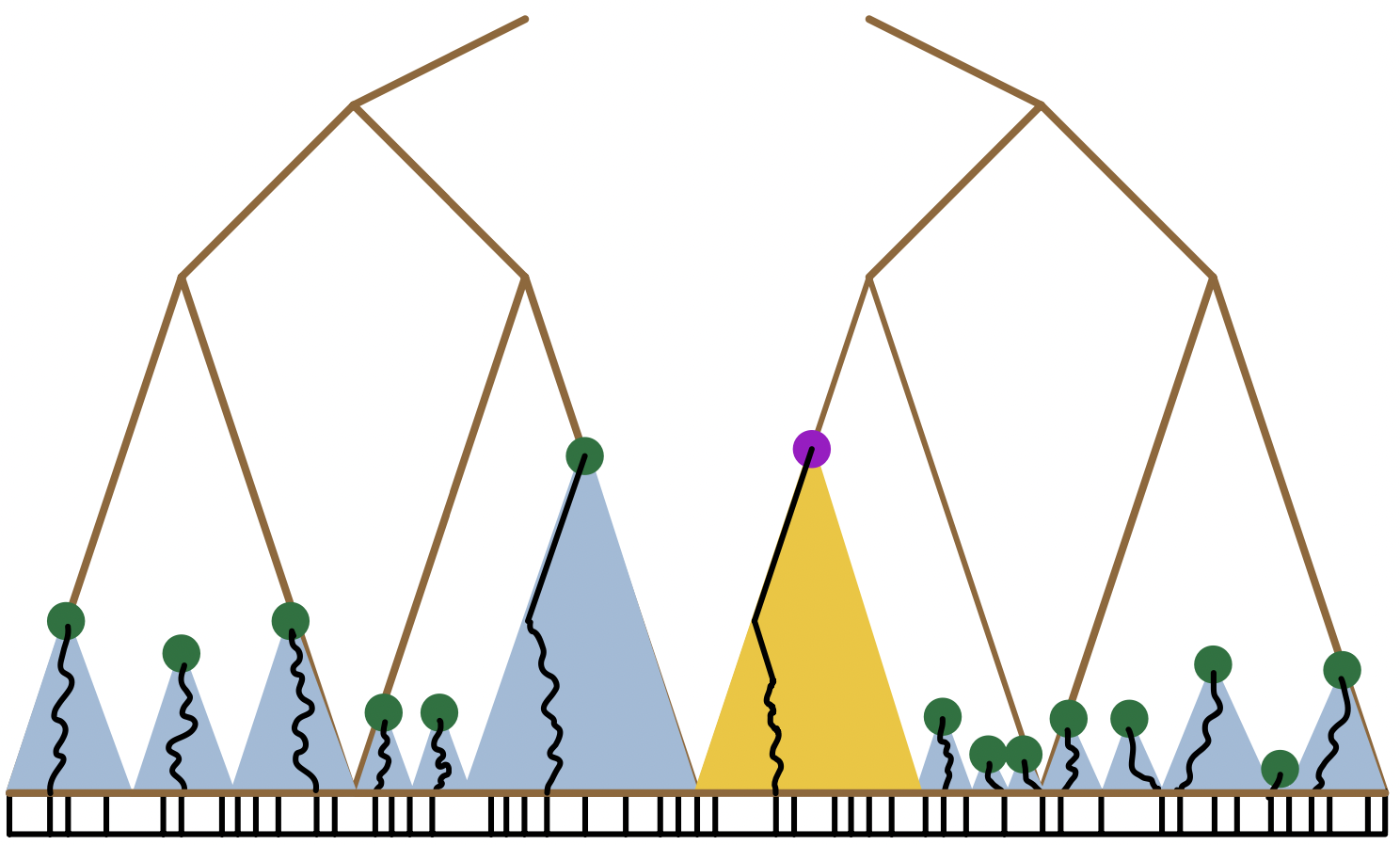}};
        \node[font = \tiny, color = red] at (-3.15,-.1) {$w_i'$};
        \node[font = \tiny, color = blue!50!red] at (-2.85,-1.075) {$w_i$};
        \node[font = \tiny, color = blue!50!red] at (5.05,-.13) {$w_{i+1}$};
        \node[font = \tiny, color = black] at (-4,-2.45) {$\xi$};
        \node[font = \tiny, color = black] at (4,-2.45) {$\xi$};
        \end{tikzpicture}
        \caption{Two steps of the revealing process used in Lemma~\ref{coupling-with-separating-set}. Left: the configurations on the sub-tree of $w_i\in \mathcal P_{i-1}$ are revealed, and in the configuration with $\xi$-boundary conditions, $w_i$ is not wired down to $\xi$. Thus, its parent $w_i'$ is added to $\mathcal P_i$. In the next step, this vertex is $w_{i+1}$, and when the remainder of its sub-tree is revealed to indeed include a wiring to $\xi$, the vertex is removed from $\mathcal P_i$ but its parent is no longer added.}
        \label{fig:tree-revealing-process}
    \end{figure}
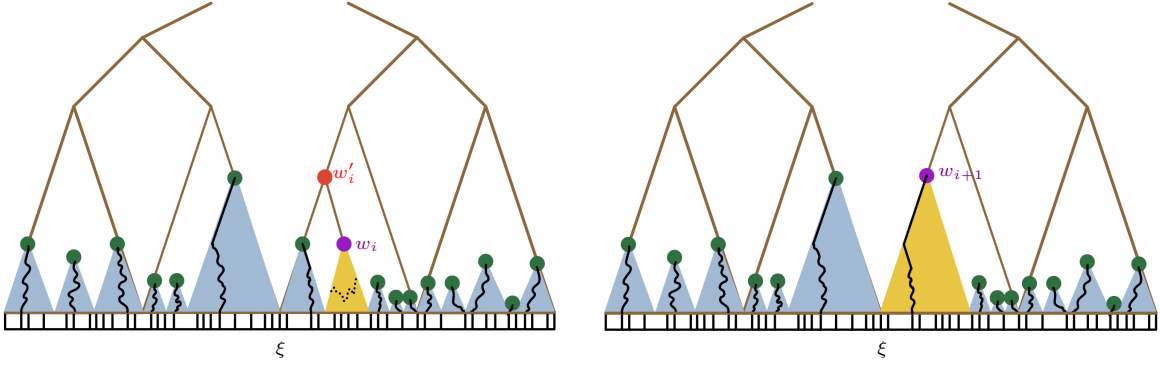

    Now, consider the subset $\partial_{\textsc{top}} \cR$ of vertices of $\cR$ whose parents are not in $\cR$.
    As long as the process terminates with $\cR \cap E(\cT_{h/2}) = \emptyset$, the vertices in $\partial_{\textsc{top}}\mathcal R$ will be connected to the wired component of $\xi$ in $\omega_\xi(\cT_{h,w})$, and so they will form a wired separating set. Conversely, if there exists a wired separating set $S$ in $\omega_\xi(\cT_{h,w})$, no parent of any vertex in $S$ will ever be added to $\cP$, and therefore, the vertices from $\partial_{\textsc{top}}$ will form exactly the lowest wired separating set. 
    
    Upon termination of the iterative procedure above, we can then sample $\omega_\xi(\cT_h\setminus \cR) = \omega_\one(\cT_h\setminus \cR)$, since both $\omega_\xi(\mathcal R)$ and $\omega_\one(\mathcal R)$ induce wired boundary conditions on $E(\cT_h)\setminus \mathcal R$.  
    Therefore, under this coupling, we always have $\omega_\xi(\cT_h\setminus \cR) = \omega_\one(\cT_h\setminus \cR)$. 
    Since the process terminates at the lowest wired separating set of $\omega_\xi$, on the event $\omega_\xi \in \cS_{h,\xi}$, necessarily $\cT_h \setminus \cR$ contains all of $\cT_{h/2}$. As such, using $\mathbb P$ to denote the probability under the coupling we just designed, 
    \begin{align*}
        \|\pi_{\cT_{h/2}}^\xi - \pi_{\cT_{h/2}}^\one \|_\tv \le \mathbb P(\cR \cap \cT_{h/2} \ne \emptyset) \le \mathbb P(\omega_\xi \in \cS_{h,\xi}^c) = \pi_{\cT_h}^\xi(\mathcal S_{h,\xi}^c)\,,
    \end{align*}
    as claimed.
\end{proof}

\subsubsection{Good boundary conditions}
Our aim is now to control the probability of $\cS_{h,\xi}^c$ under a random $r$-wired boundary condition $\xi \sim \mathsf{P}$. Recalling the definition of $\cS_{h,\xi}$, notice that it is an increasing event in the random-cluster configuration. Since the random-cluster measure
with parameters $p$ and $q$ stochastically dominates independent percolation with parameter $\ps = \frac{p}{q(1-p)+p}$ configuration (i.e., the random-cluster measure with parameters $q=1$ and $\ps$), it will suffice for us to consider the probability of $\cS_{h,\xi}$ under the product measure $\widetilde \pi_{\cT_h} = \pi_{\cT_h,\tilde p,1}^\xi$ for some $\tilde p\le \ps$ still going to $1$ with $p$ (e.g., $\tilde p = 2\ps -1$). 

\begin{definition}
A boundary condition $\xi$ on $\cT_h$ is called $c$-good if 
\begin{align*}
    \widetilde \pi_{\cT_h}(\widetilde \omega \in \mathcal S_{h,\xi}^c) \le \exp( - ch)\,.
\end{align*}
\end{definition}

Notice that if $\xi$ is $c$-good, then any single-component boundary condition $\xi'\ge \xi$ will also be $c$-good, since any $\widetilde \omega \in \mathcal S_{h,\xi}$ will also be in $\mathcal S_{h,\xi'}$. Therefore, the event $\{\xi \, \mbox{is $c$-good}\}$ is itself increasing in the partial order on subsets of $\partial \cT_h$. In particular, if $\mathsf{P}$ is $r$-wired, then 
$$
\mathsf{P}(\xi \mbox{ is $c$-good}) \ge \mathsf{P}_{\mathrm{Ber}(r)}(\xi \mbox{ is $c$-good})
$$
where $\mathsf{P}_{\mathrm{Ber}({r})}$ is the distribution over boundary conditions on $\partial \cT_h$ where the wired component contains each vertex independently with probability ${r}$. 
Given this, the following lemma implies Lemma~\ref{lem:tree-spatial-mixing}. 

\begin{lemma}\label{lem:xi-not-c-good}
    Suppose $\xi\sim \mathsf P_{\mathrm{Ber}(r)}$. There exists $c = c(p,q,r)$ going to $\infty$ as $p,r\to 1$, such that  
    \begin{align*}
        \mathsf P(\xi \mbox{ is not $c$-good}) \le \exp( - c (1.1)^h)\,.
    \end{align*}
\end{lemma}

\begin{proof}
    For some $\gamma = \gamma(p,q,r)>0$ to be chosen later,  going to $1$ as $p,r \to 1$, consider the event $\mathcal E_{\gamma}$ 
    that $\xi$ belongs to the set of boundary conditions on $\cT_h$ that satisfy the following property:

    \begin{enumerate}
        \item[(P1)]  for each downward path from $\partial \cT_{h/2}$ to $\partial \cT_h$: $v_0 ,v_1,...,v_{h/2}$,
    for each $i=1,...,h/4$, if we draw a configuration $\widetilde \omega (\cT_{v_i}\setminus \cT_{v_{i+1}})$ from $\widetilde \pi_{\cT_{v_i}\setminus \cT_{v_{i+1}}}$, the probability that the component of $v_i$ in $\widetilde \omega (\cT_{v_i}\setminus \cT_{v_{i+1}})$ intersects the wired component of $\xi\cap \partial (\cT_{v_i}\setminus \cT_{v_{i+1}})$ is at least $\gamma$.
    \end{enumerate}
    
    We first show that $\xi \in \mathcal E_{\gamma}$ implies $\xi$ is $c$-good for a suitable $c > 0$. 
    Fix any $\xi\in \mathcal E_{\gamma}$. In order for $\cS_{h,\xi}^c$ to occur, there must exist a path from $\partial \cT_{h/2}$ to $\partial \cT_h$ such that no $v_i$ is connected to $\xi$ through $\widetilde \omega(\cT_{v_i})$. For any fixed path, that probability is upper bounded by 
    \begin{align*}
        \widetilde\pi_{\cT_h}\Big(\bigcap_{i=1}^{h/4}\{\widetilde \omega(\cT_{v_i} \setminus \cT_{v_{i+1}}) \cap \xi = \emptyset\} \Big) = \prod_{i=1}^{h/4} \widetilde\pi_{\cT_h}(\widetilde \omega(\cT_{v_i} \setminus \cT_{v_{i+1}}) \cap \xi = \emptyset) \le (1-\gamma)^{h/4}\,.
    \end{align*}
    Here we have abused notation slightly to identify $\xi$ with the subset of $\partial \cT_h$ that is its wired component. 
    (The change from the intersection to the product comes from the fact that $\widetilde\pi_{\cT_h}$ is a product measure and that 
    the sets $\cT_{v_i}\setminus \cT_{v_{i+1}}$ are disjoint for different $i$.)
     A union bound over the $\Delta^h$ many paths implies that the probability that $\omega\notin \cS_{h,\xi}$ is at most $\Delta^{h/2} (1-\gamma)^{h/4}$ which is at most $\exp(-c h)$ for some $c > 0$  going to $\infty$ as $\gamma\to 1$, which happens as $p,r\to 1$. 

     It now suffices to find such a $\gamma > 0$, and a $\hat c  = \hat c(p,q,r)> 0$ such that  
    \begin{align}\label{nts:xi-not-c-good}
    \mathsf{P}_{\mathrm{Ber}(r)} \big(\xi \notin \mathcal E_{\gamma}\big) \le \exp( - \hat c (1.1)^{h})\,.
    \end{align}
    Fix a downward path $v_0,...,v_{h/2}$ from $\partial \cT_{h/2}$ to $\partial \cT_h$ and an $i\in \{1,...,h/4\}$ (these will subsequently be union bounded over), and consider the probability that $\xi$ is such that (P1) above holds for that path and that $i$. 

    Define the event $\mathcal F^i$ that 
    the connected component of $v_i$ in
    $\widetilde \omega(\cT_{v_i}\setminus \cT_{v_i+1})$ intersects $\partial (\cT_{v_i}\setminus \cT_{v_i+1})$ in at least $(2\tilde p - \epsilon)^{h(v_i)-1}$ many sites, where $h(v_i)$ denotes the height of $\cT_{v_i}$. Note that $h(v_i) = i-h/2 \ge h/4$, and $\epsilon$ will be chosen later. For intuition, the component of $v_i$ in $\widetilde \omega(\cT_{v_i}\setminus \cT_{v_i +1})$ is a branching process which (besides its first level) has $2\tilde p$ expected number of children since we assumed that in $\cT_{h}$ all internal vertices have degree at least $3$.
    By a standard branching process argument (see Fact~\ref{fact:branching-process} below from which this follows after asking that the first level have an open edge with probability at least $\tilde p$), we know there exists $\delta$ going to $0$ as $\tilde p\to 1$ (equivalently as $p\to 1$) such that 
    \begin{align*}
        \widetilde\pi_{\cT_h}(\widetilde \omega(\cT_{v_i}\setminus \cT_{v_i+1}) \notin \mathcal F^i) 
        \le \delta\,.
    \end{align*}
       
    By the independence of $\widetilde \omega(\cT_{v_i}\setminus \cT_{v_i+1})$ from $\xi$, we have 
    \begin{align*}
        \mathsf{P}_{\text{Ber}(r)} \otimes \widetilde\pi_{\cT_{h}}\Big( \widetilde \omega(\cT_{v_i}\setminus \cT_{v_i+1}) \cap \xi = \emptyset,~ \widetilde \omega(\cT_{v_i}\setminus \cT_{v_i+1}) \in \mathcal F^i\Big) 
         & \le (1-r)^{(2 \tilde p - \epsilon)^{h(v_i)-1}}\,.
    \end{align*}
    The left-hand side above is exactly the expected value over $\xi \sim \mathsf{P}_{\text{Ber}(r)}$ of the $\widetilde \pi_{\cT_h}$ probability of an event depending on $\xi$.
    Thus by Markov's inequality and the fact that $h(v_i)\ge h/4$, 
    \begin{align*}
        \mathsf P_{\mathrm{Ber}(r)}\Big(\xi: \widetilde\pi_{\cT_{h}}\big( \widetilde \omega(\cT_{v_i}\setminus \cT_{v_i+1}) \cap \xi & = \emptyset,~ \widetilde \omega(\cT_{v_i}\setminus \cT_{v_i+1}) \in \mathcal F^i\big) >(1-r)^{(2 \tilde p - \epsilon)^{h/4}/2}\Big) \\
       & \le  \frac{ \mathsf{E}_{\xi\sim \text{Ber}(r)} \Big[\widetilde\pi_{\cT_{h}}\Big( \widetilde \omega(\cT_{v_i}\setminus \cT_{v_i+1}) \cap \xi = \emptyset,~ \widetilde \omega(\cT_{v_i}\setminus \cT_{v_i+1}) \in \mathcal F^i\Big)\Big] }{(1-r)^{(2 \tilde p - \epsilon)^{h/4}/2}}	\\
              & =  \frac{ \mathsf{P}_{ \text{Ber}(r)} \otimes \widetilde\pi_{\cT_{h}}\Big( \widetilde \omega(\cT_{v_i}\setminus \cT_{v_i+1}) \cap \xi = \emptyset,~ \widetilde \omega(\cT_{v_i}\setminus \cT_{v_i+1}) \in \mathcal F^i\Big)}{(1-r)^{(2 \tilde p - \epsilon)^{h/4}/2}}	\\
        &\le \exp( (2 \tilde p - \epsilon)^{h/4}\log(1-r)/2)\,.
    \end{align*}
    At the same time, for any fixed $\xi$, 
    \begin{align*}
        \widetilde\pi_{\cT_{h}}(\widetilde \omega(\cT_{v_i}\setminus \cT_{v_i+1}) \cap \xi = \emptyset) & \le \widetilde\pi_{\cT_{h}}(\widetilde \omega(\cT_{v_i}\setminus \cT_{v_i+1}) \cap \xi = \emptyset, \widetilde \omega(\cT_{v_i}\setminus \cT_{v_i+1}) \in \mathcal F^i) \\
        & \quad + \widetilde\pi_{\cT_{h}}(\widetilde \omega(\cT_{v_i}\setminus \cT_{v_i+1}) \notin \mathcal F^i)\,.
    \end{align*}
    If the first of these terms is at most $(1-r)^{(2 \tilde p - \epsilon_p)^{h/4}/2}$ and the second is at most $\delta$, we set $\gamma =  1-\delta - (1-r)^{(2 \tilde p - \epsilon)^{h/4}/2}$.  Then, $\gamma$ goes to $1$ as $p\to 1$ and by the above, $ \widetilde\pi_{\cT_{h}}(\widetilde \omega(\cT_{v_i}\setminus \cT_{v_i+1}) \cap \xi \ne \emptyset)\ge \gamma$. Thus, 
    \begin{align*}
        \mathsf P_{\mathrm{Ber}(r)}\big(\xi: \widetilde\pi_{\cT_h}(\widetilde \omega(\cT_{v_i}\setminus \cT_{v_i+1}) \cap\xi= \emptyset)> 1- \gamma\big) \le \exp( (2 \tilde p - \epsilon)^{h/4}\log(1-r)/2)\,.
    \end{align*}
    We now upgrade this into the probability that $\xi$ is in $\mathcal E_\gamma$ by a union bound over all $\Delta^{h/2}$ many paths in $\cT_h$ and the $h/4$ many possible $i$'s; both these terms are absorbed by the double exponential above. As long as $\epsilon$ is sufficiently small, and $p$ is sufficiently large, $(2\tilde p-\epsilon)^{1/4}$ is greater than $1.1$, and the prefactor evidently blows up as $r\to 1$ as claimed. 
\end{proof}

For completeness, we have included the following simple branching process concentration estimate. 

\begin{fact}\label{fact:branching-process}
    In a branching process with progeny distribution stochastically dominating $\text{Bin}(2,\tilde p)$, let $Z_k$ be the population size at level $k$. Then for every $\epsilon>0$, there exists $\delta>0$ going to $0$ as $p\to 1$ such that 
    \begin{align*}
        \mathbb P\Big( \bigcap_{k\ge 1} \{Z_k \ge (2\tilde p -\epsilon)^k\}\Big) \ge 1-\delta\,.
    \end{align*}
\end{fact}

\begin{proof}
    Since the event in question is an increasing event, it suffices to show the above for the branching process with progeny distribution exactly $\text{Bin}(2,\tilde p)$. If we let $\mathcal A_k$ be the event  $\{Z_k \ge (2\tilde p - \epsilon)^k\}$, then we can write 
    \begin{align*}
        \mathbb P \Big(\bigcup_{k\ge 1} \mathcal A_k^c\Big)  \le \mathbb P(\mathcal A_1^c) +  \sum_{k\ge 2} \mathbb P\Big(\mathcal A_k^c , \bigcap_{j<k}\mathcal A_j\Big) \le \mathbb P(\mathcal A_1^c) +  \sum_{k\ge 2} \mathbb P\Big(\mathcal A_k^c \mid \bigcap_{j<k}\mathcal A_j\Big)\,.
    \end{align*}
    Since $Z_1 = 1$, the probability of $\mathcal A_1$ is $1$. For any $k\ge 2$, since $Z_k$ is Markov, it suffices to condition on $Z_{k-1}: Z_{k-1}\in \mathcal A_{k-1}$; given $Z_{k-1}$, the distribution of $\mathcal A_{k-1}$ is $$\sum_{i=1}^{Z_{k-1}} X_i \qquad \text{where $X_i$ are i.i.d.\ $\text{Bin}(2,\tilde p)$}\,. $$

    Thus, $\mathbb P(\mathcal A_k^c \mid Z_{k-1},\mathcal A_{k-1})$ is at most the  probability of a sum of $Z_{k-1} \ge (2\tilde p-\epsilon)^{k-1}$ i.i.d.\ $\text{Bin}(2,\tilde p)$ random variables, being at least $\epsilon Z_{k-1}$ below its mean. By Hoeffding's inequality (the $X_i$'s being bounded by $2$), this has probability at most $\exp( - \tfrac{1}{2}\epsilon^2 (2\tilde p -\epsilon)^{k-1})$ for every $k$. Let $K$ equal $\log_{2\tilde p-\epsilon}(2/(\epsilon^2\sqrt{1-\tilde p}))$. For the first $K$ generations, we can use the simpler union bound over the probability that one of the first $(2\tilde p-\epsilon)^k$ many $X_i$'s is not equal to $2$. 
    Putting these together, we get  
    \begin{align*}
        \mathbb P\Big(\bigcup_{k\ge 1}\mathcal A_k^c\Big) \le \frac{4}{\epsilon^2} (1-\tilde p)^{1/2} + \sum_{k>K} e^{ - \frac{1}{\sqrt{1-\tilde p}}(k-K)}\,.
    \end{align*}
    For any fixed $\epsilon$, this is then seen to be at most some $\delta$ going to zero as $\tilde p \to 1$, which happens as $p\to 1$. 
\end{proof}

\subsection{Mixing time for trees with single-component boundary conditions}\label{sec:single-component-tree-mixing-time}
In this section, we provide a bound on the mixing time of the FK dynamics on trees with single-component boundary conditions. 
In view of its application in the following section for treelike balls of expanders (possibly having a finite number of cycles), it will be helpful for us to recall 
certain standard definitions that will allow us to relate the convergence rate of various Markov chains.

For a Markov chain on a finite state space $\Omega$ with transition matrix $P$ and stationary distribution $\mu$, the Dirichlet form of the chain is defined for any function $f:\Omega \to \mathbb R$ by
\begin{align}\label{eq:Dirichlet-form}
    \mathcal E(f,f) := \frac{1}{2}\sum_{\omega,\omega'\in \Omega} \mu(\omega) P (\omega, \omega') (f(\omega) - f(\omega'))^2\,,
\end{align}
and its {spectral gap}  is given by 
\begin{align}\label{eq:lsi-constant}
    \lambda_P := \min_{f: \text{Var}_\mu[f]\ne 0} \frac{\mathcal E(f,f)}{\mbox{Var}_\mu[f]}\,, 
\end{align}
where $\text{Var}_{\mu} [f] = \mathbb E_{\mu}[f^2] - \mathbb E_\mu[f]^2$ with $\mathbb E_{\mu}[f] = \sum_{\omega \in \Omega} \mu(\omega)f(\omega)$.

The inverse of the spectral gap is closely connected to the mixing time of a Markov chain. In particular,  
\begin{align}\label{eq:tmix-trel-comparison}
	(\lambda_P^{-1}-1) \log\big(\tfrac{1}{2\epsilon}\big) \le \tmix(\epsilon) \le \lambda_P^{-1} \log\big(\tfrac{1}{\epsilon \mu_{\min}}\big)\,,
\end{align}
where $\mu_{\min} = \min_{\omega \in \Omega} \mu(\omega)$. 
(We refer the reader to e.g.,~\cite[Chapter 12.2]{LP} for more details.) 

We establish the following bounds for the mixing time and inverse spectral gap of the FK dynamics on a tree with any single-component boundary condition.

\begin{lemma}\label{lem:wired-tree-spectral-gap}
	Consider any tree $\cT_h$ of maximum degree $\Delta$ and depth $h$ with single-component boundary condition $\xi$. There exist constants $a = a(\Delta,q) > 0$
    and $C = C(p,q) > 0$
    such that the inverse spectral gap of the FK dynamics on $\cT_h$ with boundary condition $\xi$ are at most $C \exp(a h)$.

\end{lemma}
\begin{proof}   
    We use the classical bound on the spectral gap obtained from the cut-width of a graph 
    via the canonical paths method, though a little care is needed for the purpose of handling the random-cluster boundary condition. The edge-cut-width of $\cT_h$ (in other words, the cut-width of its line graph) is defined as follows: enumerate the edges of $\cT_h$ as $1,...,|E(\cT_h)|$ and define 
    \begin{align*}
        \mathsf{CW}(\cT_h) = \min_{\sigma}  \max_{i} |V(\{e_{\sigma(j)}:j\le i\}) \cap V(\{e_{\sigma(j)}:j>i\})|
    \end{align*}
    where the minimum is over permutations $\sigma$ on $\{1,...,|E(\cT_h)|\}$. 
    We claim that there exists a constant $K(\Delta)$ such that uniformly over all trees of degree at most $\Delta$, their edge-cut-width is at most $K h$;
    this follows e.g., from \cite[Lemma 2.1]{kenyon2001glauber-ptrf} and the fact that the edge-cut-width is within a factor of $\Delta$ of the (vertex) cut-width. Let $\sigma$ be the permutation that attains this edge-cut-width bound for $\cT_h$. 

    For any two random-cluster configurations $I,F\in \Omega$ on $\cT_h$, define the canonical path $\gamma_{I,F}$ as the path of FK dynamics transitions which sequentially processes the edges of $E(\cT_h)$ according to the ordering induced by $\sigma$, i.e., $e_{\sigma(1)},e_{\sigma(2)},...,e_{\sigma(|E(\cT_h)|}$, and whenever there is a discrepancy $I(e_{\sigma(i)})\ne F(e_{\sigma(i)})$, the transition is the one that changes the state of $e_{\sigma(i)}$ from $I(e_{\sigma(i)})$ to $F(e_{\sigma(i)})$. 

    For an FK dynamics transition $(\eta,\eta^i)$ where $\eta^i := \eta \oplus e_{\sigma(i)}$, construct a bijection from the set of $\{I,F:(\eta,\eta^i)\in \gamma_{I,F}\}$ to $\Omega$ by setting $$\omega_\eta(I,F) = \{I(e_{\sigma(j)}):j\le i\} \cup \{F(e_{\sigma(j)}): j>i\}\,.$$ This is a bijection because $I$ is recovered via $I=\{\omega_\eta(e_{\sigma(j)}):j\le i\}\cup \{\eta(e_{\sigma(j)}):j>i\}$ and $F$ is analogously recovered via $F=\{\eta(e_{\sigma(j)}):j\le i\}\cup \{\omega_\eta(e_{\sigma(j)}):j>i\}$. 

    For ease of notation, let $\pi = \pi_{\cT_h}^\xi$.
    The standard canonical paths bound (see~\cite[Corollary 13.20]{LP})  then ensures the inverse gap satisfies  
    \begin{align*}
        \lambda_P^{-1} &\le \max_i \max_\eta \frac{1}{\pi(\eta)P(\eta,\eta^i)}  \sum_{I,F: (\eta,\eta^i)\in \gamma_{I,F}} \pi(I)\pi(F)|\gamma_{I,F}|  \\ 
        & \le |E(\cT_h)| \max_{\eta,i} \max_{I,F: (\eta,\eta^i)\in \gamma_{I,F}} \frac{1}{P(\eta,\eta^i)} \frac{\pi(I)\pi(F)}{\pi(\eta)\pi(\omega_\eta(I,F))}\,.
    \end{align*}
    The probability $P(\eta,\eta^i)$ is at least the probability of picking the edge $e_{\sigma(i)}$ to update ($1/|E(\cT_h)|$) times the minimal probability of flipping an edge, which is some $C = C(p,q) > 0$. The ratio of probabilities is bounded by noticing that the number of edges present in the multisets $\{I,F\}$ and $\{\eta,\omega_\eta(I,F)\}$ are the same, leaving only the factor of $q$ to contribute. Without the boundary conditions on $\cT_h$, we claim that 
    \begin{align}\label{eq:component-change}
        |k(I) + k(F) - k(\eta) - k(\omega_\eta(I,F))| \le 2 |V(e_{\sigma(j):j\le i}) \cap V(e_{\sigma(j):j>i})|\,,
    \end{align}
    where we recall that $k(\omega)$ is the number of connected components in the subgraph $(V(G),\omega)$~\eqref{eq:fk-definition}. 
    To see~\eqref{eq:component-change}, note that the only component counts that can differ between $k(I)+k(F)$ and $k(\eta)+k(\omega_\eta(I,F))$ are from components that intersect the vertex boundary between $\{e_{\sigma(j)}:j\le i\}$ and $\{e_{\sigma(j)}:j>i\}$. 
    
    The addition of the boundary conditions can only change the bound on the left-hand side of~\eqref{eq:component-change} additively by at most $2$. This is because up to a change of the number of components by at most $1$, we can split the boundary condition $\xi$ into two parts, one being its part that intersects vertices of $\{e_{\sigma(j)}:j\le i\}$ and one being its part that intersects vertices of $\{e_{\sigma(j)}:j>i\}$. With this modification, the same reasoning as in the no boundary condition case holds, that the component counts only differ through components that hit the vertex boundary, and this number of components is evidently bounded above by the size of the vertex boundary. 
     Altogether, we get that the spectral gap satisfies 
    \begin{align*}
        \lambda_P^{-1} \le {C|E(\cT_h)|^2} \exp\big(2(\mathsf{CW}(\cT_h) + 1)\log q\big) \le {C} \Delta^{h} e^{ 2Kh\log q} \,,
    \end{align*}
    which implies the claimed bound up to a change of constants. 
\end{proof}

\section{Mixing time on locally treelike graphs}\label{sec:treelike-expanders}
We now use the understanding from Section~\ref{sec:trees} on the random-cluster model on trees with $r$-wired boundary conditions to control their mixing time on treelike graphs having exponentially strong supercritical phases for the edge-percolation: namely, to prove Theorem~\ref{thm:exp-growth-mixing}.  As mentioned in the introduction, the prototypical example to have in mind in this section is a treelike expander like the random $\Delta$-regular graph. 

The proof strategy of Theorem~\ref{thm:exp-growth-mixing} consists of the following three steps, which will organize the section: 
\begin{enumerate}
\item An $O(1)$ burn-in period to obtain (nearly) $r$-wired boundary conditions on the locally treelike $O(\log n)$-radius balls of the graph; 
\item Censoring of the dynamics after burn-in to localize to a treelike ball (and an estimate on how long the mixing time will be on the local ball with (nearly) $r$-wired boundary conditions);
\item A spatial mixing property with (nearly) $r$-wired boundary conditions to couple the two censored copies after they have each respectively reached equilibrium. 
\end{enumerate}

\subsection{Burn-in to induce $(r,L)$-wired boundary conditions on treelike balls}
In this subsection, we demonstrate that FK dynamics after an $O(1)$ burn-in period will be such that the boundary conditions it induces on $\frac{\eta}{2} \log n$ sized balls are (almost) $r$-wired. This will be essential to the application of the spatial mixing results of the previous section. Since the property of being $r$-wired is a monotone increasing property on the distribution over single-component boundary conditions, it will essentially suffice to establish it for the $\mbox{Ber}(\tilde p)$ edge percolation $\widetilde \omega$, and use Lemma~\ref{lem:domination-after-burn-in}. (In reality, there are some added complications by the possible $O(1)$ many extra wirings outside the single-component, as a bound on those extra wirings is no longer an increasing event.) 

In what follows, suppose $G$ is a graph of minimum degree $3$. Let $B_h= B_h(o)$ be the ball of radius $h$ about a fixed vertex $o$ and suppose it is $K$-treelike, meaning the removal of at most $K$ edges from $E(B_h)$ leaves a tree.   
By taking the breadth-first exploration of $B_h$, any vertex $w$ can be assigned a height via $h - d(w,o)$ and \emph{children} which are all vertices adjacent to $w$ having smaller height. Let the \emph{descendant graph} of $w$, denoted $D_w$ (in analogy with $\cT_{h,w}$), be the set of all descendants of $w$, together with their descendants, etc.

\begin{definition}
    A distribution $\sf{P}$ over boundary conditions on a $K$-treelike ball $B_h$ is called $(r,L)$-wired if it is generated as follows: 
    \begin{itemize}
        \item Some arbitrary $L$ vertices in $\partial B_h$ are chosen and an arbitrary wiring is placed on them;
        \item On the remainder, a subset  stochastically dominating the product $\mbox{Ber}(r)$ subset is wired together into one large component.
    \end{itemize}
\end{definition}

The main result of this subsection is that the boundary conditions induced on a treelike ball by the FK dynamics after an $O(1)$ burn-in time are $(\sf{p},L)$-wired. This will follow from the following.

\begin{lemma}\label{lem:burn-in-boundary-conditions} 
    Suppose $G$ has an exponentially strong supercritical phase per Definition~\ref{def:supercritical-Bernoulli-exp}, and for $h = \eta \log n$, the graph $G$ is $(K,h)$-treelike. 
    For every $r$, there exists $\tilde p(\eta,r,c_{\tilde p},\tilde p_0)$ such that if $\omega \succeq \widetilde \omega$, then for every $o\in V(G)$, the distrbituion over boundary conditions induced by $\omega(E(G)\setminus E(B_{h/2}))$ on $\partial B_{h/2}$ is within TV-distance $n^{-5}$ of a $(r,K)$-wired boundary condition. 
\end{lemma}

The complications for establishing the above lemma are that the class of $(r,L)$-wired boundary conditions are not a monotone family, and that the wirings of boundary vertices of $B_h$ are dictated by events of connectivity to a giant  which are not independent even under $\widetilde \omega$. We develop an auxiliary set of events on the random-cluster configuration which provide the necessary monotonicity and independence. For this, we note that all cycles in $B_h$ have either $1$ or $2$ vertices of minimal height; for a vertex $w\in \partial B_{h/2}$, we say the descendant graph $D_w$ in $B_h$ is a simple subtree if it does not contain any vertices of minimal height of a cycle of $B_h$.

\begin{itemize}
    \item Define $E_1$ as the event that every component of $\omega$ of size at least $\frac{\eta}{2} \log n$ coincides (i.e., $\omega$ has at most one component of size greater than $\frac{\eta}{2}\log n$). 
    \item For vertices $w\in \partial B_{h/2}$ whose descendant graph $D_w$ in $B_{h}$ is a simple subtree, let $E_w$ be the event that the configuration $\omega(D_w)$ has size greater than $\frac{\eta}{2}\log n$. (For other $w$, let $E_w$ be vacuous.) 
\end{itemize}
On the event $E_1$, the subset of $w\in \partial B_{h/2}$ such that $E_w$ holds will all be wired together through the giant, and the set of additional wirings of boundary vertices must be confined to those $w$ for which its descendant graph in $B_h$ is not a simple subtree, which is deterministically at most $2K$ since $G$ is $(K,h)$-treelike. It therefore suffices for us to establish that for all $\tilde p$ sufficiently large, if $\omega \succeq \widetilde \omega$, 
\begin{align*}
    \mathbb P(\omega \notin E_1) \le n^{-5} \qquad \mbox{and} \qquad \{w:\omega \in E_w\} \succeq \bigotimes_w \text{Ber}(\sf{p})\,,
\end{align*}
whence on the event $E_1$, the boundary conditions induced by $\omega(E(G)\setminus E(B_{h/2})$ on $\partial B_{h/2}$ would be $(r,K)$-wired. 
In particular, Lemma~\ref{lem:burn-in-boundary-conditions} is an immediate consequence of the following two lemmas. 

\begin{lemma}
    Suppose $G$ satisfies Definition~\ref{def:supercritical-Bernoulli-exp} with some $c_{\tilde p}, \tilde p_0$. For every $\eta>0$, there exists $\tilde p_0'$ such that if $\tilde p \ge \tilde p_0'$, and $\omega \succeq \widetilde \omega$, then 
    \begin{align*}
        \mathbb P(\omega \notin E_1)\le n^{-5}\,.
    \end{align*}
\end{lemma}
\begin{proof}
    In order for $\omega \notin E_1$, there must exist a component of size between $\frac{\eta}{2} \log_\Delta n$ and $n/2$ (two components both of size at least $n/2$ evidently coincide). Thus, by a union bound the probability of $\omega \notin E_1$ is bounded by the probability that there exists some connected set $A$ of size between $\frac{\eta}{2}\log_\Delta n$ and $n/2$ with $\omega(\partial_e A) \equiv 0$; this being a decreasing event, it suffices to upper bound its probability under $\widetilde \omega$, whence we can apply Definition~\ref{def:supercritical-Bernoulli-exp} with $\ell = \frac{\eta}{2}\log_\Delta n$ to get 
    \begin{align*}
        \widetilde \pi(\widetilde \omega \notin E_1) \le n e^{ - c_{\tilde p}\eta \log n}\,.
    \end{align*}
    Since $c_{\tilde p}$ goes to $\infty$ as $\tilde p \to 1$, as long as $\tilde p$ is larger than some $\tilde p_0(\eta)$ 
    the right-hand side will be at most $n^{-5}$.    
\end{proof}

\begin{lemma}
    Suppose $G$ is $(K,h)$-treelike For every $r<1$, there exists $\tilde p$ such that if $\omega \succeq \widetilde \omega$ then the distribution of $\{w: \omega \in E_w\}$ stochastically dominates a $\text{Ber}(r)$ subset of $\partial B_{h/2}$.
\end{lemma}
\begin{proof}
    Since the $E_w$ are increasing events, $\{w: \omega \in E_w\}$ stochastically dominates $\{w: \widetilde \omega \in E_w\}$, it suffices to establish the above for $\widetilde \omega$. Since the descendant graphs $D_w$ are disjoint for the $w\in \partial B_{h/2}$ for which $D_w$ are subtrees  (and hence $E_w$ is not vacuous), the events $E_w$ are independent under $\widetilde \omega$. It remains to argue that $\mathbb P(\widetilde \omega \in E_w)\ge r$ as long as $\tilde p$ is sufficiently large. Since the descendant graphs $D_w$ contain full binary trees as subgraphs, the probability of $\widetilde \omega \in E_w$ is larger than the probability that the branching process with offspring distribution $\text{Bin}(2,\tilde p)$ survives $\frac{\eta}{2}\log n$ generations, which happens with probability going to $1$ as $\tilde p\to 1$: see e.g.,~Fact~\ref{fact:branching-process}. In particular, for every $r$, the probability of $\widetilde \omega\in E_w$ is greater than $r$ as long as $\tilde p$ is large enough (depending only on $r$). 
\end{proof}

\subsection{Spatial mixing on treelike balls with $(r,L)$-wired boundary}\label{subsec:spatial-mixing-treelike-balls}

We now describe how minor adjustments to Section~\ref{sec:tree-spatial-mixing} lead to spatial mixing on treelike graphs with boundary conditions that are $(r,L)$-wired. 

\begin{lemma}\label{lem:treelike-spatial-mixing}
	Suppose $B_h$ is $K$-treelike, and suppose $\mathsf{P}$ is $(r,L)$-wired. Then except with $\sf P$-probability $e^{ - c_{r} (1.1)^{h/(2K)}}$ (with $c_{r}>0$ for $r$ large), $\xi$ is such that 
    \begin{align*}
        \|\pi_{B_h}^\xi(\omega(B_{h/2})\in \cdot) - \pi_{B_h}^\one(\omega(B_{h/2})\in \cdot)\|_\tv \le Ce^{ - c_p (h/(2K) - K - L)}\,.
	\end{align*}
 for some constant $c_p$ going to $\infty$ as $p\to 1$. 
\end{lemma}

\begin{proof}
    Since $B_h$ is $K$-treelike, there is an edge-set $H$ of size at most $K$ such that $B_h \setminus H$ is a tree. Following the breadth-first search of $B_h$, there is a stretch of at least $h/(2K)$ consecutive levels between depth $h/2$ and $h$ such that the restriction of $B_h$ to those levels is a forest and every tree in that forest contains a full binary tree as a subgraph (using the minimum degree $3$ condition). There exists $m\ge h/2+ h/(2K)$ such that this forest is $B_{m}\setminus B_{m-h/(2K)}$. 
    
    Let $\tilde S_{h,\xi}$ be the event that there is a wired separating surface in every one of the constituent trees of height $h/(2K)$ in the configuration $\omega(B_h\setminus H)$. (Notice that since the boundary conditions are at depth $h$, this depends on the full configuration, not just the restriction to the stretch of $h/(2K)$ heights; moreover by definition of wired separating surface, if $\omega(B_h\setminus H)$ has a wired separating surface then so does $\omega(B_h)$.)  
    
    The first claim is that under this definition, there is a minor modification of the revealing procedure of Lemma~\ref{coupling-with-separating-set} such that the probability of $\omega_\xi \in \tilde S_{h,\xi}^c$ upper bounds the total-variation distance. This is done by first revealing the entire configurations $\omega_\xi$ and $\omega_\one$ on $(B_h\setminus H)\setminus B_{m+h/(2K)}$ under the monotone coupling. In this manner, the revealed part of $\omega_\xi$ induces some single-component boundary conditions $\tilde \xi$ on $B_{m+h/(2K)}$. We can then apply Lemma~\ref{coupling-with-separating-set} to each of the constituent trees of $B_{m}\setminus B_{m-h/(2K)}$ and it follows that on the event that they all have wired separating surfaces in $\omega_\xi$, then $\omega_\xi$ is coupled to $\omega_\one$ above those wired separating surfaces and in particular on all of $B_m$. As such, we have the analogous
    \begin{align*}
       \|\pi_{B_h}^\xi(\omega(B_{h/2})\in \cdot) - \pi_{B_h}^\one(\omega(B_{h/2})\in \cdot)\|_\tv \le \pi_{B_h^\xi}(\omega(B_{h}\setminus H)\in \tilde S_{h,\xi}^c)\,.
    \end{align*}

    To control the probability of $\tilde S_{h,\xi}^c$, we follow the reasoning of Lemma~\ref{lem:xi-not-c-good}, with the modifications being minimal. The only difference that arises is that when considering the probability of the event in item (2) in that proof, the sub-tree from a vertex $v_i\in B_{m+h/(2K)}\setminus B_m$ in $B_h\setminus H$ no longer necessarily contains a full binary tree: for up to $K$ many vertices, it could be the pruning of a binary tree with up to $K$ many subtrees of depth at least $h/2K$ deleted from it. The easiest thing to do is to simply disregard these vertices which leads to a change from $h$ to $h-K$ in the concentration quality. Similarly, at most $L$ of the subtrees contain at their boundary one of the $L$ vertices where the $(r,L)$ boundary conditions are arbitrarily rewired, and we can disregard these $L$ vertices as well. Nothing else will be affected in the proof. 
\end{proof}

\subsection{Local mixing time on treelike balls with $(r,L)$-wired boundary}\label{subsec:mixing-time-treelike-balls}
We now show that the $\eta \log n$-radius balls in the treelike expander, with the boundary conditions induced on them by the remainder of the FK dynamics configuration, have a polynomial mixing time. By Lemma~\ref{lem:burn-in-boundary-conditions}, after a burn-in, these will look like treelike balls with boundary conditions that have an $O(1)$ number of distinct components, one of them being macroscopic, and the rest all being $O(1)$ sized. At this point we can appeal to comparison estimates for Markov chains together with the bound on the inverse spectral gap on trees with single-component mixing times from Lemma~\ref{lem:wired-tree-spectral-gap}. 

\begin{remark}
We do not use any further information on the boundary conditions (like the $r$-wired property or the randomness), only the fact that there is at most $1$ large component and $O(1)$ many vertices in the union of all other components. 
In theory, it is likely that we could use the $r$-wired property to get a significantly better bound on the mixing time of a tree. However, the extra $O(1)$ wirings and the $O(1)$ many non-tree edges force us to at some point perform a comparison through a spectral gap, and getting back a mixing time bound from this will anyways end up costing a factor of the volume per~\eqref{eq:tmix-trel-comparison}.   
\end{remark}

 To perform comparisons, we formalize a notion of distance between boundary conditions. Two ``similar'' random-cluster boundary conditions (in terms of the wiring they induce)
have similar effects on the underlying random-cluster distribution and on the behavior of the corresponding FK dynamics.
In turn, the Dirichlet form, and spectral gaps of their corresponding dynamics should be ``close" to one another. 
We compile a few definitions and results that formalize this idea.

\begin{definition}
	 For two boundary conditions $\phi \leq \phi'$, define $D(\phi,\phi') := k(\phi) - k(\phi')$ where $k(\phi)$ is the number of components in $\phi$. For two partitions $\phi, \phi'$ that are not comparable, let $\phi''$ be the smallest partition such that $\phi'' \geq \phi$ and $\phi'' \geq \phi'$ and set $D(\phi,\phi') = k(\phi) - k(\phi'')+ k(\phi')- k(\phi'')$. 
\end{definition}

The following lemma is then straightforward from the definition of the random-cluster measure~\eqref{eq:fk-definition}. 

\begin{lemma}[E.g., Lemma 2.2 from~\cite{BGVfull}]
	\label{lemma:simple-rc-bound}
	Let $G$ be arbitrary, $p \in (0,1)$ and $q > 0$. Let $\phi$ and $\phi'$ be any two partitions of $V(G)$. Then, for all random-cluster configurations $\omega\in \{0,1\}^E$, we have
	$$
	q^{-2D(\phi,\phi')} {\pi_{G}^{\phi'}(\omega)} \le \pi_{G}^{\phi}(\omega) \le q ^{2D(\phi,\phi')} \pi_{G}^{\phi'}(\omega)\,.
	$$  
\end{lemma}

The following corollary is a standard comparison of spectral gaps, and follows from Lemma~\ref{lemma:simple-rc-bound}, the definition of the transition matrix of the FK dynamics, and Theorem 4.1.1 in~\cite{SClecture-notes}.

\begin{cor}
    \label{cor:simple-spectral-gap-bound}
	Let $G=(V,E)$ be an arbitrary graph, $p \in (0,1)$ and $q > 0$.
	Consider the FK dynamics on $G$ with boundary conditions $\phi$ and $\phi'$,
	and let $\lambda$, $\lambda'$ denote their respective spectral gaps.
	Then, 
	$$
	q^{-5 D(\phi,\phi')} \lambda' \le \lambda \leq   q^{ 5 D(\phi,\phi')} \lambda'\,.
	$$
\end{cor}

Using the above, we are able to deduce the following bound. 

\begin{lemma}\label{lem:treelike-mixing-time}
	Consider a $K$-treelike ball $B_h$ with boundary conditions $\xi$ that have one component of arbitrary size together with at most $L$ many additional boundary wirings. There exists $a(\Delta,q)$ such that the inverse spectral gap on $B_h^\xi$ is at most $C_{p,q} \exp(a(h+K+L))$. 
\end{lemma}

\begin{proof}
	Consider the modification of $B_h^\xi$ where all endpoints of the set $H$ are wired up to one another via a boundary condition, and denote it by $\widetilde B_h^\xi$. By Corollary~\ref{cor:simple-spectral-gap-bound}, their spectral gaps are within a factor of $q^{ 10|H|}$ of one another. 
     The FK dynamics on $\widetilde B_h^\xi$ are a product of the FK dynamics on $B_h\setminus H$ with boundary conditions $\xi$ and the wirings of the edges of $H$ (call this $(\widetilde{B_h\setminus H})^\xi$), and $|H|$ independent FK dynamics chains on single edges with wired boundary. By tensorization of the spectral gap (see e.g.,~\cite{SClecture-notes}) the spectral gap of the FK dynamics on $B_h$ is then the minimum of the gap on those individual edges, and the gap of the FK dynamics on $(\widetilde{B_h\setminus H})^\xi$. 
	
    The spectral gaps of the individual edges are clearly some constant depending on $p$, so it suffices to bound the spectral gap of FK dynamics on $(\widetilde{B_h\setminus H})^\xi$. For this purpose, notice that with a cost of $q^{5(|H|+ L)}$ we can perform a further boundary modification and remove the wirings of the edges in $H$ as well as those $L$ additional wirings in $\xi$ to end up with the tree $B_h \setminus H$ with a single-component boundary condition.  At this point, we can bound the spectral gap of this resulting single-component tree of depth $h$ using Lemma~\ref{lem:wired-tree-spectral-gap}. Putting together the costs from the various comparisons we obtain the desired. 
\end{proof}

\subsection{Mixing times on locally treelike graphs}\label{subsec:proof-locally-treelike}
In this section, we combine the above parts to establish the near-linear mixing time bound of Theorem~\ref{thm:exp-growth-mixing} for FK dynamics on locally treelike graphs, as long as they have an exponentially strong supercritical phase for edge-percolation.  

\begin{proof}[\textbf{\emph{Proof of Theorem~\ref{thm:exp-growth-mixing}}}]
    Recall that we use $X_t^{\omega_0}$ to denote the FK dynamics at time $t$ from initialization $\omega_0$.
    By monotonicity, under the grand coupling, we have 
    \begin{align}\label{eq:disagreement-probability-monotonicity}
        \max_{\omega,\omega'} \mathbb P(X_t^\omega \ne X_t^{\omega'}) & \le \sum_{e\in E(G)} \mathbb P(X_t^\omega(e)\ne X_t^{\omega'}(e)) \nonumber \\
        & \le \sum_{e\in E(G)} \mathbb P(X_t^\zero(e)\ne X_t^{\one}(e)) =  \sum_{e\in E(G)} \mathbb E[X_t^\one(e)] -\mathbb E[X_t^\zero(e)]\,.
    \end{align}
    Fix an edge $e\in E(G)$ and consider the difference in expectations on the right. 
     Set $\bar X_t^\zero$ to be the censored FK dynamics that agrees with $X_t^\zero$ for all times until some $T_0$ but that then censors (ignores) all updates after time $T_0$ outside of $B_{h} = B_h(o)$ for some $o\in e$ and for $h= \frac{\eta}{2}\log n$. Let $\bar X_t^\one$ be the Markov chain that censors \emph{all} updates of $X_t^\one$ outside of $B_h$ (regardless of $t$). 
    Then, by the censoring inequality of~\cite{PWcensoring}, 
    \begin{align}\label{eq:censoring}
        \mathbb E[X_t^\one (e) ] - \mathbb E[X_t^\zero(e)]\le \mathbb E[\bar X_t^\one(e)] - \mathbb E[ \bar X_t^\zero(e)]\,.
    \end{align}
    By Lemma~\ref{lem:domination-after-burn-in} and Lemma~\ref{lem:burn-in-boundary-conditions}, for every $r$, there exist $T_0(r,q)$ and $p_0$ large enough that for all $p\ge p_0$, the FK dynamics $X_{T_0}^\zero, X_{T_0}^\one$ induce $(r,K)$-wired boundary conditions on $\partial B_{h/2}$.  
    
    Let $A_\gamma = A_\gamma(B_h)$ be the set of boundary conditions $\xi$ on $B_h$ that are single-component together with at most $K$ additional wirings, and furthermore that are such that the inequality of Lemma~\ref{lem:treelike-spatial-mixing} holds with constant $c_{p} = \gamma$. Since $X_{T_0}^\zero, X_{T_0}^\one$ induce $(r,K)$-wired boundary conditions on $\partial B_{h/2}$, by Lemma~\ref{lem:treelike-spatial-mixing}, 
    \begin{align}\label{eq:burnt-in-bc-good}
        \mathbb P(X_{T_0}^\zero(B_{h/2}^c)\notin A_\gamma^c) \le e^{ - c_{r}(1.1)^{h/2K}}\,.
    \end{align}
    At the same time, by definition of $A_\gamma$, we have
    \begin{align}\label{eq:local-spatial-mixing}
        \max_{\xi \in A_\gamma} \big(\pi_{B_h}^\one(\omega_e) - \pi_{B_h}^{\xi}(\omega_e)\big) \le Ce^{ - \gamma h/(3K)}\,.
    \end{align}
    so long as $n$ is large enough that $h/(2K)- K - K \ge h/(3K)$. 
    Since the event $\{\bar X_{T_0}^\zero(B_h^c) \in A_\gamma\}= \{X_{T_0}^\zero(B_h^c) \in A_\gamma\}$ is measurable w.r.t.\ $\mathcal F_{T_0}$ (the filtration generated by the grand coupling up to time $T_0$), 
    \begin{align*}
        \mathbb E[\bar X_{T_0+S}^\one(e)] - \mathbb E[\bar X_{T_0+S}^\zero(e)] & \le \mathbb P(\bar X_{T_0}^\zero(B_h^c) \notin A_\gamma) + \mathbb P(\bar X_{T_0 + S}^\one (e) \ne \bar X_{T_0 + S}^\zero(e) \mid \bar X_{T_0}^\zero(B_h^c) \in A_\gamma) \\ 
        & \le \mathbb P(X_{T_0}^\zero(B_h^c) \notin A_\gamma) + \max_{\xi\in A_\gamma(B_h)} (\mathbb E[ Y_{S,B_h^\one}^\one(e)] - \mathbb E[Z_{S,B_h^\xi}^\zero(e)])\,,
    \end{align*}
    where $Y_{s,B_h^\one}^\one$ and $Z_{s,B_h^\xi}^\zero$ are Glauber chains on $B_h$ with boundary conditions $\one$ and $\xi$ respectively, initialized from $\one$ and $\zero$ respectively. We have used here the definition of censored dynamics and monotonicity. The first term is at most $e^{ - c_{r}(1.1)^{h/2K}}$ by~\eqref{eq:burnt-in-bc-good}. 
    
    For the second term, fix any $\xi \in A_\gamma(B_h)$, let $Y_s = Y_{s,B_h^\one}^\one$ and $Z_s = Z_{s,B_h^\xi}^\zero$, and write 
    \begin{align*}
        \mathbb E[ Y_{S}(e)] - \mathbb E[Z_{S}(e)] = \big(\mathbb E[Y_S^\one(e)] - \pi_{B_h^\one}(\omega_e)\big) + \big( \pi_{B_h^\one}(\omega_e) - \pi_{B_h^\xi}(\omega_e)\big) + \big(\pi_{B_h^\xi}(\omega_e) - \mathbb E[Z_S^\zero(e)] \big)\,.
    \end{align*}
    The middle term is at most $Ce^{- \gamma h/(3K)}$ by~\eqref{eq:local-spatial-mixing}. For the first and third, suppose  
    \begin{align*}
        S \ge C_0 \log n \cdot \max_{\xi\in A_\gamma(B_h)} \log\Big(\frac{1}{\min_{\omega} \pi_{B_h}^\xi(\omega)}\Big)\cdot \lambda_{B_h^\xi}^{-1}\,,
    \end{align*}
    where $\lambda_{B_h^\xi}^{-1}$ is the inverse spectral gap of the FK dynamics on $B_h$ with boundary conditions $\xi$. Then by~\eqref{eq:tmix-trel-comparison} and sub-multiplicativity of TV-distance to stationarity, for a universal constant $C_0$, both the first and third terms will be at most $n^{-5}$.
    Combining, we get  
    \begin{align}\label{eq:censored-dynamics-bound}
        \mathbb E[\bar X_{T_0+S}^\one(e)] - \mathbb E[\bar X_{T_0+S}^\zero(e)] \le C(e^{ - c_{r} (1.1)^{h/(2K)}} + e^{ - \gamma h/(3K)} + n^{-5})  \,.
    \end{align}
    At this point we make the following choices for $\eta,\gamma,T_0,S$:
    \begin{enumerate}
     \item $h = \frac{\eta}{2} \log n$ for $\eta>0$ sufficiently small that for every $\xi\in A_\gamma(B_h)$, 
     $$\log\Big(\frac{1}{\min_{\omega} \pi_{B_h}^\xi(\omega)}\Big)\cdot \lambda_{B_h^\xi}^{-1}\le \Delta^{\frac{\eta}{2}\log n} \lambda_{B_h^\xi}^{-1} \log(1-p)\quad \text{is at most}\quad n^{\epsilon/2}\log(1-p) \,;$$
     \item $S= C_1 n^{\epsilon/2} \log n$ where $C_1 = C_0 \log(1-p)$; 
     \item $\gamma$ large enough that $e^{- \gamma h/(2K)}$ is at most $n^{-5}$;
     \item $r$ large enough that $c_{r}>0$ 
     \item $T_0$ and $p$ large enough that $X_{T_0}^\zero$ induces $(r,K)$-wired boundary on $\partial B_{h/2}$.
     \end{enumerate}
     The existence of such an $\eta(\Delta,q,K)$ follows from Lemma~\ref{lem:treelike-mixing-time} and the fact that $a$ only depends on $\Delta,q$. Furthermore, by taking $r$ large, we can make $\gamma$ and $c_{r}$ arbitrarily large, to satisfy items (3)--(4). Finally, Lemma~\ref{lem:burn-in-boundary-conditions} ensures we can take $T_0,p_0$ large enough that for all $p\ge p_0$, item (5) is satisfied. 

     With these choices, for $n$ sufficiently large, the right-hand side of~\eqref{eq:censored-dynamics-bound} is at most $n^{-4}$. 
     Combining these with~\eqref{eq:disagreement-probability-monotonicity}--\eqref{eq:censoring}, we get that there exists $p_0(c_{\tilde p},\tilde p_0,\Delta,\eta,q,K)$ such that for all $p\ge p_0$, 
     \begin{align*}
        \max_{\omega}\|\mathbb P(X_{T_0 + S}^\omega\in \cdot) - \pi\|_\tv \le \max_{\omega,\omega'}\mathbb P(X_{T_0 + S}^\omega \ne X_{T_0 + S}^{\omega'})  = o(1)\,,
     \end{align*}
     which concludes the proof since evidently $T_0 + S = O(n^{\epsilon/2}\log n) = O(n^{\epsilon})$.  
\end{proof}

\section{Graphs with slow mixing at arbitrarily low temperatures}\label{sec:slow-mixing-graphs}

In this section, we establish the following slow mixing result for the FK dynamics. Theorem~\ref{thm:slow-mixing} is the special case where we have restricted to integer $q$ and used the comparison results of~\cite{Ullrich-random-cluster}. 

\begin{theorem}\label{thm:slow-mixing-FK}
    Fix $\Delta \ge 3$ and $p_0<1$. 
    \begin{enumerate}
        \item For any $q\ge3$, there exists $p>p_0$ and a sequence of locally treelike graphs $(G_n)_n$ of maximum degree $\Delta$ such that the FK dynamics on $G_n$ have $\tmix\ge \exp(\Omega(n))$. 
        \item For any $q>4$ (possibly non-integer), there exists $p>p_0$ and a sequence $(G_n)_n$ of polynomial volume growth and maximum degree $\Delta$ such that the FK dynamics of $G_n$ have $\tmix \ge \exp(\Omega(\sqrt{n}))$. 
    \end{enumerate}
\end{theorem}

The key tool for this proof will be the so-called ``series law'' for the random-cluster model. That is, splitting any edge $e$ into two edges, 
and changing the edge probability parameter for $e$ from $p_e$ to approximately $\sqrt{p_e}$ for each of the two new resulting edges preserves the random-cluster measure when one re-identifies the edges and takes the status of $e$ as being open if the two new edges are open. This gives a mechanism for boosting any fixed $p$ into a $p'$ which gets closer and closer to $1$, while only increasing the number of edges and vertices in the graph by a constant factor. In this manner, if a graph $H$ has slow mixing for its FK dynamics at some value of $p$ (no matter how small), then its modification $G$ (obtained by multiple applications of the series law) can be made to have slow mixing at $p'$, for $p'$ that can be arbitrarily close to $1$. Moreover, the graph modifications do not distort the maximum degree and volume growth (though they importantly do distort the isoperimetric dimension, and expansion rates, which we recall were fundamental to the presence of a strongly supercritical phase for the edge-percolation on the graph). 

\begin{remark}
    Both items (1)--(2) of Theorem~\ref{thm:slow-mixing-FK} should hold for all $q>2$.  The gap for $q$ non-integer in item (1) and $q\in (2,4]$ in item (2)  come from the present lack of proof (to our knowledge) of a slowdown for FK dynamics on bounded degree graphs (satisfying the corresponding graph condition) at those values of $q$. Such a slowdown is widely expected at the critical points both for the random regular graph and on $(\mathbb Z/m\mathbb Z)^d$ for large $d$. 
    The values of $q$ for which we can establish our lower bound come from the slowdowns of the random regular graph at integer $q\ge 3$~\cite{PottsRGMetastabilityCMP}, and the torus $(\mathbb Z/\sqrt{n}\mathbb Z)^2$ at $q>4$~\cite{GL1}.  
\end{remark}

Let us precisely recall the series law of the random-cluster model from~\cite[Theorem 3.89]{Grimmett}. 
\begin{lemma}\label{lem:RC-series-rule}
Two edges $e,f$ of a graph $G=(V,E)$ are in \emph{series} if $e= \{u,v\}$ and $f=\{v,w\}$ and $v$ has no other incident edges. Let 
\begin{align}\label{eq:rc-series-law}
    \sigma(x,y,q) = \frac{xy}{1+(q-1)(1-x)(1-y)}\,.
\end{align}
Let $G' = (V\setminus \{v\}, (E\setminus \{e,f\})\cup \{u,w\})$. 
For a random-cluster configuration $\omega$ on $G$, define $\omega'$ on $G'$ by
setting $\omega'(a) = \omega(a)$ for $a \in E \setminus \{e,f\}$ and $\omega'(\{u,w\})= \omega(e) \cdot \omega(f)$. Then, if $\omega$ is sampled from the random-cluster distribution on $G$ with parameters $(p_a)_{a \in E}$ and $q > 0$, $\omega'$ is distributed according to the random-cluster distribution on $G'$ with parameters $(p_a)_{a \in E\setminus \{e,f\}}$ and $p_{\{u,w\}} = \sigma(p_e,p_f,q)$.  
\end{lemma}

\begin{proof}[\textbf{\emph{Proof of Theorem~\ref{thm:slow-mixing-FK}}}]
For item (1), let $(H_n)_n$ be a locally treelike sequence of $n$-vertex graphs of degree at most $\Delta$ such that the mixing time of FK dynamics at some fixed value of $p\in (0,1)$ is $\exp(\Omega(n))$. From Theorem 1.2 of~\cite{PottsRGMetastabilityCMP}, we know that a randomly drawn sequence of $\Delta$-regular graphs satisfies this bound with high probability if $q\ge 3$ is an integer and $p = p_c(q,\Delta)$. For item (2), let $(H_n)_n$ be the torii $(\mathbb Z/\sqrt{n}\mathbb Z)^2$. We then know from Theorem~2 of~\cite{GL1}  that at fixed $p= p_c(q)\in (0,1)$ the mixing time of FK dynamics on $H_n$ is $\exp(\Omega(\sqrt{n}))$ for all real $q>4$.

Let $G_n$ be the modification of $H_n$ in which every edge of $H_n$ is split into $2^K$ edges in series, for $K$ determined as follows. 
Let $\zeta_q(p)$ be the inverse of $\sigma(x,x,q)$ from~\eqref{eq:rc-series-law}, i.e., 
$$\sigma(\zeta_q(p),\zeta_q(p),q)=p.$$
Such an inverse exists and is increasing for $p\in (0,1)$ by virtue of the fact that $\sigma(x,x,q)$ is continuously increasing from $0$ to $1$ as $x$ ranges from $[0,1]$. Since $x^2/q \le \sigma(x,x,q) \le x^2$
it must be the case that $\zeta_q(p)^2/q \le \sigma(\zeta_q(p),\zeta_q(p),q) \le \zeta_q(p)^2$. 
In order for this to be equal to $p$, it must be that $\sqrt{p}\le \zeta_q(p)\le \sqrt{pq}$.

Take $K$ such that 
$$
     \zeta_q^{\circ K}(p):= \underbrace{\zeta_q(\,\zeta_q(\,\cdots \,\zeta_q}_{K}(p)\cdots))\ge p_0\,.
$$
The inequality $\sqrt{p} \le \zeta_q(p)$ ensures that for any
$$
    K \ge \log_2\Big(\frac{\log p}{\log p_0}\Big)\,,
$$
we have $\zeta_q^{\circ K}(p) \ge p_0$. By Lemma~\ref{lem:RC-series-rule}, if $\omega'$ is a random-cluster configuration on $G_n$ with parameters $p'= \zeta_q^{\circ K}(p)$ and $q$, and  $\omega_e = \prod_{e_i} \omega_{e_i}'$ for every edge $e\in E(H_n)$, where $(e_i)_{i=1}^{2^K}$ are the edges in $G_n$ derived from splitting $e$, then $\omega$ is a sampled from the random-cluster measure on $H_n$ with parameters $p$ and $q$. 

We now claim that the FK dynamics on $G_n$ with parameters $p'$ and $q$ have $\exp(\Omega(n))$ mixing if $H_n$ has $\exp(\Omega(n))$ mixing at parameters $(p,q)$ (the reasoning that it has $\exp(\Omega(\sqrt{n}))$ mixing if $H_n$ has $\exp(\Omega(\sqrt{n})$ mixing is identical, so we omit it). This is achieved by lifting a bottleneck set from $H_n$ to $G_n$. In what follows, it is convenient to work with the discrete-time Glauber dynamics (though the same proof would work in the continuous-time setting as well). Let $P_H$ denote the transition matrix of the discrete-time FK dynamics on $H_n$ at parameter $p$ and $P_G$ the transition matrix of the discrete-time FK dynamics on $G_n$ at parameter $p'$. Similarly, $\pi_{H_n}$ is at parameter $p$ while $\pi_{G_n}$ is at parameter $p'$.  

By assumption, the mixing time of FK dynamics on $H_n$ is exponential in $n$. As such, there exists a subset of configurations $A \subset\{0,1\}^{E(H_n)}$ of exponentially small conductance. 
More precisely, 
there must exist $A \subset\{0,1\}^{E(H_n)}$ such that $\pi_{H_n}(A) \le 1/2$ and 
\begin{equation}
    \label{eq:conductance}
    \Phi_H(A) := \frac{\sum_{x \in A, y\in A^c} \pi_{H_n}(x) P_{H}(x,y)}{\pi_{H_n}(A)} = \exp(-\Omega(n))\,;
\end{equation}
see, e.g.,~\cite[Theorem 13.10]{LP}. Let $\partial_{P_H} A = \{x \in A: P_{H}(x,y)>0~\text{for some}~y\in A^c\}$.
Since for each $x \in \partial_{P_H} A$, 
% 
% $|\{y \in A^c: P_{H_n}(x,y)>0\}| = \Theta(E(H_n))$,
each entry of $P_{H_n}$ is at least $\frac{1}{n} \ps$, 
% and $\pi_H(A^c) > 1/2$, 
we see that there exists a constant $c > 0$ such that
\begin{align}\label{eq:H-bottleneck}
    \frac{\pi_{H_n}(\partial_{P_{H}} A)}{\pi_{H_n}(A)} = \frac{\sum_{x\in \partial_{P_H} A} \pi_{H_n} (x)}{\pi_{H_n}(A)} \le  n \ps^{-1} \Phi_H(A) \le \exp( - c n)\,.
\end{align}
Let $T_A$ be the subset of $\{0,1\}^{E(G_n)}$ defined by $T_A = \{\omega':  \omega \in A\}$ where the relationship between $\omega$ and $\omega'$ is defined per the operation described above. By Lemma~\ref{lem:RC-series-rule}, $\pi_{G_n}(T_A) = \pi_{H_n}(A)$. 
Every configuration in $\partial T_A$ must be in $T_{\partial A}$ because if $P_{G}(\omega',\sigma')>0$, for $\sigma'\notin T_A$, then $\sigma'$ projects down to a configuration in $A^c$, so $\omega'$ must project into $\partial A$. Therefore, by Lemma~\ref{lem:RC-series-rule}, $\pi_{G_n}(\partial_{P_{G}} T_A)\le \pi_{H_n}(\partial_{P_{H}} A)$. Altogether, it follows from~\eqref{eq:H-bottleneck} that 
\begin{align*}
    \frac{\pi_{G_n}(\partial_{P_{G}} T_A)}{\pi_{G_n}(T_A)} \le \exp( - cn)\,.
\end{align*}
Using the facts that for every $x$, the number of $y\in {T_A}^c$ for which $P_G(x,y)$ is positive is at most $E(G_n)$,  
\begin{align*}
    \frac{\sum_{x\in T_A,y\in T_A^c} \pi_{G_n}(x) P_G(x,y)}{\pi_{G_n}(T_A)} \le |E(G_n)|\frac{\pi_{G_n}(\partial_{P_{G}} T_A)}{\pi_{G_n}(T_A)}\,,
\end{align*}
implies that $T_A$ is a set of exponentially small conductance for the FK dynamics on $G_n$. 
This then implies the inverse gap and mixing time of FK dynamics on $G_n$ are both exponential in $n$. 

It remains to reason that the number of vertices and edges of $G_n$ are of the same order as  the number of vertices and edges of $H_n$, so that the resulting bounds are indeed exponential in $|V(G_n)|$. 
Notice that as long as $p  = \Omega(1)$ and $1-p_0 =  \Omega(1)$, then $K = O(1)$. As such, $G_n$ will have $|V(G_n)| \le |V(H_n)| + 2^K |E(H_n)| =  O(|V(H_n)|)$, and $|E(G_h)| \le 2^K |E(H_n)|$. This yields the claimed bound. 
\end{proof}

 \subsection*{Acknowledgements} 
The authors thank the anonymous referees for their careful reading and comments.
The research of AB was supported in part by NSF grant CCF-2143762. The research of R.G.\ was supported in part by NSF DMS-2246780.

\bibliographystyle{alpha}
\bibliography{references}
\end{document}